\documentclass[11pt]{amsart}

\usepackage{amsmath}
\usepackage{amsfonts}
\usepackage{amssymb}
\usepackage{amsthm}
\usepackage{graphicx}
\usepackage{pb-diagram}
\usepackage{epstopdf}
\usepackage{pdfpages}
\usepackage{bbm}
\usepackage{multirow}
\usepackage[mathscr]{eucal}
\usepackage{epsfig,epsf,epic}
\usepackage{pdfsync}
\usepackage{enumerate}
\usepackage{etex}
\usepackage[all]{xy}
\usepackage{fancyref}
\usepackage{comment}
\usepackage{todonotes}
\usepackage{hyperref}

\newtheorem{theorem}{Theorem}[section]
\newtheorem*{theorem*}{Theorem}
\newtheorem{prop}[theorem]{Proposition}
\newtheorem*{prop*}{Proposition}

\newtheorem{definition}[theorem]{Definition}
\newtheorem{example}[theorem]{Example}
\newtheorem{lemma}[theorem]{Lemma}
\newtheorem{assum}[theorem]{Assumption}
\theoremstyle{definition}
\newtheorem{notation}[theorem]{Notations}
\newtheorem{remark}[theorem]{Remark}
\numberwithin{equation}{section}

\newtheorem*{remark*}{Remark}

\DeclareMathOperator{\Hom}{Hom}
\DeclareMathOperator{\Ker}{Ker}

\DeclareMathOperator{\Spec}{Spec}

\newcommand{\real}{\mathbb{R}}
\newcommand{\comp}{\mathbb{C}}

\newcommand{\inte}{\mathbb{Z}}

\newcommand{\dd}[1]{\frac{\partial}{\partial #1}}

\newcommand{\half}{\frac{1}{2}}
\newcommand{\ddd}[2]{\frac{\partial#1}{\partial{#2}}}

\newcommand{\pdpd}[3]{\frac{\partial^2 #1}{\partial #2\partial #3}}

\newcommand{\facs}{\varphi}

\newcommand{\incoming}{\Pi}

\newcommand{\rtree}[1]{\mathtt{R}\mathtt{T}_{#1}}
\newcommand{\tr}{T}
\newcommand{\rtr}{\mathbf{T}}
\newcommand{\wtr}{\Gamma}
\newcommand{\wtree}[1]{\mathtt{W}\mathtt{T}_{#1}}
\newcommand{\wrtr}{\mathcal{T}}
\newcommand{\wrtree}[1]{\mathtt{W}\mathtt{R}_{#1}}

\newcommand{\ltr}{L}
\newcommand{\ltree}[1]{\mathtt{L}\mathtt{T}_{#1}}
\newcommand{\lrtr}{\mathcal{L}}
\newcommand{\lrtree}[1]{\mathtt{L}\mathtt{R}_{#1}}

\newcommand{\mtr}{J}
\newcommand{\mtree}[1]{\mathtt{M}\mathtt{T}_{#1}}
\newcommand{\mrtr}{\mathcal{J}}
\newcommand{\mrtree}[1]{\mathtt{M}\mathtt{R}_{#1}}

\newcommand{\hp}{\hslash}  

\newcommand{\sem}{p}

\newcommand{\Supp}{\text{\textnormal{Supp}}}
\newcommand{\Joints}{\text{\textnormal{Joints}}}
\newcommand{\Id}{\text{\textnormal{Id}}}

\newcommand{\sgn}{\text{\textnormal{sgn}}}
\newcommand{\ad}{\text{\textnormal{ad}}}
\newcommand{\Aut}{\text{\textnormal{Aut}}}
\newcommand{\Iso}{\text{\textnormal{Iso}}}
\newcommand{\Int}{\text{\textnormal{Int}}}

\makeatletter
\providecommand\@dotsep{5}
\renewcommand{\listoftodos}[1][\@todonotes@todolistname]{%
  \@starttoc{tdo}{#1}}
\makeatother

\begin{document}

\title[Refined scattering diagrams, theta functions from MC equations]{Refined scattering diagrams and theta functions from asymptotic analysis of Maurer--Cartan equations}

\author[Leung]{Naichung Conan Leung}
\address{The Institute of Mathematical Sciences and Department of Mathematics\\ The Chinese University of Hong Kong\\ Shatin \\ Hong Kong}
\email{leung@math.cuhk.edu.hk}

\author[Ma]{Ziming Nikolas Ma}
\address{The Institute of Mathematical Sciences and Department of Mathematics\\ The Chinese University of Hong Kong\\ Shatin \\ Hong Kong}
\email{nikolasming@outlook.com}

\author[Young]{Matthew B. Young}
\address{Max Planck Institute for Mathematics\\
Vivatsgasse 7\\
53111 Bonn, Germany}
\email{myoung@mpim-bonn.mpg.de}

\begin{abstract}
We further develop the asymptotic analytic approach to the study of scattering diagrams. We do so by analyzing the asymptotic behavior of Maurer--Cartan elements of a differential graded Lie algebra constructed from a (not-necessarily tropical) monoid-graded Lie algebra. In this framework, we give alternative differential geometric proofs of the consistent completion of scattering diagrams, originally proved by Kontsevich--Soibelman, Gross--Siebert and Bridgeland. We also give a geometric interpretation of theta functions and their wall-crossing. In the tropical setting, we interpret Maurer--Cartan elements, and therefore consistent scattering diagrams, in terms of the refined counting of tropical disks. We also describe theta functions, in both their tropical and Hall algebraic settings, in terms of flat sections of the Maurer--Cartan-deformed differential. In particular, this allows us to give a combinatorial description of Hall algebra theta functions for acyclic quivers with non-degenerate skew-symmetrized Euler forms.
\end{abstract}

\maketitle

\section*{Introduction}
\subsection*{Motivation}

The notion of a scattering diagram was introduced by Kontsevich--Soibelman \cite{kontsevich-soibelman04} and Gross--Siebert \cite{gross2011real} in their studies of the reconstruction problem in Strominger--Yau--Zaslow mirror symmetry \cite{syz96}. In this setting, scattering diagrams encode and control the combinatorial data required to consistently glue local pieces of the mirror manifold. Since their introduction, scattering diagrams have found important applications to integrable systems \cite{kontsevich2014wall}, cluster algebras \cite{gross2018canonical}, enumerative geometry \cite{gross2010tropical} and combinatorics \cite{reading2017}, amongst other areas. Motivated by Fukaya's approach to the reconstruction problem \cite{fukaya05}, an asymptotic analytic perspective on scattering diagrams was developed in \cite{kwchan-leung-ma}. In this paper, we further develop this approach to give a differential geometric approach to refined and Hall algebra scattering diagrams.

The most basic form of scattering diagrams is closely related to the Lie algebra of Poisson vector fields on a torus. For many applications, it is necessary to study quantum, or refined, variants of scattering diagrams, in which the torus Lie algebra is replaced by the so-called quantum torus Lie algebra or, more generally, by an abstract monoid-graded Lie algebra satisfying a tropical condition \cite{kontsevich2014wall, gross2018canonical, mandel2015refined}. For example, a number of conjectures in the theory of quantum cluster algebras were proved using scattering diagram techniques in \cite{gross2018canonical}. Refined scattering diagrams were shown to be related to the refined tropical curve counting of Block--G\"{o}ttsche \cite{block2016refined} by Filippini--Stoppa \cite{filippini2015block} and Mandel \cite{mandel2015refined}, which also appear in study of $K3$ in \cite{lin2019refined}. These refined curve counts are also related to the refined enumeration of real plane curves by Mikhalkin \cite{mikhalkin2017quantum}, to higher genus Gromov--Witten invariants by Bousseau \cite{bousseau2018tropical}. These connections could be anticipated from the central role of scattering diagrams in the reconstruction problem.

A further generalization of scattering diagrams was introduced by Bridgeland \cite{bridgeland2017scattering} under the name $\mathfrak{h}$-complex. Here $\mathfrak{h}$ is a (not-necessarily tropical) monoid-graded Lie algebra. The flexibility of allowing non-tropical Lie algebras allows one to define, for example, scattering diagrams based on the motivic Hall--Lie algebra of a three dimensional Calabi--Yau category. Bridgeland showed that each quiver with potential $(\mathsf{Q},W)$ defines a consistent $\mathfrak{h}$-complex with values in the motivic Hall--Lie algebra, the wall-crossing automorphisms of the $\mathfrak{h}$-complex encoding the motivic Donaldson--Thomas invariants of $(\mathsf{Q},W)$. Under mild assumptions, the (refined) cluster scattering diagram of $(\mathsf{Q},W)$ is then obtained by applying a Hall algebra integration map to this $\mathfrak{h}$-complex. Using these ideas, Bridgeland was able to connect scattering diagrams to the geometry of the space of stability conditions on the triangulated category associated to $(\mathsf{Q},W)$.

In \cite{fukaya05}, Fukaya suggested that much of the combinatorial behavior of instanton corrections to the $B$-side complex structure which arise near the large volume limit could be described in terms of the asymptotic limit of Maurer--Cartan elements of the Kodaira--Spencer differential graded (dg) Lie algebra. In the context of scattering diagrams, this idea was made precise and put into practice in \cite{kwchan-leung-ma}, where it was shown that the asymptotic behavior of Maurer--Cartan elements of a certain dg Lie algebra admits an alternative interpretation in terms of consistent classical scattering diagrams. Moreover, the passage from an initial scattering diagram to its consistent completion, a procedure which exists due to works of Kontsevich--Soibelman \cite{kontsevich-soibelman04,kontsevich2014wall} and Gross--Siebert \cite{gross2011real}, can be understood in terms of the perturbative construction of Maurer--Cartan elements. These ideas were pursued in the setting of toric mirror symmetry to study the deformation theory of the Landau--Ginzburg mirror of a toric surface $X$ and its relation to tropical disk counting in $X$ \cite{kwchan-ma-p2}.

\subsection*{Main results}

In this paper we further develop the asymptotic approach to illustrate how refined (or more generally tropical) and Hall algebraic (or non-tropical) scattering diagrams, as well as the relevant theta functions, are controlled by asymptotic limits of Maurer--Cartan elements. To describe our results, we require some notation. Let $M$ be a lattice of rank $r$ and let $N = \Hom_{\mathbb{Z}}(M,\mathbb{Z})$. Write $M_\real = M \otimes_{\mathbb{Z}} \real$ and $N_\real = N \otimes_{\mathbb{Z}} \real$. Let $\sigma \subset M_\real$ be a strictly convex cone and set $M_\sigma^+ = (M \cap \sigma) \setminus \{0\}$. Let $\mathfrak{h}$ be a $M_\sigma^+$-graded Lie algebra, which for the moment we assume to be tropical.

Following \cite{kwchan-leung-ma, kwchan-ma-p2}, we consider differential forms on $M_\real$ which depend on a parameter $\hp \in \mathbb{R}_{> 0}$. Let $\mathcal{W}^{0}_*$ be the dg algebra of such differential forms which approach a bump form along a closed tropical polyhedral subset $P \subset M_{\mathbb{R}}$ as $\hp \rightarrow 0$. See Figure \ref{fig:delta_function}. The subspace $\mathcal{W}^{-1}_* \subset \mathcal{W}^{0}_*$ of differential forms which satisfy $\lim_{\hp \rightarrow 0} \alpha = 0$ is a dg ideal and
\[
\mathcal{H}^* := \bigoplus_{m \in M_\sigma^+} \left(\mathcal{W}^{0}_*/ \mathcal{W}^{-1}_*\right) \otimes_\comp \mathfrak{h}_{m}
\]
is a tropical dg Lie algebra. Our goal is to construct and interpret Maurer--Cartan elements of $\mathcal{H}^*$.

\begin{figure}[h]
	\centering
	\includegraphics[scale=0.25]{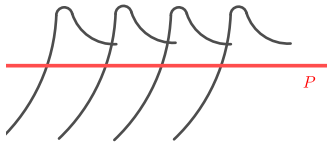}
	\caption{A bump form concentrating along $P$.}
	\label{fig:delta_function}
\end{figure}

Our first result relates Maurer--Cartan elements of $\mathcal{H}^*$ to the counting of tropical disks in $M_{\mathbb{R}}$. Let $\mathscr{D}_{in}$ be an initial scattering diagram. To each wall $\mathbf{w}$ of $\mathscr{D}_{in}$, whose support is a hyperplane $P_\mathbf{w}$ of $M_{\real}$ and whose wall-crossing factor is $\log(\Theta_\mathbf{w})$, we associate the term
\[
\incoming^{\mathbf{w}}:= -\delta_{P_\mathbf{w}} \log(\Theta_\mathbf{w}) \in \mathcal{H}^*.
\]
Here $\delta_{P_\mathbf{w}}$ is an $\hbar$-dependent $1$-form which concentrates along $P_\mathbf{w}$ as $\hp \rightarrow 0$. We take $\incoming = \sum_{\mathbf{w} \in \mathscr{D}_{in}} \incoming^{\mathbf{w}}$ as input data to solve the Maurer--Cartan equation. Our first main result, whose proof uses a modification of a method of Kuranishi \cite{Kuranishi65}, describes a Maurer--Cartan element $\varPhi$ constructed perturbatively from $\incoming$ using a propagator $\mathbf{H}$ (see Section \ref{sec:modified_homotopy_operator}).

\begin{theorem*}[See Theorems \ref{thm:theorem_1} and \ref{thm:theorem_1_modified}]\label{thm:introduction_theorem_1_modified}
	The Maurer--Cartan element $\varPhi$ can be written as a sum over tropical disks $\ltr$ in $(M_\real,\mathscr{D}_{in})$,
	$$
	\varPhi = \sum_{\ltr} \frac{1}{|\Aut(\ltr)|} \alpha_{\ltr} g_\ltr.
	$$
	Here $\alpha_\ltr$ is a $1$-form concentrated along $P_\ltr\subset M_\real$, the locus traced out by the stop of $\ltr$ as it varies in its moduli, and $g_\ltr$ is the Block--G\"{o}ttsche-type multiplicity of $\ltr$. Moreover, when $\dim_{\real}(P_\ltr) = r-1$, there exists a polyhedral decomposition $\mathcal{P}_\ltr$ of $P_\ltr$ such that, for each maximal cell $\sigma$ of $\mathcal{P}_\ltr$, there exists a constant $c_{\ltr,\sigma}$ such that 
	$
	\lim_{\hp\rightarrow 0 }\int_{\varrho} \alpha_{\ltr} = -c_{\ltr,\sigma}
	$
	for any affine line $\varrho$ intersecting positively with $\sigma$ in its relative interior.
	
	Furthermore, if we generically perturb the scattering diagram $\mathscr{D}_{in}$, then $c_{L,\sigma}=1$, so that the limit $\lim_{\hp\rightarrow 0 }\int_{\varrho} \alpha_{\ltr}$ is a count of tropical disks. 
	\end{theorem*}

In Section \ref{sec:non_perturbed_scattering} we associate to the Maurer--Cartan element $\varPhi$ a scattering diagram $\mathscr{D}(\varPhi)$. The walls of $\mathscr{D}(\varPhi)$ are labeled by the maximal cells $\sigma$ of the polyhedral decompositions $\mathcal{P}_L$ and have wall-crossing automorphisms $\exp(\frac{c_{L,\sigma}}{|\Aut (L)|} g_L)$. The diagram $\mathscr{D}(\varPhi)$ extends $\mathscr{D}_{in}$ and is in fact a consistent scattering diagram; see Proposition \ref{prop:consistence_from_solution}. In this way, we obtain an enumerative interpretation of the consistent completion of $\mathscr{D}_{in}$.

Next, we turn to theta functions. Let $\mathcal{C}$ be a cone satisfying $ \sigma \subset \mathcal{C} \subset M_{\real}$ with associated monoid $P = \mathcal{C} \cap M$ and let $A$ be a $P$-graded algebra with a graded $\mathfrak{h}$-action. The dg Lie algebra $\mathcal{H}^*$ acts naturally on the dg algebra
\begin{equation}\label{eqn:introduction_dga}
\mathcal{A}^* := \bigoplus_{m \in P} \left(\mathcal{W}^{0}_*/ \mathcal{W}^{-1}_*\right) \otimes_\comp A_{m}.
\end{equation}
Given a Maurer--Cartan element $\varPhi \in \mathcal{H}^*$, it is natural to study the space of flat sections $\text{Ker}(d_\varPhi)$ of the deformed differential $d_\varPhi = d + [\varPhi,\cdot ]$. The algebra structure on $\mathcal{A}^*$ induces an algebra structure on $\Ker(\varPhi)$. The following result describes the wall-crossing behavior of the $\hp \rightarrow 0$ limit of flat sections.

\begin{theorem*}[See Theorem \ref{thm:wall_crossing_for_sections}]
	Let $s \in \Ker (d_\varPhi)$ and $Q,Q^{\prime} \in M_\real \setminus \Supp (\mathscr{D}(\varPhi))$. Then, for any path $\gamma \subset M_\real \setminus \Joints(\mathscr{D})$ from $Q$ to $Q^{\prime}$, we have
	\[
	\lim_{\hp \rightarrow 0}s_{Q^{\prime}} = \Theta_{\gamma,\mathscr{D}}(\lim_{\hp \rightarrow 0} s_{Q}),
	\]
	where $\Theta_{\gamma,\mathscr{D}(\varPhi)}$ is a wall-crossing factor and $s_{Q^{\prime}}$, $s_Q$ are the restrictions of $s$ to $Q$, $Q^{\prime}$, respectively.
	\end{theorem*}

To connect with theta functions, we work in the square zero extension dg Lie algebra $\mathcal{H}^* \oplus \mathcal{A}^*[-1]$ where, for each $m \in P$, we perturbatively solve the Maurer--Cartan equation with input $\incoming + z^m$. The resulting Maurer--Cartan element is of the form $\varPhi + \theta_m$, with $\varPhi$ as above and $\theta_m \in \Ker (d_\varPhi)$. On the other hand, associated to $m \in P$ is the (standard) theta function
\[
\vartheta_{m,Q} := \sum_{\substack{\textnormal{broken lines } \gamma\\ \textnormal{ending at $(m,Q)$}}} a_{\gamma}
\]
defined in terms of broken lines ending at $(m,Q)$, that is, piecewise linear maps $\gamma : (-\infty,0] \rightarrow M_\real$ which bend only at the walls of $\mathscr{D}(\varPhi)$. Each broken line $\gamma$ has an associated weight $a_\gamma \in A$.

\begin{theorem*}[See Theorem \ref{thm:theta_function_comparsion}]
	The equality
	$$
	\lim_{\hp \rightarrow 0} \theta_{m}(Q) = \vartheta_{m,Q}
	$$
	holds for all $Q \in M_\real \setminus \Supp (\mathscr{D}(\varPhi))$, where $\theta_{m}(Q)$ denotes the value of $\theta_{m}$ at $Q$.
\end{theorem*}

Finally, in Section \ref{sec:hall_algebra} we study the above constructions in the setting of non-tropical Lie algebras. One advantage of the differential geometric approach of this paper is that it is applicable to non-generic cases without perturbing $\mathscr{D}_{in}$. With a mild commutativity condition on the wall-crossing automorphisms of the walls of $\mathscr{D}_{in}$ which, for example, is satisfied in the Hall algebra setting, we obtain new results in the non-tropical case, where perturbation of $\mathscr{D}_{in}$ is not possible. Theorem \ref{thm:consistent_scattering_nontropical} generalizes to the non-tropical setting the construction of a Maurer--Cartan element $\varPhi$ from an initial scattering diagram $\mathscr{D}_{in}$ and associates to $\varPhi$ a consistent completion of $\mathscr{D}_{in}$. We also prove that the completed scattering diagram is equivalent to that constructed algebraically by Bridgeland \cite{bridgeland2017scattering}. Moreover, we construct, for each $n \in N_{\sigma}^+$, a theta function $\theta_n \in \Ker(d_{\varPhi})$ as a perturbative Maurer--Cartan element and prove that it agrees with Bridgeland's Hall algebra theta function \cite{bridgeland2017scattering}.

In Section \ref{sec:hall_algebra_acyclic_quiver} we restrict attention to the case in which the Lie algebra is the motivic Hall--Lie algebra of an acyclic quiver. In this case, there is a canonical choice for the propagator $\mathbf{H}$, leading to a combinatorial formula for $\varPhi$ and $\theta_n$ in terms of tropical disks.

\begin{theorem*}[See Theorem \ref{thm:non_tropical_WCF}]
\phantomsection
Let $\mathfrak{h}$ be the Hall--Lie algebra of an acyclic quiver with non-degenerate skew-symmetrized Euler form. Then $\varPhi$ can be written as a sum over labeled trees, 
$$
\varPhi = \sum_{k \geq 1} \sum_{\substack{\ltr \in \ltree{k}\\ \mathfrak{M}_{\ltr}(N_\real,\mathscr{D}_{in}) \neq \emptyset}} \frac{1}{|\Aut(\ltr)|} \alpha_{\ltr} g_\ltr,
$$
and $\theta_n$ can be written as a sum over marked tropical trees,
$$
\theta_n = \sum_{k \geq 1} \sum_{\substack{\mtr \in \mtree{k}\\ P_{\mtr} \neq \emptyset}} \frac{1}{|\Aut(J)|} \alpha_{\mtr} a_{\mtr}.
$$
Moreover, $\theta_n$ is related to Bridgeland's Hall algebra theta function $\vartheta_{n,Q}$ by $\vartheta_{n,Q} = \theta_n(Q)$.
\end{theorem*}

Here $\ltree{k}$ and $\mtree{k}$ are the sets of labeled, respectively marked, $k$-trees and $g_{\ltr}$, $a_\mtr$ are Hall algebraic Block--G\"{o}ttsche-type multiplicities. The formula for $\theta_n$ can be regarded as a replacement for a description of $\vartheta_{n,Q}$ in terms of Hall algebra broken lines. Indeed, it was recently shown by Cheung and Mandel \cite{cheung2019} that, contrary to Bridgeland's theta functions, Hall algebra theta functions which are defined in terms of Hall algebra broken lines do not, in general, satisfy the wall-crossing formula.

\subsection*{Acknowledgements}
The authors would like to thank Kwokwai Chan, Man-Wai Cheung and Travis Mandel for many useful discussions and suggestions. Naichung Conan Leung was supported in part by a grant from the Research Grants Council of the Hong Kong Council of the Hong Kong Special Administrative Region, China (Project No. CUHK 14302215 and 14303516).

\section{Scattering diagrams and theta functions}
\label{sec:counting}

We collect background material on scattering diagrams and theta functions. Fix a lattice $M$ of rank $r$ with dual lattice $N = \Hom_{\mathbb{Z}}(M,\mathbb{Z})$. Write $\langle \cdot,\cdot \rangle : M \times N \rightarrow \mathbb{Z}$ for the canonical pairing. Let $M_\real = M \otimes_\inte \real$ and $N_\real = N \otimes_\inte \real$.



\subsection{Tropical Lie algebras and scattering diagrams}\label{sec:scattering_diagram}
Following \cite{gross2018canonical, mandel2015refined}, we recall the definition of scattering diagrams. Compared to \cite{mandel2015refined}, the roles of $M$ and $N$ are reversed.

\subsubsection{Tropical Lie algebras}\label{sec:tropical_Lie_algebras}

Fix a strictly convex polyhedral cone $\sigma \subset M_\real$. Let $M_\sigma = \sigma \cap M$ and $M_\sigma^+ = M_\sigma \setminus \{0\}$. For each $k \in \inte_{> 0}$, set $k M_\sigma^+ = \{ m_1 + \cdots + m_k \mid m_i \in M_\sigma^+\}$.

Let $\mathfrak{h}= \bigoplus_{m \in M^{+}_\sigma} \mathfrak{h}_m$ be a $M^+_\sigma$-graded Lie algebra over $\comp$. For each $k \in \mathbb{Z}_{> 0}$, set $\mathfrak{h}^{\geq k}= \bigoplus_{m \in kM_\sigma^+} \mathfrak{h}_m$. Then $\mathfrak{h}^{< k}:= \mathfrak{h}/\mathfrak{h}^{\geq k}$ is a nilpotent Lie algebra. Associated to the pro-nilpotent Lie algebra $\hat{\mathfrak{h}}:= \varprojlim_{k} \mathfrak{h}^{< k}$ is the exponential group $\hat{G} := \exp(\hat{\mathfrak{h}})$. Similarly, for each $m \in M^{+}_\sigma$, set $\mathfrak{h}_m^{\parallel} = \bigoplus_{k\geq 1} \mathfrak{h}_{km}$ and $\hat{\mathfrak{h}}^{\parallel}_{m} = \prod_{k\in \mathbb{Z}_{>0}} \mathfrak{h}_{km} \subset \hat{\mathfrak{h}}$ with associated exponential group $\hat{G}_m^{\parallel}$.

To define theta functions, we require a second (not necessarily strictly) convex polyhedral cone $\mathcal{C} \subsetneq M_\real$ which contains $\sigma$. Let $P = \mathcal{C} \cap M$ be the corresponding monoid. Suppose that $\mathfrak{h}$ acts on a $P$-graded $\comp$-algebra $A = \bigoplus_{m \in P} A_m$ by derivations so that $\mathfrak{h}_m \cdot A_{m'} \subset A_{m+m'}$. Then $A^{\geq k} := \bigoplus_{m \in kM_\sigma^+ + P} A_m$ is a graded ideal of $A$. Set $A^{< k} = A \slash A^{\geq k}$ and $\hat{A} = \varprojlim_k A^{<k}$. There is an induced action of $\hat{\mathfrak{h}}$, and hence also of $\hat{G}$, on the algebra $\hat{A}$.

More generally, given a sublattice $L \subset M$, let $\mathfrak{h}_L = \bigoplus_{m \in L\cap M_\sigma^+ } \mathfrak{h}_m$ and $A_L = \bigoplus_{m \in L\cap P } A_m$ with associated completions $\hat{\mathfrak{h}}_L$ and $\hat{A}_L$.

Let $K \subset M$ be a saturated sublattice which satisfies the following conditions:
\begin{enumerate}
	\item $\mathfrak{h}_K$ is a central Lie subalgebra of $\mathfrak{h}$.
	
	\item The induced $\mathfrak{h}_K$-action on $A$ is trivial.
	
	\item The induced $\mathfrak{h}$-action on $A_K$ is trivial. 
	\end{enumerate}
Denote by $\pi_K : M \rightarrow \overline{M}:=M/K$ the canonical projection and by $\overline{N}:= \overline{M}^\vee \hookrightarrow N$ the embedding of $\overline{M}^\vee$ into $N$ as the orthogonal $K^\perp$.

The following assumption will be used in Section \ref{sec:theta_function_def}.

\begin{assum}[\cite{mandel2015refined}]
\phantomsection
\label{assum:P_and_K_relation}
\begin{enumerate}
	\item The monoid $P$ satisfies $\overline{M} = \pi_K(P)$.
	\item There is a fan structure on $\overline{M}_\real$ and a piecewise linear section $\varphi: \overline{M}_\real \rightarrow M_\real$ of $\pi_K$ which satisfies $\varphi(0) = 0$ and $P=\varphi(\overline{M})+(K\cap P)$. 
	
	\item We are given elements $z^{\varphi(\mathsf{m})} \in A_{\varphi(\mathsf{m})}$, $\mathsf{m} \in \overline{M}$, which satisfy
	\begin{enumerate}
		\item $z^{\varphi(0)}=1$,
		\item for any $a \in \hat{A}_K \setminus \{0\}$ and $\mathsf{m} \in \overline{M}$, we have $a z^{\varphi(\mathsf{m})} \neq 0$, and
		
		\item for any $\mathsf{m} \in \overline{M}$, we have $ A_{\varphi(\mathsf{m})+P\cap K}= z^{\varphi(\mathsf{m})} A_K$.
	\end{enumerate}
\end{enumerate}
\end{assum}

\begin{definition}\label{def:tropical_lie_algebra_for_diagram}
	The Lie algebra $\mathfrak{h}$ is called {\em tropical} if, for each pair $(m, n) \in  M_\sigma^+ \times\overline{N}$ satisfying $\langle m,n \rangle = 0$, it is equipped with a subspace $\mathfrak{h}_{m,n} \subset \mathfrak{h}_{m}$. These subspaces are required to satisfy
	\begin{enumerate}
		\item $\mathfrak{h}_{m,0} = \{0\}$ and $\mathfrak{h}_{m,kn} = \mathfrak{h}_{m,n}$ for each $k \neq 0$,
		
		\item $[\mathfrak{h}_{m_1,n_1},\mathfrak{h}_{m_2,n_2}] \subset \mathfrak{h}_{m_1+m_2,n}$, where $n = \langle m_2,n_1 \rangle n_2 - \langle m_1,n_2 \rangle n_1$, and
		
		\item $\mathfrak{h}_{m_1,n} \cdot A_{m_2} = \{0\}$ if $\langle m_2,n \rangle = 0$.
	\end{enumerate} 
\end{definition}

Examples of tropical and non-tropical Lie algebras can be found in \cite[Example 2.1]{mandel2015refined}. See also Section \ref{sec:motivicHall}. Until mentioned otherwise, we will assume that $\mathfrak{h}$ is tropical.

Observe that if $(m,n) \in M^{+}_\sigma \times \overline{N}$ with $\langle m,n \rangle = 0$, then $\mathfrak{h}_{m,n}^\parallel:= \bigoplus_{k\in \inte_{>0}} \mathfrak{h}_{km,n}$ is an abelian Lie subalgebra of $\mathfrak{h}_{m}^\parallel$. Denote by $\hat{\mathfrak{h}}_{m,n}^\parallel$ the completion of $\mathfrak{h}_{m,n}^\parallel$.

Finally, given a commutative unital $\comp$-algebra $R$, there are $R$-linear versions of the above definitions. For example, $\mathfrak{h}_R := \mathfrak{h} \otimes_\comp R$ is a Lie algebra over $R$ which acts on $A \otimes_{\comp} R$ by the $R$-linear extension of the rule $t_1 h \cdot t_2 a= t_1 t_2 (h \cdot a)$. The completion $\hat{\mathfrak{h}}\hat{\otimes}_{\mathbb{C}} R$ acts on $\hat{A} \hat{\otimes}_{\mathbb{C}}R$. The corresponding exponential group is $G_R$ with completion $\hat{G}_R$. Similarly, there are abelian Lie subalgebras $\hat{\mathfrak{h}}_{m,n,R}^\parallel \subset \hat{\mathfrak{h}}_{m,R}^\parallel$ and, given a saturated sublattice $L \subset M$, we can form $\mathfrak{h}_{L,R}$, $A_{L,R}$ and so on.

\subsubsection{Scattering diagrams}\label{sec:scattering_diagram_def}

We continue to follow \cite{gross2018canonical, mandel2015refined}. Fix a commutative unital $\mathbb{C}$-algebra $R$. Recall that $r$ is the rank of $M$.

\begin{definition}\label{wall}
	A {\em wall} $\mathbf{w}$ (over $R$) in $M_\real$ is a tuple $(m,n, P, \Theta)$ consisting of
	\begin{enumerate}
		\item
		a primitive element $m \in M^{+}_\sigma$ and an element $n \in \overline{N} \setminus \{0\}$ which satisfy $\langle m,n \rangle=0$,
		\item
		an $(r-1)$-dimensional closed convex rational polyhedral subset $P$ of $m_0 + n^{\perp}\subset M_\real $ for some $m_0 \in M_\real$, called the {\em support} of $\mathbf{w}$, and
		\item
		an element $\Theta \in \hat{G}_{m,n,R}:= \exp( \hat{\mathfrak{h}}^\parallel_{m,n,R})$, called the {\em wall-crossing automorphism} of $\mathbf{w}$.
	\end{enumerate}
\end{definition}

A wall $\mathbf{w} = (m,n,P, \Theta)$ is called {\em incoming} (resp. {\em outgoing}) if $P + t m \subset P$ for all $t \in \real_{>0}$ (resp.  $t \in \real_{\leq 0}$). The vector $-m$ is called the {\em direction} of $\mathbf{w}$.

\begin{definition}\label{def:scattering_diagram}
	A {\em scattering diagram} $\mathscr{D}$ over $R$ is a countable set of walls $\left\{ ( m_i, n_i, P_i, \Theta_i) \right\}_{i\in I}$ such that, for each $k \in \mathbb{Z}_{> 0}$, the image of $\log(\Theta_i)$ in $\mathfrak{h}^{< k} \otimes_\comp R$ is zero for all but finitely many $i \in I$.
\end{definition}

Let $k \in \mathbb{Z}_{> 0}$. Using the canonical projection $\hat{\mathfrak{h}}_R \rightarrow \mathfrak{h}^{< k} \otimes_\comp R$, a scattering diagram $\mathscr{D}$ induces a finite scattering diagram $\mathscr{D}^{< k}$ with wall-crossing automorphisms in $\exp(\mathfrak{h}^{< k} \otimes_\comp R)$.

	The {\em support} and {\em singular set} of a scattering diagram $\mathscr{D}$ are
	$$
	\Supp(\mathscr{D}) := \bigcup_{\mathbf{w} \in \mathscr{D}} P_{\mathbf{w}},
	\qquad
	\Joints(\mathscr{D}) := \bigcup_{\mathbf{w} \in \mathscr{D}} \partial P_{\mathbf{w}} \cup \bigcup_{\substack{ \mathbf{w}_1, \mathbf{w}_2 \in \mathscr{D} \\ \dim(\mathbf{w}_1 \cap \mathbf{w}_2)=r-2}} P_{\mathbf{w}_1} \cap P_{\mathbf{w}_2}.
	$$

\subsubsection{Path ordered products}\label{sec:path_ordered_product}

An embedded path $\gamma: [0,1] \rightarrow N_\real \setminus \Joints(\mathscr{D})$ is said to {\em intersect $\mathscr{D}$ generically}
if $\gamma$ intersects all walls of $\mathscr{D}$ transversally, $\gamma(0), \gamma(1) \notin \Supp(\mathscr{D})$ and $\text{Im}(\gamma) \cap \Joints(\mathscr{D}) = \emptyset$. The {\em path ordered product} of such a path is $\Theta_{\gamma, \mathscr{D}}:= \varprojlim_k \Theta_{\gamma,\mathscr{D}}^{< k }$, where $
\Theta_{\gamma, \mathscr{D}}^{< k} := \prod^{\gamma}_{\mathbf{w} \in \mathscr{D}^{< k}} \Theta_{\mathbf{w}} \in \exp(\mathfrak{h}^{< k} \otimes_{\mathbb{C}} R)$
is defined in \cite[\S 1.3]{gross2010tropical}.

\begin{definition}\label{consistent_def}
\begin{enumerate}
\item	A scattering diagram $\mathscr{D}$ is called {\em consistent} if $\Theta_{\gamma,\mathscr{D}} = \Id$ for any embedded loop $\gamma$ intersecting $\mathscr{D}$ generically.

\item Scattering diagrams $\mathscr{D}_1$, $\mathscr{D}_2$ are called {\em equivalent} if $\Theta_{\gamma,\mathscr{D}_1} = \Theta_{\gamma,\mathscr{D}_2}$
	for any embedded path $\gamma$ intersecting both $\mathscr{D}_1$ and $\mathscr{D}_2$ generically.
\end{enumerate}
\end{definition}

The following result is fundamental in the theory of scattering diagrams.

\begin{theorem}[\cite{kontsevich-soibelman04,gross2011real}]\label{thm:kontsevich_soibelman_thm}
	Let $\mathscr{D}_{in}$ be a scattering diagram consisting of finitely many walls supported on full affine hyperplanes. Then there exists a scattering diagram $\mathcal{S}(\mathscr{D}_{in})$ which is consistent and is obtained from $\mathscr{D}_{in}$ by adding only outgoing walls. Moreover, the scattering diagram $\mathcal{S}(\mathscr{D}_{in})$ is unique up to equivalence.
\end{theorem}

Using asymptotic analytic techniques, an independent proof of the existence part of Theorem \ref{thm:kontsevich_soibelman_thm} will be given in Proposition \ref{prop:consistence_from_solution}.

\subsection{Broken lines and theta functions}\label{sec:theta_function_def}

We follow \cite{mandel2015refined} to define broken lines. Fix a consistent scattering diagram $\mathscr{D}$ over $R$.


\begin{definition}\label{def:broken_lines}
	A {\em broken line $\gamma$ with end} $(\mathsf{m},Q) \in \overline{M} \setminus \{0\} \times M_{\real} \setminus \Supp(\mathscr{D})$ is the data of a partition $-\infty < t_0 \leq t_1 \leq \cdots \leq t_l = 0$, a piecewise linear map $\gamma : (-\infty,0] \rightarrow M_\real \setminus \Joints(\mathscr{D})$ and elements $a_i \in A_{m_i}\otimes_{\mathbb{C}} R$, $i = 0 ,\dots, l$, with $ m_i \neq 0$. This data is required to satisfy the following conditions:
	\begin{enumerate}
		\item  $a_0 = z^{\varphi(\mathsf{m})}$.
		
		\item $\gamma(0) = Q$.
		
		\item $\{t_0,\dots,t_{l-1}\} \subseteq \gamma^{-1}(\Supp(\mathscr{D}))$. 
		
		\item $\gamma'|_{(t_{i-1},t_i)} \equiv -m_i$ for $i = 0,\dots,l$, where $t_{-1} := -\infty$, and all {\em bends} $m_{i+1} - m_i$ are non-zero.
		
		\item For each $i = 0,\dots, l-1$, set $
		\Theta_i := \prod_{\substack{\mathbf{w} \in \mathscr{D}\\ \gamma(t_i) \in P_\mathbf{w}}} \Theta_{\mathbf{w}}^{\sgn \langle m_i, n_\mathbf{w} \rangle } \in \hat{G}_R$. Then $a_{i+1}$ is a homogeneous summand of $\Theta_i \cdot a_i$.
	\end{enumerate}
\end{definition}

In the notation of Definition \ref{def:broken_lines}, we will write $a_\gamma$ for $a_{l}$.

\begin{definition}\label{def:theta_functions}
	The {\em broken line theta function} associated to $(\mathsf{m},Q) \in \overline{M} \setminus \{0\} \times M_{\real} \setminus \Supp(\mathscr{D})$ is
	\[
	\vartheta_{\mathsf{m},Q} = \sum_{\textnormal{End}(\gamma) = (\mathsf{m},Q)} a_{\gamma} \in \hat{A} \hat{\otimes}_{\mathbb{C}} R,
	\]
	the sum being over all broken lines with end $(\mathsf{m},Q)$. Define also $\vartheta_{0,Q} = 1$.
\end{definition}

In the present setting, well-definedness of theta functions was proved in \cite{mandel2015refined}. Observe that 
$
\vartheta_{\mathsf{m},Q} \in z^{\varphi(\mathsf{m})} + \hat{A}_{\varphi(\mathsf{m})+M_\sigma^+},
$
where $\hat{A}_{\varphi(\mathsf{m})+M_\sigma^+}$ is the completion of $A_{\varphi(\mathsf{m})+M_\sigma^+} \otimes_{\mathbb{C}} R$.

\begin{prop}[\cite{CPS,mandel2015refined}]\label{prop:theta_functions_properties}
Under Assumption \ref{assum:P_and_K_relation}, the following statements hold:
\begin{enumerate}
\item For each $Q \in M_\real \setminus \Supp(\mathscr{D})$, the set $\{ \vartheta_{\mathsf{m},Q}\}_{\mathsf{m} \in \overline{M}}$ is linearly independent over $\hat{A}_K \hat{\otimes}_{\mathbb{C}} R$ and, for each $k \in \mathbb{Z}_{> 0}$, additively generates $A^{<k} \otimes_{\mathbb{C}} R$ over $A_K^{<k} \otimes_{\mathbb{C}} R$.

\item Let $\mathscr{D} = \mathcal{S}(\mathscr{D}_{in})$ and let $\rho : [0,1] \rightarrow M_\real \setminus \Joints(\mathscr{D})$ be a path with generic endpoints which do not lie in $\Supp(\mathscr{D})$. Then the equality
	$
	\vartheta_{\mathsf{m},\rho(1)} = \Theta_{\rho,\mathscr{D}}(\vartheta_{\mathsf{m},\rho(0)})
	$
	holds for all $\mathsf{m} \in \overline{M}$.
\end{enumerate}
\end{prop}

\subsection{Tropical disk counting}\label{sec:tropical_counting}
We recall some definitions from \cite{mandel2015refined}, modified so as to incorporate the work of \cite{kwchan-leung-ma}. Fix a scattering diagram $\mathscr{D}_{in} = \{ \mathbf{w}_i = (m_i, n_i, P_i,\Theta_i)\}_{i \in I}$ and let $g_i = \log(\Theta_i)$. Write
\begin{equation}\label{eqn:initial_walls}
g_i = \sum_{j \geq 1} g_{ji} \in \big( \prod_{j \geq 1} \mathfrak{h}_{jm_i} \big) \cap \hat{\mathfrak{h}}_{m_i,n_i}^\parallel, \qquad g_{ji} \in \mathfrak{h}_{jm_i}.
\end{equation}   

For each $l \geq 0$, define commutative rings $R =\comp[\{t_i \mid  i\in I\}]$ and $R_l = \comp[\{t_i \mid  i\in I\}]/\langle t_i^{l+1}\ | \ i\in I\rangle$, as in \cite{gross2010tropical, mandel2015refined}. There is a ring homomorphism
\[
R_l \rightarrow \tilde{R}_{l}:= \frac{\comp\left[\{u_{ij} \mid i \in I,\ 1 \leq j \leq l\}\right]}{\langle u_{ij}^2  \mid i \in I,\ 1 \leq j \leq l \rangle}, \qquad t_i \mapsto \sum_{j=1}^l u_{ij}.
\]


\begin{definition}\label{def:initial_diagram}
	A {\em perturbation} $\tilde{\mathscr{D}}_{in,l}$ of $\mathscr{D}_{in}$ over $\tilde{R}_{l}$ is a scattering diagram over $\tilde{R}_{l}$ consisting of a wall
	$\mathbf{w}_{iJ} = (m_i, n_i, P_{iJ}, \Theta_{iJ})$
	for each $i \in I$ and $J \subset \{1,\dots,l\}$ with $\# J \geq 1$ such that
	\begin{enumerate}
		\item
		each $P_{iJ}$ is a translate of $n_i^\perp$ and $P_{iJ} \neq P_{i'J'}$ unless $i=i^{\prime}$ and $J = J^{\prime}$, and
		\item
		the equality $\log(\Theta_{iJ}) = (\# J)! g_{(\#J) i} \prod_{s\in J} u_{is}$
 		holds.
 	\end{enumerate}
\end{definition}

We follow \cite{kwchan-leung-ma, filippini2015block, gross2010tropical, Mikhalkin05} and introduce tropical disks in $\mathscr{D}_{in}$ or $\tilde{\mathscr{D}}_{in,l}$.

\begin{definition}\label{def:k_tree_def}\label{def:ribbon_k_tree}
	A {\em (directed) $k$-tree} $\tr$ is the data of finite sets of vertices $\bar{\tr}^{[0]}$ and edges $\bar{\tr}^{[1]}$, a decomposition $\bar{\tr}^{[0]} = \tr^{[0]}_{in} \sqcup \tr^{[0]} \sqcup \{v_{out}\}$ into incoming, internal and outgoing vertices, and boundary maps $\partial_{in} , \partial_{out} : \bar{\tr}^{[1]} \rightarrow \bar{\tr}^{[0]}$. This data is required to satisfy the following conditions:
	\begin{enumerate}
		\item The set $\tr^{[0]}_{in}$ has cardinality $k$.
		
		\item
		Each vertex $v \in \tr^{[0]}_{in}$ is univalent and satisfies $\# \partial_{out}^{-1}(v) = 0$ and $\# \partial_{in}^{-1}(v) = 1$.
		
		\item
		Each vertex $v \in \tr^{[0]}$ is trivalent and satisfies $\# \partial_{out}^{-1}(v) = 2$ and $\# \partial_{in}^{-1}(v) = 1$.
		
		\item
		We have $\# \partial_{out}^{-1}(v_{out}) = 1$ and $\# \partial_{in}^{-1}(v_{out}) = 0$.
		
		\item
		The topological realization
		$|\bar{\tr}| := \left( \coprod_{e \in \bar{\tr}^{[1]}} [0,1] \right) / \sim$, where $\sim$ is the equivalence relation which identifies boundary points of edges if their images in $\tr^{[0]}$ agree, is connected and simply connected.
	\end{enumerate}
\end{definition}

Two $k$-trees are isomorphic if there exist bijections between their sets of vertices and edges which preserve the respective decompositions and boundary maps. Set $T^{[0]}_{\infty} = T_{in}^{[0]} \sqcup \{v_{out}\}$ and $\tr^{[1]} = \bar{\tr}^{[1]}\setminus \partial_{in}^{-1}(\tr^{[0]}_{in})$. The edge $e_{out} := \partial_{out}^{-1}(v_{out})$ is called the {\em outgoing edge}. The {\em root vertex} $v_r$ is the unique vertex satisfying $e_{out} = \partial^{-1}_{in}(v_r)$.


\begin{definition}
\phantomsection
\label{def:weighted_tree_tropical}\label{label_tree_def}
	\begin{enumerate}
	\item A {\em labeled $k$-tree} is a $k$-tree $\ltr$ with a labeling of each edge $e \in \partial_{in}^{-1}(\ltr^{[0]}_{in})$ by a wall $\mathbf{w}_{i_e} = (m_{i_e},n_{i_e},P_{i_e},\Theta_{i_e})$ in $\mathscr{D}_{in}$ and an element $m_e\in M_\sigma^+$ such that $m_e = k_{e} m_{i_e}$ for some $k_{e} \in \mathbb{Z}_{> 0}$.
	
	\item A {\em marked $k$-tree} is a $k$-tree $\mtr$ with a marked edge $\breve{e} \in \partial_{in}^{-1}(\ltr^{[0]}_{in})$ and an associated element $m_{\breve{e}} = \varphi(\mathsf{m})$ for some $\mathsf{m} \in \overline{M} \setminus \{0\}$, together with a labeling of each edge $e \in \partial_{in}^{-1}(\ltr^{[0]}_{in}) \setminus \{\breve{e}\}$ by a wall $\mathbf{w}_{i_e} = (m_{i_e},n_{i_e},P_{i_e},\Theta_{i_e})$ in $\mathscr{D}_{in}$ as for labeled $k$-tree.
		
	\item A {\em weighted $k$-tree} is a $k$-tree $\wtr$ with a weighting of each edge $e \in \partial_{in}^{-1}(\wtr^{[0]}_{in})$ by a wall $\mathbf{w}_{i_e J_e} = (m_{i_e},n_{i_e},P_{i_e J_e},\Theta_{i_e J_e})$ in $\tilde{\mathscr{D}}_{in,l}$ and a pair $(m_e, u^{\vec{J}_e})$, where $u^{\vec{J}_e} := \prod_{i \in I} \prod_{j \in J_{e,i}} u_{ij} \in \tilde{R}_{l}$, such that $m_e= (\# J_e) m_{i_e}$ and $\vec{J}_e$ is an $I$-tuple of finite subsets of $\{1,\dots,l\}$ such that $J_{e,i_e} =J_e$ and $J_{e,j} = \emptyset $ for $j \in I \setminus \{i_e\}$. Moreover, the weights of incoming edges are required to be pairwise distinct.
	\end{enumerate}
\end{definition}

	Two labeled $k$-trees are isomorphic if they are isomorphic as $k$-trees by a label preserving isomorphism, and similarly for marked and weighted cases. The set of isomorphism classes of labeled, marked and weighted $k$-trees will be denoted by $\ltree{k}$, $\mtree{k}$ and $\wtree{k}$, respectively.

Let $\ltr$ be a labeled $k$-tree. Inductively define a labeling of all edges of $\ltr$ by requiring that for a vertex $v \in \ltr^{[0]}$ with incoming edges $e_1, e_2$ (so that $\partial_{out}^{-1}(v) = \{e_1,e_2\}$) and outgoing edge $e_3$, the equality $m_{e_3} = m_{e_1} + m_{e_2}$ holds. A similar procedure applies to marked and weighted $k$-trees, where in the latter case we also require $u^{\vec{J}_{e_3}} = u^{\vec{J}_{e_1}} u^{\vec{J}_{e_2}}$. Write $m_{\ltr \slash \mtr \slash \wtr} = m_{e_{out}}$ and $u^{\vec{J}_{\wtr}} = u^{\vec{J}_{e_{out}}}$.


\begin{definition}
	\label{def:weighted_ribbon_k_tree}
	A {\em labeled ribbon $k$-tree} $\lrtr$ is a labeled $k$-tree with a ribbon structure, that is, a cyclic ordering of $\partial_{in}^{-1}(v) \sqcup \partial_{out}^{-1}(v)$ for each $v \in \lrtr^{[0]}$. A {\em marked ribbon $k$-tree} is defined analogously.
\end{definition}

Labeled ribbon $k$-trees are isomorphic if they are isomorphic as $k$-trees by an isomorphism which preserves the ribbon structure and labels. The set of isomorphism classes of labeled ribbon $k$-trees will be denoted by $\lrtree{k}$. Similarly, $\mrtree{k}$ and $\wrtree{k}$ are the sets of isomorphism classes of marked and weighted ribbon $k$-trees, respectively. Note that the topological realization of a labeled (or marked, weighted) ribbon $k$-tree $\lrtr$ admits a canonical embedding into the unit disc $D$ so that $\lrtr^{[0]}_\infty \subset \partial D$.

\begin{definition}[\cite{mandel2015refined}]\label{def:multiplicity}
	Given a labeled $k$-tree $\ltr$ (resp. weighted $k$-tree $\wtr$), associate to each $e \in \bar{\ltr}^{[1]}$ (resp. $e \in \bar{\wtr}^{[1]}$) a pair $\pm(n_e,g_e)$, defined up to sign,\footnote{As the signs of $n_{e_3}$ and $g_{e_3}$ depend on the cyclic ordering $e_1,e_2,e_3$ in the same way, only $\pm(n_{e_3}, g_{e_3})$ is defined.} with $n_e \in \overline{N}$ and $g_{e} \in \mathfrak{h}_{m_e,n_e}$ (resp. $g_{e} \in \mathfrak{h}_{m_e,n_e,\tilde{R}_l}$), inductively along the direction of the tree as follows:
	\begin{enumerate}
		\item Associated to each $e \in \partial_{in}^{-1}(\ltr^{[0]}_{in})$ (resp. $e \in \partial_{in}^{-1}(\wtr^{[0]}_{in})$) is a unique initial wall $\mathbf{w}_{i_e} =  (m_{i_e},n_{i_e}, P_{i_e}, \Theta_{i_e})$ (resp. $\mathbf{w}_{i_eJ_e} =  (m_{i_e},n_{i_e}, P_{i_eJ_e}, \Theta_{i_eJ_e})$). Set $n_e = n_{i_e}$ and $g_{e} = g_{k_{e} i_e }$ (resp. $g_{e} = g_{(\#J_{e,i_e}) i_e}$), where $g_{ji}$ is given by equation \eqref{eqn:initial_walls}.
		
		\item At a trivalent vertex $v \in \ltr^{[0]}$ (resp. $v \in \wtr^{[0]}$) with incoming edges $e_1, e_2$ and outgoing edge $e_3$, set $n_{e_3}  = \langle m_{e_2},n_{e_1} \rangle n_{e_2} - \langle m_{e_1},n_{e_2} \rangle n_{e_1}$ and $g_{e_3}  = [g_{e_1},g_{e_2}]$.
\end{enumerate}

For a labeled (resp. weighted) ribbon tree $\lrtr$ (resp. $\wrtr$), the label $(n_e,g_e)$ of $e \in \bar{\lrtr}^{[1]}$ (resp. $e \in \bar{\wrtr}^{[1]}$) can be defined without the sign ambiguity by requiring that $\{e_1,e_2,e_3\}$ be clockwise oriented.
\end{definition}

Write $(n_\ltr, g_\ltr)$ or $(n_\wtr, g_\wtr)$ for the pair associated to $e_{out}$. Note that if $v \in \wtr^{[0]}$ has incoming edges $e_1, e_2$ and outgoing edge $e_3$ and $m_{e_1}, m_{e_2} \in M_{\real}$ are linearly dependent, then $n_{e_3} = 0$ and hence $g_\wtr = 0$, as follows from the vanishing $\mathfrak{h}_{m,0} = \{0\}$.

\begin{definition}\label{def:marked_tree_core}
	The {\em core} $\mathfrak{c}_{\mtr}$ of $\mtr \in \mtree{k}$ is the directed path of edges $\breve{e} =e_0, e_1,\dots,e_l = e_{out}$ joining $\breve{e}$ to $e_{out}$. Removing $\mathfrak{c}_{\mtr}$ results in $l$ disconnected labeled trees $\ltr_{1},\dots,\ltr_{l}$ according to the order of attaching to $\mathfrak{c}_{\mtr}$. We assign $a_{e_i} \in A_{m_{e_i}}$ inductively along $\mathfrak{c}_{\mtr}$ as follows:
	\begin{enumerate}
		\item Associated to the marked edge $\breve{e}$ is the element $z^{m_{\breve{e}}} = z^{\varphi(\mathsf{m})}$.

		\item With $a_{e_i}$ defined, define $a_{e_{i+1}} = g_{\ltr_{i+1}} \cdot a_{e_{i}}$.
	\end{enumerate}
	Associated to the edge $e_{out}$ is $a_{\mtr} = a_{e_{out}} \in A_{m_{\mtr}}$. We also let $\epsilon_{\mtr} = \prod_{i=1}^l \sgn(\langle - m_{e_{i-1}},n_{\ltr_{i}} \rangle)$.
	\end{definition}
	
	Definition \ref{def:marked_tree_core} applies without change to marked ribbon trees. Note that the product $\epsilon_{\mtr} a_{\mtr}$ is well-defined without the specification of a ribbon structure on $\mtr$.

Given a weighted $k$-tree $\wtr$ and $\vec{s} := (s_e)_{e\in \wtr^{[1]}} \in (\real_{<0})^{|\wtr^{[1]}|}$, the associated {\em realization} of $\wtr$ is
$
|\wtr_{\vec{s}}| := \Big( \big(\bigsqcup_{e \in \partial_{out}^{-1}(\wtr^{[0]}_{in})} (\real_{\leq 0})_e \big) \sqcup \big(\bigsqcup_{e \in \wtr^{[1]}} [s_e, 0] \big) \Big) / \sim.
$
Here $(\real_{\leq 0})_e$ is a copy of $\real_{\leq 0}$ and $\sim$ is the equivalence relation which identifies boundary points of edges if their images in $\wtr^{[0]}$ agree. For labeled (resp. marked) $k$-trees $\ltr$ (resp. $\mtr$), we allow $s_{e} = 0$ for $e \in \ltr^{[1]}$ (resp. $e \in \mtr^{[1]}$). 

\begin{definition}\label{def:tropical_curve}\label{def:P_lrtr}
	A {\em tropical disk} in $(M_\real, \mathscr{D}_{in})$ (resp. $(M_\real, \tilde{\mathscr{D}}_{in,l})$) consists of
	\begin{enumerate}
		\item
		a labeled $k$-tree $\ltr$ (resp. weighted $k$-tree $\wtr$), with labeling of $e \in \partial_{in}^{-1}(\ltr^{[0]}_{in})$ by a wall $\mathbf{w}_{i_e} =  (m_{i_e},n_{i_e}, P_{i_e}, \Theta_{i_e})$ and $m_e \in M_\sigma^+$ (resp. labeling of $e \in \partial_{in}^{-1}(\wtr^{[0]}_{in})$ by a wall $\mathbf{w}_{i_e J_{e,i_e}} =  (m_{i_e},n_{i_e}, P_{i_e J_{e,i_e}}, \Theta_{i_e J_{e,i_e}})$ and $(m_e, u^{\vec{J}_e})$), 
		\item
		a tuple of parameters $\vec{s} = (s_e)_{e\in \ltr^{[1]}} \in (\real_{ \leq 0})^{|\ltr^{[1]}|}$ (resp. $\vec{s} = (s_e)_{e\in \wtr^{[1]}} \in (\real_{<0})^{|\wtr^{[1]}|}$), and
		\item
		a proper map $\varsigma: |\ltr_{\vec{s}}| \rightarrow M_\real$ (resp. $\varsigma: |\wtr_{\vec{s}}| \rightarrow M_\real$)
	\end{enumerate}
	such that the following conditions are satisfied:
	\begin{enumerate}[(i)]
		\item
		For each $e\in \partial^{-1}_{in}(\ltr^{[0]}_{in})$ (resp. $e\in \partial^{-1}_{in}(\wtr^{[0]}_{in})$), we have $\varsigma_{\vert_{(\real_{\leq 0})_e}}(0) \in P_{i_e}$ (resp. $\varsigma_{\vert_{(\real_{\leq 0})_e}}(0) \in P_{i_e J_{e,i_e}}$) and $\varsigma_{\vert_{(\real_{\leq 0})_e}}(s) = \varsigma_{\vert_{(\real_{\leq 0})_e}}(0) + s (-m_{e})$ for all $s \in \real_{\leq 0}$.
		\item
		For each $e \in \ltr^{[1]}$ (resp. $e \in \wtr^{[1]}$), we have $\varsigma_{\vert_{[s_e,0]}}(s) = \varsigma_{\vert_{[s_e,0]}}(0) + s (-m_e)$.
	\end{enumerate}
\end{definition}

The point $\varsigma(v_{out}) := \varsigma_{\vert_{[s_{e_{out}},0]}}(0) \in M_\real$ is called the {\em stop} of the tropical disk $\varsigma$. Given a tropical disk $\varsigma$ in $(M_\real, \mathscr{D}_{in})$ (resp. $(M_\real, \tilde{\mathscr{D}}_{in,l})$), denote by $\pm (n_\varsigma,g_\varsigma)$ the pair $\pm (n_\ltr,g_\ltr)$ (resp. $\pm (n_\wtr,g_\wtr)$) associated to the underlying labeled (resp. weighted) tree.

One can also define tropical disks in $M_\real$ of type $\ltr$ without specifying a scattering diagram by relaxing condition \textit{(i)} in Definition \ref{def:tropical_curve} to
	$$\varsigma_{\vert_{(\real_{\leq 0})_e}}(s) = \varsigma_{\vert_{(\real_{\leq 0})_e}}(0) + s (-m_{e}) \text{ for all $s \in \real_{\leq 0}$}$$
	and allowing each $s_e$ to take on the value $0$.
	Tropical disks in $M_\real$ form a moduli space $\mathfrak{M}_{\ltr}(M_\real)$. Under the identification $\mathfrak{M}_{\ltr}(M_\real)\cong \real_{\leq 0}^{|\ltr^{[1]}|} \times M_\real$, the evaluation map $ev : \mathfrak{M}_{\ltr}(M_\real) \rightarrow M_\real$, obtained by taking the stop of tropical disks, is the projection to $M_\real$. Similar comments and notation apply to tropical disks in $M_\real$ of type $\wtr$.

	We denote by $\mathfrak{M}_{\ltr}(M_\real,\mathscr{D}_{in})$ (resp. $\mathfrak{M}_{\wtr}(M_\real,\tilde{\mathscr{D}}_{in,l})$) the set of all tropical disks in $(M_\real,\mathscr{D}_{in})$ (resp. $(M_\real,\tilde{\mathscr{D}}_{in,l})$) when $n_\ltr \neq 0$ (resp. $n_\wtr \neq 0$ and $u^{\vec{J}_{\wtr}} \neq 0$) with underlying labeled $k$-tree $\ltr$ (resp. weighted $k$-tree $\wtr$). By definition, $\mathfrak{M}_{\ltr}(M_\real,\mathscr{D}_{in})$ and $\mathfrak{M}_{\wtr}(M_\real,\tilde{\mathscr{D}}_{in,l})$ are subsets of $\mathfrak{M}_{\ltr}(M_\real)$. Denote their closures by an overbar. Define affine subspaces of $M_{\real}$ by $P_\ltr = ev (\overline{\mathfrak{M}}_{\ltr}(M_\real,\mathscr{D}_{in}))$ and $P_\wtr = ev (\overline{\mathfrak{M}}_{\wtr}(M_\real,\mathscr{D}_{in}))$. For marked $k$-tree $\mtr$ with $m_{\breve{e}} = \varphi(\mathsf{m})$, we define set of tropical disks $\mathfrak{M}_{\mtr}(M_\real,\mathscr{D}_{in},\mathsf{m})$ similarly (allowing $s_e = 0$ for $e \in \mtr^{[1]}$), and let $P_{\mtr} = ev (\overline{\mathfrak{M}}_{\mtr}(M_{\real},\mathscr{D}_{in},\mathsf{m}))$. 
	
	 When $M$ has rank two, $P_\wtr$ is a line when $k = 1$ and is a ray when $k > 1$. The latter case is illustrated in Figure \ref{fig:tropical_curve}.
	\begin{figure}[h]
		\centering
		\includegraphics[scale=0.4]{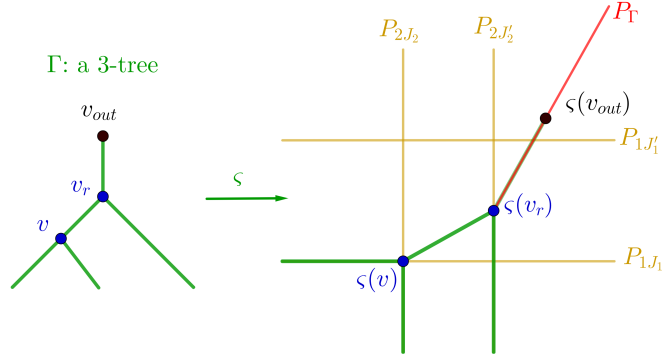}
		\caption{The affine subspace $P_{\Gamma}$ from moduli of tropical disks.}
		\label{fig:tropical_curve}
	\end{figure}

\begin{lemma}\label{lem:normal_lemma}
	If $P_\wtr$ is non-empty, then it is orthogonal to $n_{\wtr}$.
\end{lemma}

\begin{proof}
	We proceed by induction on the cardinality of $\wtr^{[0]}$. In the initial case, $\wtr^{[0]} = \emptyset$, the only tree is that with a unique edge and the statement is trivial.
	
	For the induction step, suppose that $v_r \in \wtr^{[0]}$ is adjacent to the outgoing edge $e_{out}$ and incoming edges $e_1$, $e_2$. Split $\wtr$ at $v_r$, thereby obtaining trees $\wtr_1$ and $\wtr_2$ with outgoing edges $e_1$ and $e_2$ and $k_1$ and $k_2$ incoming edges, respectively. We have
	$$
	\left( \overline{\mathfrak{M}}_{\wtr_1}(M_\real,\tilde{\mathscr{D}}_{in,l}) {}_{ev}\times_{ev} \overline{\mathfrak{M}}_{\wtr_2} (M_\real,\tilde{\mathscr{D}}_{in,l}) \right) \times \real_{\geq 0} \cdot (-m_{\wtr})
\cong\overline{\mathfrak{M}}_{\wtr}(M_\real,\tilde{\mathscr{D}}_{in,l}),
	$$
	implying that $
	P_\wtr = (P_{\wtr_1} \cap P_{\wtr_2}) + \real_{\geq 0} \cdot (-m_{\wtr}).$
	By the induction hypothesis, $n_{\wtr_i}$ is orthogonal to $P_{\wtr_i}$, $i=1,2$, and hence $n_\wtr$ is orthogonal to $P_{\wtr_1} \cap P_{\wtr_2}$. A direct computation using the definition of $n_\wtr$ shows that $\langle m_\wtr,n_\wtr \rangle= 0$. The lemma follows.
\end{proof}

\begin{definition}\label{def:tropical_generic_assumption}
	A scattering diagram $\tilde{\mathscr{D}}_{in,l}$ is called {\em generic} if for any weighted trees $\wtr_1$, $\wtr_2$ such that $u^{\vec{J}_{\wtr_1}} \cdot u^{\vec{J}_{\wtr_2}} \neq 0$ and $P_{\wtr_1}$ intersects $P_{\wtr_2}$ transversally,\footnote{Here `tranversally' means that the unique affine subspaces containing $P_{\wtr_1}$ and $P_{\wtr_2}$ intersect transversally.} the intersection $P_{\wtr_1} \cap P_{\wtr_2} \subset M_{\real}$ has codimension two and is contained in the boundary of neither $P_{\wtr_1}$ nor $P_{\wtr_2}$.
\end{definition}

The next result, which was proved by various authors in increasing levels of generality, relates consistent scattering diagrams to the counting of tropical disks.

\begin{theorem}[\cite{gross2010tropical, filippini2015block, mandel2015refined}]\label{thm:GPS_theorem}
	Let $\tilde{\mathscr{D}}_{in,l}$ be a generic initial scattering diagram. There is a bijective correspondence between walls $\mathbf{w} \in \mathcal{S}(\tilde{\mathscr{D}}_{in,l})$ and weighted trees $\wtr$ with $\mathfrak{M}_{\wtr}(M_\real,\tilde{\mathscr{D}}_{in,l}) \neq \emptyset$ under which a wall $\mathbf{w} = (m,n, P, \Theta)$ corresponds to the weighted tree $\wtr$ with $(n_{\wtr},P_{\wtr}) = (n,P)$ and $\log(\Theta) = \left( \prod_{e \in \partial^{-1}_{in}(\wtr^{[0]}_{in})} (\# J_{e,i_e})! \right) g_{\wtr}  u^{\vec{J}_{\wtr}}$.
\end{theorem}

\section{Pertubative solution of the Maurer--Cartan equation}\label{sec:MC_equation}

We introduce a differential graded (dg) Lie algebra whose Maurer--Cartan equation governs the scattering process from $\mathscr{D}_{in}$ to $\mathcal{S}(\mathscr{D}_{in})$, or its generic perturbation.

\subsection{Differential forms with asymptotic support}\label{sec:asymptotic_support}

We begin by recalling some background material from \cite[\S 4.2.3]{kwchan-leung-ma} and \cite[\S 3.2]{kwchan-ma-p2}.
	
	Let $U$ be a convex open subset of $M_\real$, or more generally, of an integral affine manifold, as in \cite[\S 3.2]{kwchan-ma-p2}. Introduce the notation $\Omega^k_\hp(U) := \Gamma(U \times \mathbb{R}_{>0}, \bigwedge^{\raisebox{-0.4ex}{\scriptsize $k$}} T^{\vee} U)$, where the coordinate of $\mathbb{R}_{>0}$ is $\hp$. Let $\mathcal{W}^{-\infty}_k(U) \subset \Omega^k_\hp(U)$ be the set of $k$-forms $\alpha$ such that, for each $q \in U$, there exists a neighborhood $q \in V \subset U$ and constants $D_{j,V}$, $c_V$ such that $\|\nabla^j \alpha\|_{L^\infty(V)} \leq D_{j,V} e^{-c_V/\hp}$ for all $j \geq 0$. Similarly, let $\mathcal{W}^{\infty}_k(U) \subset \Omega^k_\hp(U)$ be the set of $k$-forms $\alpha$ such that, for each $q \in U$, there exists a neighborhood $q \in V \subset U$ and constants $D_{j,V}$ and $N_{j,V} \in \inte_{>0}$ such that $\|\nabla^j \alpha\|_{L^\infty(V)} \leq D_{j,V} \hp^{-N_{j,V}}$ for all $j \geq 0$. The assignment $U \mapsto \mathcal{W}^{-\infty}_k(U)$ (resp. $U \mapsto \mathcal{W}^{\infty}_k(U)$) defines a sheaf $\mathcal{W}^{-\infty}_k$ (resp. $\mathcal{W}^{\infty}_k$) on $M_\real$. Note that $\mathcal{W}^{-\infty}_k$ and $\mathcal{W}^{\infty}_k$ are closed under the wedge product,  $\nabla_{\dd{x}}$ and the de Rham differential $d$. Since $\mathcal{W}^{-\infty}_k$ is a dg ideal of $\mathcal{W}^{\infty}_k$, the quotient $\mathcal{W}^{\infty}_*/\mathcal{W}^{-\infty}_*$ is a sheaf of dg algebras when equipped with the de Rham differential.

By a {\em tropical polyhedral subset} of $U$ we mean a connected convex subset which is defined by finitely many affine equations or inequalities over $\mathbb{Q}$.

\begin{definition}\label{def:asypmtotic_support_pre}
	A $k$-form $\alpha \in \mathcal{W}^{\infty}_k(U)$ is said to have {\em asymptotic support on a closed codimension $k$ tropical polyhedral subset $P \subset U$ with weight $s$}, denoted $\alpha \in \mathcal{W}_{P}^s(U)$, if the following conditions are satisfied:
	\begin{enumerate}
		\item
		For any $p \in U \setminus P$, there is a neighborhood $p \in V \subset U \setminus P$ such that $\alpha|_V \in \mathcal{W}^{-\infty}_k(V)$.
		
		\item
		There exists a neighborhood $W_P \subset U$ of $P$ such that $\alpha =  h(x,\hp) \nu_P + \eta$ on $W_P$, where $\nu_P \in \bigwedge^k N_{\real}$ is the unique affine $k$-form which is normal to $P$, $h(x,\hp) \in C^\infty(W_P \times \real_{>0})$ and $\eta \in \mathcal{W}^{-\infty}_k(W_P)$.
		
		\item
		For any $p \in P$, there exists a convex neighborhood $p \in V \subset U$ equipped with an affine coordinate system $x = (x_1,\dots, x_n)$ such that $x' := (x_1, \dots, x_k)$ parametrizes codimension $k$ affine linear subspaces of $V$ parallel to $P$, with $x' = 0$ corresponding to the subspace containing $P$. With the foliation $\{(P_{V, x'})\}_{x' \in N_V}$, where $P_{V,x'} = \{ (x_1,\dots,x_n) \in V  \ | \ (x_1,\dots,x_k) = x' \}$ and $N_V$ is the normal bundle of $V$, we require that, for all $j \in \inte_{\geq 0}$ and multi-indices $\beta = (\beta_1,\dots,\beta_k) \in \inte_{\geq 0}^k$, the estimate
		\[
		\int_{x'}  (x')^\beta \left(\sup_{P_{V,x'}}|\nabla^j (\iota_{\nu_P^\vee} \alpha)| \right) \nu_P \leq D_{j,V,\beta} \hp^{-\frac{j+s-|\beta|-k}{2}}
		\]
		holds for some constant $D_{j,V,\beta}$ and $s \in \inte$, where $|\beta| = \sum_l \beta_l$ and $\nu_P^\vee = \dd{x_1}\wedge\cdots \wedge\dd{x_k}$.
	\end{enumerate}
\end{definition}

Observe that $\nabla_{\dd{x_l}} \mathcal{W}^s_P(U) \subset \mathcal{W}^{s + 1}_P(U)$ and $(x')^{\beta}\mathcal{W}^s_P(U) \subset \mathcal{W}^{s-|\beta|}_P(U)$. It follows that
\[
(x')^{\beta} \nabla_{\dd{x_{l_1}}}\cdots \nabla_{\dd{x_{l_j}}} \mathcal{W}^s_P(U) \subset \mathcal{W}^{s+j-|\beta|}_P(U).
\]

The weight $s$ defines a filtration of $\mathcal{W}^{\infty}_k$ (we drop the $U$ dependence from the notation whenever it is clear from the context):\footnote{Note that $k$ is equal to the codimension of $P \subset U$.}
\[
\mathcal{W}^{-\infty}_k \subset \cdots \subset \mathcal{W}^{-1}_P\subset \mathcal{W}^0_P \subset \mathcal{W}^1_P \subset \cdots \subset \mathcal{W}^{\infty}_k \subset \Omega^k_\hp(U).
\]
This filtration, which keeps track of the polynomial order of $\hp$ for $k$-forms with asymptotic support on $P$, provides a convenient tool to express and prove results in asymptotic analysis.

\begin{definition}\label{def:asymptotic_support}\label{def:asy_support_algebra}
	A differential $k$-form $\alpha$ is in $\tilde{\mathcal{W}}^{s}_k(U)$ if there exist polyhedral subsets $P_1, \dots, P_l \subset U$ of codimension $k$ such that $\alpha \in \sum_{j=1}^l \mathcal{W}^s_{P_j}(U)$.
	If, moreover, $d \alpha \in \tilde{\mathcal{W}}^{s+1}_{k+1}(U)$, then we write $\alpha \in \mathcal{W}^s_k(U)$. For every $s \in \inte$, let $\mathcal{W}^s_*(U) = \bigoplus_k \mathcal{W}^{s+k}_k(U)$.
\end{definition}

We say that closed tropical polyhedral subsets $P_1, P_2 \subset U$ of codimension $k_1, k_2$ {\em intersect transversally} if the affine subspaces of codimension $k_1$ and $k_2$ which contain $P_1$ and $P_2$, respectively, intersect transversally. This definition applies also when $\partial P_i \neq \emptyset$.

\begin{lemma}[{\cite[Lemma 3.11]{kwchan-ma-p2}}]
\phantomsection
\label{lem:support_product}\label{prop:support_product}\label{prop:asy_support_algebra_ideal}
\begin{enumerate}
\item Let $P_1, P_2, P \subset U$ be closed tropical polyhedral subsets of codimension $k_1$, $k_2$ and $k_1+k_2$, respectively, such that $P$ contains $P_1 \cap P_2$ and is normal to $\nu_{P_1} \wedge \nu_{P_2}$. Then $\mathcal{W}^s_{P_1}(U) \wedge \mathcal{W}^r_{P_2}(U) \subset \mathcal{W}^{r+s}_P(U)$ if $P_1$ and $P_2$ intersect transversally and $\mathcal{W}^s_{P_1}(U) \wedge \mathcal{W}^r_{P_2}(U) \subset \mathcal{W}^{-\infty}_{k_1 + k_2}(U)$ otherwise.

\item We have $\mathcal{W}^{s_1}_{k_1}(U) \wedge \mathcal{W}^{s_2}_{k_2}(U) \subset \mathcal{W}^{s_1+s_2}_{k_1+k_2}(U)$. In particular, $\mathcal{W}^0_*(U) \subset \mathcal{W}^{\infty}_*(U)$ is a dg subalgebra and $\mathcal{W}^{-1}_*(U) \subset \mathcal{W}^0_*(U)$ is a dg ideal.
\end{enumerate}
\end{lemma}

\subsubsection{Homotopy operators}\label{sec:integral_lemma}
Let $P \subset U$ be a closed tropical polyhedral subset. In the remainder of this section, we study the behavior of $\mathcal{W}^s_P(U)$ under the application of a homotopy-type operator $I$. To do so, fix a reference tropical hyperplane $R \subset U$ which divides $U$ into $U\setminus R = U_+ \sqcup U_-$. Fix also an affine vector field $v$ (meaning $\nabla v = 0$) which is not tangent to $R$ and points into $U_+$.

By shrinking $U$ if necessary, we can assume that, for any $p \in U$, the unique flow line of $v$ in $U$ passing through $p$ intersects $R$ at a unique point, say $x \in R$.
The time $t$ flow along $v$ then defines a diffeomorphism
$\tau : W \rightarrow U,\ (t, x) \mapsto \tau(t,x),$
where $W \subset \real \times R$ is the maximal domain of definition of $\tau$.
For each $x \in R$, set $\tau_x(t) = \tau(t,x)$. 
Let $P_\pm = P \cap \overline{U}_{\pm}$ and define
\[
I(P)_+ =( P_+ + \real_{\geq 0} \cdot v ) \cap U, \quad
I(P)_-  = (P_- + \real_{\leq 0 } \cdot v) \cap U.
\]

Define an integral operator $I$ by
\[
I(\alpha)(t,x) = \int_{0}^t \iota_{\dd{s}}(\tau^*(\alpha))(s,x) ds, \qquad \alpha \in \mathcal{W}^s_P(U).
\]
Despite the notation, $I$ depends on the choice of the tropical hyperplane $R$ and the vector field $v$.

\begin{lemma}[{cf. \cite[Lemmas 3.12, 3.15]{kwchan-ma-p2}}]\label{lem:integral_lemma}\label{lem:integral_lemma_modified}
	Let $\alpha \in \mathcal{W}^s_P(U)$. Then $I(\alpha) \in \mathcal{W}^{-\infty}_{k-1}(U)$ if $v$ is tangent to $P$ and $I(\alpha) \in \mathcal{W}^{s-1}_{I(P)_+}(U) + \mathcal{W}^{s-1}_{I(P)_-}(U)$ otherwise. Moveover, if $\alpha \in \tilde{\mathcal{W}}^s_k(U)$ (resp. $\alpha \in \mathcal{W}^s_k(U)$), then $I(\alpha) \in \tilde{\mathcal{W}}^{s-1}_{k-1}(U)$ (resp. $I(\alpha) \in \mathcal{W}^{s-1}_{k-1}(U)$).
\end{lemma}

Using the affine coordinates determined by $\tau$, define a tropical hypersurface $\mathbf{i}: R \rightarrow U$, $x \mapsto (0,x)$, and an affine projection $\mathbf{p} : U \rightarrow R$, $(t,x) \mapsto x$.

\begin{lemma}[{cf. \cite[Lemmas 3.13, 3.14]{kwchan-ma-p2}}]
\phantomsection
\label{lem:inclusion_projection_lemma}
\begin{enumerate}
\item Let $k = \textnormal{codim}_\real(P \subset U)$. For $\alpha \in \mathcal{W}^s_P(U)$, we have $\mathbf{i}^* \alpha \in \mathcal{W}^s_{Q}(R)$ if $P$ intersects $R$ transversally, where $Q \subset R$ is any codimension $k$ polyhedral subset which contains $P \cap R$ and is normal to $\mathbf{i}^*\nu_P$, and $\mathbf{i}^*\alpha \in \mathcal{W}^{-\infty}_k(R)$ otherwise. Moreover, the pullback along $\mathbf{i}$ is a map $\mathbf{i}^* : \mathcal{W}^s_k(U) \rightarrow \mathcal{W}^s_k(R)$.

\item For $\alpha \in \mathcal{W}^s_{P}(R)$, we have $\mathbf{p}^* \alpha \in \mathcal{W}^s_{\mathbf{p}^{-1}(P)}(U)$. Moreover, the pullback along $\mathbf{p}$ is a map $\mathbf{p}^* : \mathcal{W}^s_k(R) \rightarrow \mathcal{W}^s_k(U)$.
\end{enumerate}
\end{lemma}

Finally, we extend the above construction to define an integral operator which retracts $U$ to a chosen point $q_0$. Consider a chain of affine subspaces $\{q_0\}=U_0 \subseteq U_1 \subseteq \cdots \subseteq U_r = U$ with $\dim_\real(U_j) = j$. Denote by $\mathbf{i}_j : U_j \rightarrow U_{j+1}$ and $\mathbf{p}_j : U_{j+1} \rightarrow U_j$ the inclusions and affine projections, respectively. Let $v_j$ be a constant affine vector field on $U_{j+1}$ which is tangent to the fiber of $\mathbf{p}_j$. Composition of the inclusion operators gives $\mathbf{i}_{i,j} : U_i \rightarrow U_j$, $i <j$, and similarly for the projection operators. Let $I_j : \mathcal{W}^s_{k}(U_{j+1}) \rightarrow \mathcal{W}^{s-1}_{k-1}(U_{j+1})$ be the integral operator defined using $v_j$, as above. For the purpose of solving the Maurer--Cartan equation, we will choose $q_0$ to be an irrational point in $U_1$. While $\{q_0\}$ is not a tropical polyhedral subset of $U_1$, the definition of $\mathbf{p}_{0,j}^*$ remains valid if it is treated as the inclusion of constant functions. The operator $I_0$ defines a map $\mathcal{W}^s_1(U_1) \rightarrow \mathcal{W}^{s-1}_{0}(U_1)$, even if $q_0$ is irrational. Indeed, each $\alpha \in \mathcal{W}^s_{1}(U_1)$ can be written as a finite sum $\sum_{l} \alpha_l$ with $\alpha_l \in \tilde{\mathcal{W}}^s_{P_l}(U_1)$ for some rational points $P_l$ of $U_1$ which, in particular, are distinct from $q_0$. It follows that $I_0(P_l)$ is still a tropical subspace of $U_1$.

With the above notation, define $I:\mathcal{W}^s_*(U) \rightarrow \mathcal{W}^{s-1}_{*-1}(U)$ by
\begin{equation}\label{eqn:generalized_integral_operator}
I = \mathbf{p}_{1,r}^* I_{0}\mathbf{i}_{1,r}^* + \dots + \mathbf{p}_{r-1,r}^* I_{r-2}\mathbf{i}_{r-1,r}^* + I_{r-1}.
\end{equation}
Write $\mathbf{i}^*:=\mathbf{i}_{0,r}^*$ for evaluation at $q_0$ and $\mathbf{p}^* : =  \mathbf{p}_{0,r}^*$.

\begin{prop}[{cf. \cite[Lemma 3.16]{kwchan-ma-p2}}]\label{prop:homotopy_operator_identity}
	The equality
	$
	d I + I d= \Id - \mathbf{p}^* \mathbf{i}^*
	$
	holds.
\end{prop}

\subsection{The tropical differential graded Lie algebra}\label{sec:tropical_dgla}

\subsubsection{Abstract tropical dg Lie algebras}

Consider a $M_\sigma^+$-graded tropical Lie algebra $\mathfrak{h}$ acting on a $P$-graded algebra $A$, as in Section \ref{sec:scattering_diagram}. Fix a convex open subset $U \subset M_\real$.

\begin{definition}\label{def:abstract_tropical_dgla}
	The {\em tropical dg Lie algebra associated to $\mathfrak{h}$} is 
	$$
	\mathcal{H}^*(U) := \bigoplus_{m \in M_\sigma^+} \left(\mathcal{W}^{0}_*(U)/ \mathcal{W}^{-1}_*(U)\right) \otimes_\comp \mathfrak{h}_{m}
	$$
	with differential $d (\alpha  h) = (d\alpha ) h$ and Lie bracket $
	[\alpha h,\alpha^{\prime}  h'] = ( \alpha \wedge \alpha^{\prime}) [h,h']$, where $\alpha, \alpha^{\prime} \in \mathcal{W}^{0}_*(U)/ \mathcal{W}^{-1}_*(U)$ and $h, h^{\prime} \in \mathfrak{h}$.
\end{definition}

Denote by $\hat{\mathcal{H}}^*(U)$ the completion associated to the monoid ideals $k M_\sigma^+ \subset M_\sigma^+$, $k \in \mathbb{Z}_{>0}$. Given a commutative algebra $R$, set $\mathcal{H}^*_R(U) = \mathcal{H}^*(U) \otimes_{\mathbb{C}} R$ and $\hat{\mathcal{H}}^*_{R}(U) = \hat{\mathcal{H}}^*(U) \hat{\otimes}_{\mathbb{C}} R$. We also introduce the dg Lie algebra $\mathcal{G}^{*}(U) := \bigoplus_{m \in M_\sigma^+} \mathcal{W}^0_*(U) \otimes_{\mathbb{C}} \mathfrak{h}_m$ and its dg Lie ideal
	$
	\mathcal{I}^{*}(U) := \bigoplus_{m \in M_\sigma^+} \mathcal{W}^{-1}_*(U) \otimes_{\mathbb{C}} \mathfrak{h}_m.
	$
	Observe that $\mathcal{G}^*(U)/\mathcal{I}^*(U) \simeq \mathcal{H}^*(U)$. When $U = M_\real$, we will often omit $U$ from the notation.

We will be interested in solving the Maurer--Cartan equation in $\hat{\mathcal{H}}^*$ or $\mathcal{H}^*_R$, which reads
	\begin{equation}\label{eqn:MC_equation}
	d \varphi + \half[\varphi,\varphi] = 0.
	\end{equation}

\begin{definition}\label{def:tropical_dga}
	The {\em tropical dg algebra associated to $A$} is
	$
	\mathcal{A}^*(U) := \bigoplus_{m \in P} \left(\mathcal{W}^{0}_*(U)/ \mathcal{W}^{-1}_*(U)\right) \otimes_\comp A_{m},
	$
	with differential $d (\alpha  f) = (d\alpha ) f$ and product $(\alpha f)\wedge (\alpha^{\prime} f^{\prime}) = (\alpha \wedge \alpha^{\prime}) (f f^{\prime})$, where $\alpha, \alpha^{\prime} \in \mathcal{W}^{0}_*(U)/ \mathcal{W}^{-1}_*(U)$ and $f, f^{\prime} \in A$. As for $\mathcal{H}^*(U)$, we can also define $\hat{\mathcal{A}}^*(U)$ and $\mathcal{A}_R(U)$.
	\end{definition}
	 
	 There is a left $\mathcal{H}^*(U)$-action on $\mathcal{A}^*(U)$ given by $
	 (\alpha h)\cdot (\alpha^{\prime} f):= (\alpha \wedge \alpha^{\prime}) (h \cdot f)$. The square zero extension dg Lie algebra $(\mathcal{H} \oplus \mathcal{A}[1])^*(U) := \mathcal{H}^*(U) \oplus \mathcal{A}^{*}(U)[-1]$ has the bracket
	 \[
	 [h+f,h^{\prime}+f^{\prime}] = [h,h^{\prime}] + h \cdot f^{\prime} -(-1)^{|h^{\prime}| |f|} h^{\prime} \cdot f, \qquad h, h^{\prime} \in \mathcal{H}^*(U), \;\; f, f^{\prime} \in \mathcal{A}^*(U)[-1].
	 \]

\subsubsection{Homotopy operator}\label{sec:homotopy_operator}

We will solve equation \eqref{eqn:MC_equation} using Kuranishi's method \cite{Kuranishi65}, in which a solution is written as a sum over trivalent trees. We take $U=M_\real$ for the remainder of this section.

	
Fix an affine metric $g_0$ on $M_\real$. For each $m \in M_\sigma^+$, fix a chain of affine subspaces $\{pt\}=U^m_{0} \subseteq U_{1}^m \subseteq \dots \subseteq  U^m_{r} = M_\real$. We assume that $U^m_{0}$ is an irrational point of $U^{m}_1$. Denote by $\mathbf{p}^m_{j}$ the affine projection determined by the vector field $v^m_{j}$, with the convention that $v^m_{1} = -m$.

	Given these choices, we obtain a homotopy operator $\mathsf{H}_m : \mathcal{W}^{0}_* \rightarrow \mathcal{W}^0_{*-1}$ using equation \eqref{eqn:generalized_integral_operator} (denoted there by $I$). Let $\mathsf{P}_m : \mathcal{W}^0_*\rightarrow \mathcal{W}^0_0(U^m_0) $ be the projection $\mathsf{P}_m (\alpha ) := \alpha|_{U^m_{0}}$ and let $\iota_m :\mathcal{W}^0_0(U^m_0) \rightarrow \mathcal{W}^0_*$ be given by $\iota_m (\alpha) := \alpha$, the embedding of constant functions on $M_\real$. As in \cite{kwchan-leung-ma}, these operators satisfy
	\[
	d \mathsf{H}_m + \mathsf{H}_m d= \Id - \iota_m \mathsf{P}_m,
	\]
	so that $\mathsf{H}_* = \bigoplus_m \mathsf{H}_m$ is a homotopy retracting $\mathcal{W}^0_*$ to its cohomology $H^*(\mathcal{W}^0_*,d) \simeq \mathcal{W}^0_0(U^m_0)$. Moreover, these operators descend to the quotient $\mathcal{W}^0_*/\mathcal{W}^{-1}_*$, thereby contracting its cohomology to $\comp \cong \mathcal{W}^0_0(U^m_0)/\mathcal{W}^{-1}_0(U^m_0)$.

\begin{definition}\label{MC_homotopy}
\phantomsection
\label{pathspacehomotopy}
\begin{enumerate}
	\item For each $m \in M_\sigma^+$, let $\mathcal{H}^*_m:= \big(\mathcal{W}^0_*/\mathcal{W}^{-1}_*\big) \otimes_{\mathbb{C}} \mathfrak{h}_m$ and define the homotopy operator $\mathsf{H}_m : \mathcal{H}^{*+1}_m \rightarrow \mathcal{H}^*_m$ by $\mathsf{H}_m(\alpha h) = \mathsf{H}_m(\alpha) h$. Denote by $\mathsf{H} = \bigoplus_m \mathsf{H}_m$ the induced operator on $\mathcal{H}^*$.
	\item Define operators $\mathsf{P} = \bigoplus_m \mathsf{P}_m$ and $\iota = \bigoplus_m \iota_m$ similarly.	
\end{enumerate}
\end{definition}

Taking inverse limits defines operators on $\hat{\mathcal{H}}^*$. Similar definitions apply to $\mathcal{H}^*_R$ and $\hat{\mathcal{H}}^*_R$. We can also apply the construction to the dg Lie algebra $\mathcal{H} \oplus \mathcal{A}[-1]$, obtaining operators $\mathsf{H}$, $\mathsf{P}$ and $\iota$.

\subsection{Solving the Maurer--Cartan equation}\label{sec:abstract_solving_equation}

\subsubsection{Input of the Maurer--Cartan equation}\label{sec:input_to_MC}

Consider an initial scattering diagram $\mathscr{D}_{in}$, or its perturbation $\tilde{\mathscr{D}}_{in,l}$. We will associate to each wall a term in $\hat{\mathcal{H}}^1$ or $\hat{\mathcal{H}}^1_{\tilde{R}_l}$ to serve as inputs to solve the Maurer--Cartan equation. 

Consider first $\mathscr{D}_{in} = \{ \mathbf{w}_i = (m_i, n_i, P_i,\Theta_i)\}_{i \in I}$, with $\log(\Theta_i) = \sum_{j} g_{j i}$ as in equation \eqref{eqn:initial_walls}. Consider an affine function $\eta_i = \langle \cdot, n_i \rangle + c$ such that $P_i = \{ x \in M_{\mathbb{R}} \mid \eta_i(x)= 0 \}$ and set
\[
\delta_{P_i} = \big(\frac{1}{\pi \hp}\big)^{1/2} e^{-(\eta_i^2)/\hp}d\eta_i.
\]

\begin{lemma}[{\cite[\S 4]{kwchan-leung-ma}}]\label{lem:input_delta}
	We have $\delta_{P_i} \in \mathcal{W}^{1}_{P_i}(M_\real)$. 
\end{lemma}

Set
\begin{equation}\label{eqn:incoming_for_each_wall}
\incoming^{(i)} = -\delta_{P_i} \log(\Theta_i) \in \hat{\mathcal{H}}^1
\end{equation}
and take $\incoming = \sum_{i \in I} \incoming^{(i)}$ as the input to solve the Maurer--Cartan equation.

If instead we begin with a perturbed diagram $\tilde{\mathscr{D}}_{in,l}$, then we have walls $\mathbf{w}_{iJ} = (m_{iJ},n_{iJ},P_{iJ},\Theta_{iJ})$, leading to $\delta_{P_{iJ}} \in \mathcal{W}^{1}_{P_i{iJ}}(M_\real)$ and $\tilde{\incoming} = \sum_{i,J} \tilde{\incoming}^{(i)}_J \in \mathcal{H}_{\tilde{R}_l}^1$.

\subsubsection{Summation over trees}\label{sec:sum_over_tree_formula}

Motivated by Kuranishi's method \cite{Kuranishi65} of solving the Maurer--Cartan equation of the Kodaira--Spencer dg Lie algebra, and its generalization to general dg Lie algebras (see \cite{manetti2005differential}), instead of solving equation \eqref{eqn:MC_equation}, we first look for solutions $\breve{\Phi} \in \hat{\mathcal{H}}^1$ of the equation
\begin{equation}\label{eqn:pseudoMC}
\breve{\Phi}= \incoming - \half \mathsf{H} [\breve{\Phi},\breve{\Phi}].
\end{equation}
In the perturbed setting, we look for solutions $\tilde{\breve{\Phi}} \in \mathcal{H}_{\tilde{R}_l}^1$ of the equation $\tilde{\breve{\Phi}}= \tilde{\incoming} - \half \mathsf{H} [\tilde{\breve{\Phi}},\tilde{\breve{\Phi}}]$.

\begin{prop}\label{prop:pseudoMC_to_MC}
	If $\breve{\Phi}$ satisfies equation \eqref{eqn:pseudoMC}, then $\breve{\Phi}$ satisfies equation \eqref{eqn:MC_equation} if and only if $\mathsf{P}[\breve{\Phi},\breve{\Phi}] = 0$. An analgous statement holds for $\tilde{\breve{\Phi}}$.
\end{prop}

The unique solution $\breve{\Phi}$ of equation \eqref{eqn:pseudoMC} can be expressed as a sum over directed trees, as we now recall. An analogous statement holds for $\tilde{\breve{\Phi}}$. Further details can be found in \cite[\S 5.1]{kwchan-leung-ma}.

	\begin{definition}\label{def:label_tree_operation}\label{def:weighted_tree_operation}
		Given $\lrtr \in \lrtree{k}$ (resp. $\wrtr \in \wrtree{k}$), number the incoming vertices by $v_1,\dots,v_k$ according to their cyclic ordering and let $e_1,\dots,e_k$ be their incoming edges. Define $\mathfrak{l}_{k,\lrtr}: (\hat{\mathcal{H}}^{*+1})^{\otimes k} \rightarrow \hat{\mathcal{H}}^{*+1}$ (resp. $\mathfrak{l}_{k,\wrtr}: (\mathcal{H}_{\tilde{R}_l}^{*+1})^{\otimes k} \rightarrow \mathcal{H}_{\tilde{R}_l}^{*+1}$) so that its value on $\zeta_1, \dots, \zeta_k \in \hat{\mathcal{H}}_{(\tilde{R}_l)}^{*+1}$ is given by
		\begin{enumerate}
			\item
			extracting the component of $\zeta_i$ in $\mathcal{H}^{*+1}_{m_{e_i}}$ (resp. $\mathcal{H}^{*+1}_{m_{e_i}} u^{J_{e_i}}$) and aligning it as the input at $v_i$,
			
			\item then applying $m_2$ at each vertex in $\lrtr^{[0]}$ (resp. $\wrtr^{[0]}$), where $m_2: \hat{\mathcal{H}}^{*+1} \otimes \hat{\mathcal{H}}^{*+1} \to \hat{\mathcal{H}}^{*+1}$ is the graded symmetric operator $m_2(\alpha,\beta) = (-1)^{\bar{\alpha}(\bar{\beta}+1)}[\alpha,\beta]$, where $\bar{\alpha}$ and $\bar{\beta}$ denote the degrees of $\alpha$ and $\beta$, and finally
			
			\item applying the homotopy operator $-\mathsf{H}$ to each edge in $\lrtr^{[1]}$ (resp. $\wrtr^{[1]}$).
		\end{enumerate}
	\end{definition}
	
	Having defined $\mathfrak{l}_{k,\lrtr}$ and $\mathfrak{l}_{k,\wrtr}$, we can write
	\begin{equation}\label{eqn:weighted_tree_MC_solution}
	\breve{\Phi} = \sum_{k \geq 1} \frac{1}{2^{k-1}}\sum_{\lrtr \in \lrtree{k}} \mathfrak{l}_{k,\lrtr}(\incoming,\dots,\incoming),\qquad \tilde{\breve{\Phi}} = \sum_{k \geq 1} \frac{1}{2^{k-1}}\sum_{\wrtr \in \wrtree{k}} \mathfrak{l}_{k,\wrtr}(\tilde{\incoming},\dots,\tilde{\incoming}).
	\end{equation}
It is not hard to see that the sum defining $\breve{\Phi}$ converges in $\hat{\mathcal{H}}^*$. The sum defining $\tilde{\breve{\Phi}}$ is finite in $\mathcal{H}_{\tilde{R}_l}^*$ because the maximal ideal of $\tilde{R}_l$ is nilpotent.

\section{Tropical counting and theta functions from Maurer--Cartan solutions}\label{sec:main_theorem}

\subsection{Tropical counting from Maurer--Cartan solutions}\label{sec:tropical_counting_thm} 
The goal of this section is to relate Maurer--Cartan elements of $ \mathcal{H}^1_{\tilde{R}_l}$ to the counting of tropical disks in $(M_\real,\tilde{\mathscr{D}}_{in,l})$. Similar results for $M$ of rank two can be found in \cite{kwchan-ma-p2}.

\subsubsection{A partial homotopy operator}\label{sec:modified_homotopy_operator}

Recall the homotopy operator $\mathsf{H} = \bigoplus_{m \in M_\sigma^+} \mathsf{H}_m$ from Section \ref{sec:homotopy_operator}. It will be useful to replace $\mathsf{H}_m$ with the partial homotopy operator
\begin{equation}\label{eqn:modified_homotopy_operator}
\mathbf{H}_m(\alpha)(x):= \int_{-\infty}^0 (\iota_{\dd{s}}(\tau^m)^*(\alpha)(s,x)) ds,
\end{equation}
where $\tau^m : \real \times M_\real \rightarrow M_\real$ is the flow with respect to the vector field $-m$. Given $\lrtr \in \lrtree{k}$ (resp. $\wrtr \in \wrtree{k}$), denote by $\mathbf{L}_{k,\lrtr}$ (resp. $\mathbf{L}_{k,\wrtr}$) the operation obtained by replacing the operator $\mathsf{H}$ with $\mathbf{H}$ in Definition \ref{def:label_tree_operation} (denoted there by $\mathfrak{l}$ instead of $\mathbf{L}$).

The reason for introducing $\mathbf{H}$, $\mathbf{L}_{k,\lrtr}$ and $\mathbf{L}_{k,\wrtr}$ is that the operator $\mathsf{H}$ depends on the choice of a chain of affine subspaces $U^{m}_{\bullet}$ for each $m \in M_{\sigma}^+$. A drawback of the $U^m_{\bullet}$-independent $\mathbf{H}$ is that $\mathbf{H}_m(\alpha)$ is defined only when $\alpha$ is suitably behaved at infinity; see Lemma \ref{lem:integral_convergent}. For this reason, additional arguments are required to verify that the analogues of equation \eqref{eqn:weighted_tree_MC_solution} are well-defined and in fact solve the Maurer--Cartan equation; see Lemma \ref{lem:well_definedness_solution}. An alternative way of proceeding, taken in \cite{kwchan-leung-ma, kwchan-ma-p2}, is to make a careful choice of $U^m_{\bullet}$ so as to directly relate the Maurer--Cartan solution \eqref{eqn:weighted_tree_MC_solution} with scattering diagrams.


\subsubsection{Modified Maurer--Cartan solutions}\label{sec:modified_MC_solution}
In this section we prove that $\mathbf{L}_{k,\wrtr}(\tilde{\incoming},\dots,\tilde{\incoming})$ is well-defined; the proof for $\mathbf{L}_{k,\lrtr}(\incoming,\dots,\incoming)$ is similar. 

Given a weighted ribbon $k$-tree $\wrtr$, denote by $\mathfrak{M}_{\wrtr}(M_\real) \cong \real_{\leq 0}^{|\wrtr^{[1]}|} \times M_\real$ the space of tropical disks in $M_\real$ for the underlying weighted tree.

\begin{definition}\label{def:sequence_flow_def}
	Given a directed path $\mathfrak{e} = (e_0, \ldots, e_{l})$ in $\wrtr$, considered as a sequence of edges, define a map $\tau^{\mathfrak{e}}: \real_{\leq 0}^{|\wrtr^{[1]}_{\mathfrak{e}}|} \times M_\real \rightarrow M_\real$
	by
	$
	\tau^{\mathfrak{e}} (\vec{s},x) = \tau_{s_0}^{e_0} \circ \cdots \circ \tau_{s_l}^{e_l} (x),
	$
	where $\tau^{e_j} = \tau^{m_{e_j}}$ and $\wrtr^{[1]}_{\mathfrak{e}}$ is the subset $\{e_0, \ldots, e_l\} \subset \wrtr^{[1]}$. The map $\tau^{\mathfrak{e}}$ extends to a map
	$
	\hat{\tau}^\mathfrak{e}: \mathfrak{M}_{\wrtr}(M_\real) \cong \real_{\leq 0}^{|\wrtr^{[1]}|} \times M_\real  \rightarrow \real_{\leq 0}^{|\wrtr^{[1]} \setminus \wrtr^{[1]}_{\mathfrak{e}}|} \times M_\real
	$
	by taking the Cartesian product with $\real_{\leq 0}^{|\wrtr^{[1]} \setminus \wrtr^{[1]}_{\mathfrak{e}}|}$.
\end{definition}

This definition does not use the ribbon structure of $\wrtr$ and so also applies to a weighted $k$-trees.

Recall that the differential form $\delta_{P_i}$ depends on an affine function $\eta_i$ which vanishes on $P_i$. Let $N_i$ be the space of leaves obtained by parallel translation of $P_i$, equipped with the natural coordinate function $\eta_i$. Recall that associated to each edge $e \in \wrtr^{[1]}_{in}$ is a wall $\mathbf{w}_{i_e}$. Define an affine map
\begin{equation}\label{eqn:tau_map}
\vec{\tau} : \mathfrak{M}_\wrtr(M_\real) \rightarrow \prod_{e \in \wrtr^{[1]}_{in}} N_{i_{e}}
\end{equation}
by requiring $\vec{\tau}^*(\eta_{i_e}) = \eta_{i_e} (\tau^{\mathfrak{e}}(\vec{s},x))$. Write $\mathcal{I}_x$ for $\real_{\leq 0}^{\wrtr^{[1]}} \times \{x\}$.

\begin{definition}\label{def:trop_tree_orientation}
	Assign a differential form $\nu_e$ on $\real_{\leq 0}^{|\wrtr^{[1]}|}$ to each $e \in \bar{\wrtr}^{[1]}$ recursively as follows. Set $\nu_e = 1$ if $e \in \partial_{in}^{-1}(\wrtr_{in}^{[0]})$. If $v$ is an internal vertex with $\partial_{out}^{-1}(v) = \{e_1, e_2\}$ and $\partial_{in}^{-1}(v) = \{e_3\}$ such that $\{e_1, e_2, e_3\}$ is clockwise oriented, then set $\nu_{e_3} =  (-1)^{|\nu_{e_2}|} \nu_{e_1} \wedge \nu_{e_2} \wedge ds_{e_{3}}$, where $|\nu_{e_2}|$ is the cohomological degree of $\nu_{e_2}$.
\end{definition}

The form $\nu_{\wrtr}$ attached to the edge $e_{out} \in \wrtr^{[1]}$ is a volume form on $\real_{\leq 0}^{|\wrtr^{[1]}|}$.

The following result can be proved in the same way as \cite[Lemma 5.33]{kwchan-leung-ma}.

\begin{lemma}\label{lem:tau_map_iso}
	We have
	$
	\vec{\tau}^*(d\eta_{i_{e_1}} \wedge \cdots \wedge d\eta_{i_{e_k}}) =  c \nu_{\wrtr} \wedge n_\wrtr + \varepsilon
	$
	for some $c> 0$, where $n_\wrtr \in N$ is a $1$-form on $M_\real$, $\nu_{\wrtr}^{\vee}$ is the top polyvector field on $\real_{\leq 0}^{|\wrtr^{[1]}|}$ dual to $\nu_{\wrtr}$ and $\iota_{\nu_{\wrtr}^{\vee}} \varepsilon = 0$. In particular, $\vec{\tau}_{\vert \mathcal{I}_x}$ is an affine isomorphism onto its codimension one image $C(\vec{\tau},x) \subset  \prod_{e \in \wrtr^{[1]}_{in}} N_{i_e}$ when $n _\wrtr \neq 0$.

\end{lemma}

The well-definedness of $\mathbf{L}_{k,\wrtr}(\tilde{\incoming},\dots,\tilde{\incoming})$ depends on the convergence of the integral in the following lemma. Write $\alpha_j$ in place of $\delta_{P_{i_{e_j}J_{e_j}}}$ and, for each $L >0$, set $\mathcal{I}_{x,L} = [-L,0]^{\wrtr^{[1]}} \times \{x\}$.

\begin{lemma}
\phantomsection
\label{lem:integral_convergent}
\begin{enumerate}
\item The integral
	$$
	\alpha_{\wrtr}(x) := -\int_{\mathcal{I}_x} (\tau^{\mathfrak{e}_1})^*(\alpha_1) \wedge \cdots \wedge (\tau^{\mathfrak{e}_k})^*(\alpha_k)
	$$
	is well-defined. Moreover, $\alpha_{\wrtr}=0$ if $n_\wrtr = 0$ and $\alpha_{\wrtr} \in \mathcal{W}^{-\infty}_1(M_\real)$ if $P_\wrtr = \emptyset$, where $P_\wrtr$ is defined as in Section \ref{sec:tropical_counting} by forgetting the ribbon structure of $\wrtr$.
	
	\item The integral 
	$\alpha_{\wrtr,L}(x):=-\int_{\mathcal{I}_{x,L}} (\tau^{\mathfrak{e}_1})^*(\alpha_1) \wedge \cdots \wedge (\tau^{\mathfrak{e}_k})^*(\alpha_k)$ uniformly converges to $\alpha_{\wrtr}(x)$ for $x$ in any pre-compact open subset $K \subset M_\real$. Furthermore, $(\alpha_\wrtr - \alpha_{\wrtr,L})_{\vert K} \in \mathcal{W}^{-\infty}_1(K)$ for sufficiently large $L$.
	
	 \item If $P_\wrtr \neq \emptyset$, so that $\dim_{\mathbb{R}}(P_\wrtr) = r-1$, and $\varrho:(a,b) \rightarrow M_\real$ is an embedded affine line intersecting $P_\wrtr$ positively\footnote{Intersecting positively means $\langle \varrho^{\prime}, n_\wrtr \rangle >0$.} and transversally in its relative interior $\Int_{re}(P_\wrtr)$, then 
	$
	\lim_{\hp \rightarrow 0} \int_{\varrho} \alpha_{\wrtr} =-1.
	$
\end{enumerate}
\end{lemma}

\begin{proof}
	Explicitly, the integral $\alpha_{\wrtr}(x)$ under consideration is 
	$$
	\big(\frac{1}{\pi \hp}\big)^{k/2} \int_{\mathcal{I}_x} \vec{\tau}^* \big( \prod_{j=1}^k e^{-(\eta_{i_{e_j}}^2)/\hp}d\eta_{i_{e_j}} \big) = \big(\frac{1}{\pi \hp}\big)^{k/2} \int_{\mathcal{I}_x}   e^{-(\sum_{j=1}^k (\tau^{\mathfrak{e_j}})^*\eta_{i_{e_j}}^2)/\hp}  \vec{\tau}^*(d\eta_{i_{e_1}} \cdots d\eta_{i_{e_k}} ).
	$$
	By Lemma \ref{lem:tau_map_iso}, the only case that we need consider is when $n_\wrtr \neq 0$, in which case $\vec{\tau}_{\vert \mathcal{I}_x}$ is an affine isomorphism onto its image $C(\vec{\tau},x)$, a codimension one closed affine subspace. The well-definedness of the integral is due to the fact that 
	$
	\int_{C(\vec{\tau},x)} e^{-(\sum_{j=1}^k (\tau^{\mathfrak{e_j}})^*\eta_{i_{e_j}}^2)/\hp}  \mu_{C(\vec{\tau},x)} < \infty
	$
	for any affine linear volume form $\mu_{C(\vec{\tau},x)}$ on $ C(\vec{\tau},x)$. When $P_\wrtr = \emptyset$, we have $ 0 \notin C(\vec{\tau},x)$ for all $x \in M_\real$, which implies $\alpha_{\wrtr} \in \mathcal{W}^{-\infty}_1(M_\real)$.

	Notice that 
	$$
	\bigcap_{ e \in \wrtr^{[1]}_{in}} \{(\tau^{\mathfrak{e}})^*\eta_{i_e} = 0\} = \overline{\mathfrak{M}}_{\wrtr} (M_\real,\tilde{\mathscr{D}}_{in,l}) \subset \mathfrak{M}_\wrtr(M_\real).
	 $$
	 Furthermore, for any pre-compact subset $K \subset M_\real$ and $b>0$, there exists an $L_b$ such that 
	 $
	 (\mathcal{I}_x \setminus \mathcal{I}_{x,L_b}) \cap 	\bigcap_{ e \in \wrtr^{[1]}_{in}} \{|(\tau^{\mathfrak{e}})^*\eta_{i_e}| \leq b\} = \emptyset
	 $
	 for all $x \in K$, as follows from the fact that $\vec{\tau}_{\vert \mathcal{I}_x}$ is an affine isomorphism onto its image. This implies that $\alpha_{\wrtr,L}$ converges uniformly to $\alpha_{\wrtr}$ on $K$ and that $(\alpha_\wrtr - \alpha_{\wrtr,L})_{\vert K} \in \mathcal{W}^{-\infty}_1(K)$ for sufficiently large $L$. 
	 
	 Suppose now that $P_\wrtr$ and $\varrho$ are as in the final statement of the lemma. Consider the affine subspace $\mathcal{I}_\varrho := \bigcup_{t \in (a,b)} \mathcal{I}_{\varrho(t)} \subset \mathfrak{M}_\wrtr(M_\real)$. We have 
	 $$
	 \int_{\varrho} \alpha_{\wrtr } = -\int_{\mathcal{I}_\varrho}(\tau^{\mathfrak{e}_1})^*(\alpha_1) \wedge \cdots \wedge (\tau^{\mathfrak{e}_k})^*(\alpha_k)
	 = -\int_{\vec{\tau}(\mathcal{I}_\varrho)} \alpha_1 \wedge \cdots \wedge \alpha_k,
	 $$
	 where $\vec{\tau}(\mathcal{I}_\varrho) \subset \prod_{e \in \wrtr^{[1]}_{in}} N_{i_{e}}$. Here we apply Lemma \ref{lem:tau_map_iso} to conclude that $\vec{\tau}$ is an affine isomorphism onto its image when $P_\wrtr \neq \emptyset$. Since $\varrho$ intersects $P_\wrtr$ in $\Int_{re}(P_\wrtr)$ and $\tilde{D}_{in,l}$ is generic, we have $0 \in \Int(\vec{\tau}(\mathcal{I}_\varrho) )$. Together with the explicit form of $\alpha_1 \wedge \cdots \wedge \alpha_k$, we then obtain $\lim_{\hp \rightarrow 0} \int_{\varrho} \alpha_\wrtr =-1$.
	 \end{proof}

\subsubsection{Relation with tropical counting}\label{sec:relating_MC_solution_counting}

The following result is a modification of \cite[\S 5]{kwchan-leung-ma}.

\begin{lemma}\label{lem:tree_support_lemma}
	For each $\wrtr \in \wrtree{k}$, we have
	$
	\mathbf{L}_{k,\wrtr}(\tilde{\incoming},\dots,\tilde{\incoming}) = \Big( \prod_{e \in \partial^{-1}_{in}(\wrtr^{[0]}_{in})} (\# J_{e,i_e})! \Big) \alpha_{\wrtr} g_{\wrtr} u^{\vec{J}_{\wrtr}}
	$.
\end{lemma}

\begin{proof}
	We proceed by induction on the cardinality of $\wrtr^{[0]}$. In the initial case, $\wrtr^{[0]} = \emptyset$, the only tree is that with a unique edge and there is nothing to prove.
	
	For the induction step, the root vertex $v_r \in \wrtr^{[0]}$ is adjacent to the outgoing edge $e_{out}$ and two incoming edges, say $e_1$ and $e_2$. Assume that $\{e_1,e_2,e_{out}\}$ are clockwise oriented. Split $\wrtr$ at $v_r$, thereby obtaining trees $\wrtr_1$ and $\wrtr_2$ with outgoing edges $e_1$ and $e_2$ and $k_1$ and $k_2$ incoming edges, respectively. By the induction hypothesis, we can write
	$\mathbf{L}_{k_i,\wrtr_i}(\tilde{\incoming},\dots,\tilde{\incoming}) = \Big( \prod_{e \in \partial^{-1}_{in}((\wrtr_i)^{[0]}_{in})} (\# J_{e,i_e})! \Big)  \alpha_{\wrtr_i} g_{\wrtr_i} u^{\vec{J}_{\wrtr_i}}$, $i =1,2$. We therefore have
	$$\mathbf{L}_{k,\wrtr}(\tilde{\incoming},\dots,\tilde{\incoming}) =  -\Big( \prod_{e \in \partial^{-1}_{in}(\wrtr^{[0]}_{in})} (\# J_{e,i_e})! \Big) \mathbf{H} (\alpha_{\wrtr_1} \wedge \alpha_{\wrtr_2}) [g_{\wrtr_1},g_{\wrtr_2}] u^{\vec{J}_{\wrtr_1}} u^{\vec{J}_{\wrtr_2}}.$$
	By definition, $g_\wrtr = [g_{\wrtr_1},g_{\wrtr_2}]$ and $u^{\vec{J}_\wrtr} = u^{\vec{J}_{\wrtr_1}} u^{\vec{J}_{\wrtr_2}}$. Finally, the proof of \cite[Lemma 5.31]{kwchan-leung-ma} shows that 
	$$
	\mathbf{H} (\alpha_{\wrtr_1} \wedge \alpha_{\wrtr_2}) = \int_{\mathcal{I}_x} (\tau^{\mathfrak{e}_1})^*(\alpha_1) \cdots (\tau^{\mathfrak{e}_k})^*(\alpha_k) =-\alpha_{\wrtr}(x).
	$$
	Note that the well-definedness of $\mathbf{H} (\alpha_{\wrtr_1} \wedge \alpha_{\wrtr_2})$ is guaranteed by Lemma \ref{lem:integral_convergent}.
\end{proof}

\begin{lemma}\label{lem:solution_asy_support}
	For each $\wrtr \in \wrtree{k}$, we have $\alpha_\wrtr \in \mathcal{W}^{1}_{P_\wrtr}(M_\real) \cap \mathcal{W}^{1}_1(M_\real)$ if $P_\wrtr \neq \emptyset$. 
	\end{lemma}

\begin{proof}
	We proceed by induction on the cardinality of $\wrtr^{[0]}$. The initial case, $\wrtr^{[0]} = \emptyset$, holds by Lemma \ref{lem:input_delta}.
	
	For the induction step, split $\wrtr$ at $v_r \in \wrtr^{[0]}$ to obtain trees $\wrtr_1$ and $\wrtr_2$, as in the proof Lemma \ref{lem:tree_support_lemma}. We can assume that each $P_{\wrtr_i}$ is non-empty and that $P_{\wrtr_1}$, $P_{\wrtr_2}$ intersect transversally and generically, as in Definition \ref{def:tropical_generic_assumption}. Then $Q = P_{\wrtr_1} \cap P_{\wrtr_2}$ is a codimension two affine subspace of $M_\real$. The induction hypothesis implies $\alpha_{\wrtr_i} \in 
	\mathcal{W}^{1}_{P_{\wrtr_i}}(M_\real) \cap \mathcal{W}^{1}_1(M_\real)$ and Lemma \ref{lem:support_product} gives $\alpha_{\wrtr_1} \wedge \alpha_{\wrtr_2} \in \mathcal{W}^2_Q(M_\real) \cap \mathcal{W}^2_2(M_\real)$. Arguing as in the proof of Lemma \ref{lem:tree_support_lemma}, we find $\alpha_\wrtr = -\mathbf{H}_{m_{e_{out}}} (\alpha_{\wrtr_1} \wedge \alpha_{\wrtr_2} )$, which is nonzero only if $n_\wrtr \neq 0$. Note that if $n_\wrtr \neq 0$, then $-m_{\wrtr} = - m_{e_{out}} $ is not tangent to $Q$.

We would like to apply Lemma \ref{lem:integral_lemma_modified} to conclude our result. However, the operator $\mathbf{H}_{m_{e_{out}}} $ is slightly different from that appearing in Lemma \ref{lem:integral_lemma_modified}. A modification is therefore required.
	
	To simplify notation, write $m = m_{e_{out}}$. Since $m$ is not tangent to $Q$, we can assume that the chain $U^m_{\bullet}$ used to define $\mathsf{H}_m$ is such that $U^m_{r-1}$ separates $M_\real$ into $M_\real^-$ and $M_\real^+$, $-m$ points into $M_\real^+$ and $Q \subset M_\real^+$. With this choice, we obtain a homotopy operator $\mathsf{H}_m$ as in Section \ref{sec:homotopy_operator} which, by Lemma \ref{lem:integral_lemma_modified}, satisfies $\mathsf{H}_m (\alpha_{\wrtr_1} \wedge \alpha_{\wrtr_2}) \in \mathcal{W}^{1}_{P_\wrtr}(M_\real) \cap \mathcal{W}^{1}_1(M_\real)$. Since $U^{m}_{r-1} \cap Q = \emptyset$, we find that $\mathbf{H}_{m,L} (\alpha_{\wrtr_1} \wedge \alpha_{\wrtr_2} ) - \mathsf{H}_m (\alpha_{\wrtr_1} \wedge \alpha_{\wrtr_2})$ lies in $\mathcal{W}^{-\infty}_1(M_\real)$. It follows that $\mathbf{H}_{m,L} (\alpha_{\wrtr_1} \wedge \alpha_{\wrtr_2} )$ satisfies the desired property.
	\end{proof}

\begin{lemma}\label{lem:well_definedness_solution}
	The element 
	$\tilde{\varPhi}:= \sum_{k \geq 1} \frac{1}{2^{k-1}}\sum_{\wrtr \in \wrtree{k}} \mathbf{L}_{k,\wrtr}(\tilde{\incoming},\dots,\tilde{\incoming})
	$ is well-defined in $\mathcal{G}^*\otimes_{\mathbb{C}} \tilde{R}_l$. Furthermore, it solves equation \eqref{eqn:MC_equation}.
	\end{lemma}

\begin{proof}
	Well-definedness of $\tilde{\varPhi}$ follows from Lemmas \ref{lem:integral_convergent} and \ref{lem:tree_support_lemma}. The same reasoning as Section \ref{sec:sum_over_tree_formula} then shows that 
	$
	\tilde{\varPhi}= \tilde{\incoming} - \half \mathbf{H} [\tilde{\varPhi},\tilde{\varPhi}].
	$
	
	We will use Proposition \ref{prop:pseudoMC_to_MC} to show that $\tilde{\varPhi}$ solves equation \eqref{eqn:MC_equation}. Fix a pre-compact open subset $K\subset M_\real$ and consider the restriction of equation \eqref{eqn:MC_equation} to $K$. By Lemma \ref{lem:integral_convergent}, we may choose $L$ sufficiently large so as to ensure that the truncation
	$
	\tilde{\varPhi}_L := \sum_{k \geq 1} \frac{1}{2^{k-1}}\sum_{\substack{\wrtr \in \wrtree{k}\\ u^{\vec{J}_\wrtr} \neq 0}} \alpha_{\wrtr,L} g_{\wrtr} u^{\vec{J}_\wrtr}
	$
	satisfies $\alpha_\wrtr - \alpha_{\wrtr,L} \in \mathcal{W}^{-\infty}_1(K)$. Indeed, this is possible because there are only finitely many terms with $u^{\vec{J}_\wrtr} \neq 0$ in the expression for $\tilde{\varPhi}_L $, as the maximal ideal of $\tilde{R}_l$ is nilpotent. Notice that $\tilde{\varPhi}_L$ satisfies
	$
	\tilde{\varPhi}_L= \tilde{\incoming} - \half \mathbf{H}_L [\tilde{\varPhi}_L,\tilde{\varPhi}_L],
	$
	where $\mathbf{H}_L : =\bigoplus_{m \in M_\sigma^+} \mathbf{H}_{L,m}$ and
	$
	\mathbf{H}_{L,m}(\alpha)(x):= \int_{-L}^0 (\iota_{\dd{s}}(\tau^m)^*(\alpha)(s,x)) ds.
	$

	Similar to Proposition \ref{prop:pseudoMC_to_MC}, it suffices to show that $\mathbf{P}_L [\tilde{\varPhi}_L, \tilde{\varPhi}_L] = 0$ on $K$, where $\mathbf{P}_L :=\bigoplus_{m \in M_\sigma^+}  \mathbf{P}_{L,m}$ and $\mathbf{P}_{L,m} (\beta) := (\tau^{m}_{-L})^* \beta$. Since $\tilde{\varPhi}_L$ is a sum over trees, we consider $\wrtr_i \in \wrtree{k_i}$, $i=1,2$, with $u^{\vec{J}_{\wrtr_i}} \neq 0$ and the associated terms $\alpha_{\wrtr_i,L} g_{\wrtr_i}$, where $g_{\wrtr_i} \in \mathfrak{h}_{m_{\wrtr_i},n_{\wrtr_i},\tilde{R}_{l}}$. Join $\wrtr_1$ and $\wrtr_2$ to give $\wrtr$. It suffices to assume $n_\wrtr \neq 0$, since $[g_{\wrtr_1} , g_{\wrtr_2}] \in \mathfrak{h}_{m_\wrtr,0,\tilde{R}_{l}} = \{0\}$ when $n_\wrtr = 0$. If $n_\wrtr \neq 0$, then $m_{\wrtr}$ is not tangent to $P_{\wrtr_1} \cap P_{\wrtr_2}$. We may therefore choose $L$ sufficiently large so that $\tau_{-L}^{m_{\wrtr}}(K) \cap P_{\wrtr_1} \cap P_{\wrtr_2} = \emptyset$. As a result, we have $(\tau_{-L}^{m_{\wrtr}})^*(\alpha_{\wrtr_1} \wedge \alpha_{\wrtr_2}) = 0$ in $\mathcal{H}^2_{\tilde{R}_l}(K)$.	
	\end{proof}

Let $\wtr \in \wtree{k}$ with $P_{\wtr}\neq \emptyset$. Since the monomial weights $u^{\vec{J}_e}$ at incoming edges $e \in \wtr^{[1]}_{in}$ are distinct, there are, up to isomorphism, exactly $2^{k-1}$ ribbon structures on $\wtr$. Note that $\mathbf{L}_{k,\wrtr}(\tilde{\incoming},\dots,\tilde{\incoming})$ does not depend\footnote{This can also be deduced from the proof of Lemma \ref{lem:tree_support_lemma} by observing that the dependence of $\alpha_\wrtr$ and $g_\wrtr$ on the ribbon structure of $\wrtr$ cancels out in the formula for $\mathbf{L}_{k,\wrtr}(\tilde{\incoming},\dots,\tilde{\incoming})$.} on the ribbon structure of $\wrtr$, since $\tilde{\incoming} \in \mathcal{H}^{1}_{ \tilde{R}_l}$ and $\tilde{\incoming}$ commutes with odd elements of $ \mathcal{H}^{1}_{ \tilde{R}_l}$. It follows that $\tilde{\varPhi} = \sum_{k \geq 1}\sum_{\wtr \in \wtree{k}} \mathbf{L}_{k,\wrtr}(\tilde{\incoming},\dots,\tilde{\incoming})$, where $\wrtr$ is any ribbon tree whose underlying tree $\underline{\wrtr}$ is $\wtr$. Combining Lemmas \ref{lem:tree_support_lemma}, \ref{lem:solution_asy_support} and \ref{lem:well_definedness_solution}, we conclude the following theorem.

\begin{theorem}\label{thm:theorem_1}
		The Maurer--Cartan solution $\tilde{\varPhi} \in \mathcal{H}^{1}_{ \tilde{R}_l}$ of Lemma \ref{lem:tree_support_lemma} can be expressed as the following sum over trees:
		$$
		\tilde{\varPhi} = \sum_{k \geq 1} \sum_{\substack{\wtr \in \wtree{k}\\ \mathfrak{M}_{\wtr}(M_\real,\tilde{\mathscr{D}}_{in,l}) \neq \emptyset}} \alpha_{\wtr} \log(\Theta_{\wtr}).
		$$
		Here $\mathfrak{M}_{\wtr}(M_\real,\tilde{\mathscr{D}}_{in,l}) \neq \emptyset$ indicates the existence of a tropical disk in $(M_\real,\tilde{\mathscr{D}}_{in,l})$ of combinatorial type $\wtr$, the wall-crossing factor $\Theta_{\wtr}$ is given by
		\[
		\log(\Theta_{\wtr}) = \Big( \prod_{e \in \partial^{-1}_{in}(\wtr^{[0]}_{in})} (\# J_{e,i_e})! \Big) g_{\wtr}  u^{\vec{J}_{\wtr}}
		\]
		and $\alpha_{\wtr}$ is a $1$-form with asymptotic support on $P_{\wtr}$ which satisfies
		$
		\lim_{\hp\rightarrow 0 }\int_{\varrho} \alpha_{\Gamma} = -1
		$
		for any affine line $\varrho$ intersecting positively with $P_\wtr$.
	\end{theorem}

Theorem \ref{thm:theorem_1} gives a bijection between tropical disks and summands of $\tilde{\varPhi}$. Together with Proposition \ref{prop:consistence_from_solution} below, which relates Maurer--Cartan solutions with consistent scattering diagrams, this provides an alternative realization of the enumerative interpretation of Theorem \ref{thm:GPS_theorem}.

\subsection{Non-perturbed initial scattering diagram}\label{sec:non_perturbed_scattering}

In this section we study the relationship between Maurer--Cartan elements and non-perturbed scattering diagrams. We are motivated by the fact that it is not always be possible (or desirable) to perturb the incoming diagram. This is the case, for example, for Hall algebra scattering diagrams. With appropriate modifications, we find that most of the results of Sections \ref{sec:modified_MC_solution} and \ref{sec:relating_MC_solution_counting} remain true without perturbation.

Let $\mathscr{D}_{in}$ be an initial scattering diagram and consider $\mathbf{L}_{k,\lrtr}(\incoming,\dots,\incoming)$ as in Section \ref{sec:tropical_counting_thm}. The main difference between the perturbed and non-perturbed cases is that , when $P_\lrtr \neq \emptyset$, we have $\dim_\real(P_\lrtr) = r-1$ in former whereas we only have $0\leq \dim_{\real}(P_\lrtr)  \leq r-1$ in the latter.

To begin, note that the first two parts of Lemma \ref{lem:integral_convergent} remain true in the context of labeled ribbon $k$-trees $\lrtr$. However, $\lim_{\hp \rightarrow 0} \int_{\varrho} \alpha_{\lrtr}$ need not equal $-1$, even when $\dim_\real(P_\lrtr) =r-1$. Indeed, we only have $0 \in \vec{\tau}(\mathcal{I}_\varrho)$, as opposed to $0 \in \Int(\vec{\tau}(\mathcal{I}_\varrho))$, so the relevant part of the proof of Lemma \ref{lem:integral_convergent} does not apply. The replacement of the third part of Lemma \ref{lem:integral_convergent} will be given in Lemma \ref{lem:iterated_integral_constant}.

\begin{lemma}\label{lem:tree_support_lemma_labeled}
	For each $\lrtr \in \lrtree{k}$, we have
	$
	\mathbf{L}_{k,\lrtr}(\incoming,\dots,\incoming) = \alpha_{\lrtr} g_{\lrtr},
	$
	with $g_{\lrtr}$ as in Definition \ref{def:multiplicity}.   
\end{lemma}

\begin{proof}
This can be proved in the same way as Lemma \ref{lem:tree_support_lemma}.
\end{proof}

The next result gives the required modification of Lemma \ref{lem:solution_asy_support}.

\begin{lemma}\label{lem:solution_asy_support_labeled}
Let $\lrtr \in \lrtree{k}$ and let $P_\lrtr \subset P$ be the codimension one hyperplane normal to $n_\lrtr$.
\begin{enumerate}
\item We have $\alpha_\lrtr \in \mathcal{W}^{1}_{P_\lrtr}(M_\real) \cap \mathcal{W}^{1}_1(M_\real)$ if $\dim_\real(P_\lrtr) = r-1$ and $\alpha_\lrtr \in \mathcal{W}^{1}_{P}(M_\real) \cap \mathcal{W}^{1}_1(M_\real)$ otherwise. In either case, $\alpha_{\lrtr \vert M_\real \setminus P_\lrtr} \in \mathcal{W}^{-\infty}_1(M_\real \setminus P_\lrtr)$.

\item If $\dim_\real(P_\lrtr) = r-1$, then there exists a polyhedral decomposition $\mathcal{P}_\lrtr$ of $P_\lrtr$ such that $d(\alpha_\lrtr)_{\vert M_\real \setminus  |\mathcal{P}_\lrtr^{[r-2]}|} \in \mathcal{W}^{-\infty}_2(M_\real \setminus |\mathcal{P}_\lrtr^{[r-2]}|)$, where $\mathcal{P}^{[l]}$ denotes the set of $l$-dimensional strata and $|\mathcal{P}^{[l]}|$ is the underlying set of $\mathcal{P}^{[l]}$.
\end{enumerate}
\end{lemma}

\begin{proof}
We proceed by induction on the cardinality of $\lrtr^{[0]}$. The initial case, $\lrtr^{[0]} = \emptyset$, holds by Lemma \ref{lem:input_delta}.
	
	For the induction step, split $\lrtr$ at $v_r \in \lrtr^{[0]}$ to obtain $\lrtr_1$ and $\lrtr_2$, as in the proof Lemma \ref{lem:tree_support_lemma}. We can assume that $n_{\lrtr} \neq 0$, as otherwise $\alpha_{\lrtr} =0$ by Lemma \ref{lem:integral_convergent}. By the induction hypothesis, we have $\alpha_{\lrtr_i} \in \mathcal{W}^{1}_{P_{i}}(M_\real) \cap \mathcal{W}^{1}_1(M_\real)$, with $P_i = n_{\lrtr_i}^\perp$ containing $P_{\lrtr_i}$, $i=1,2$. Since $n_\lrtr \neq 0$, $P_1$ and $P_2$ intersect transversally. Applying Lemma \ref{lem:support_product} then gives $\alpha_{\lrtr_1} \wedge \alpha_{\lrtr_2} \in \mathcal{W}^2_Q(M_\real) \cap \mathcal{W}^2_2(M_\real)$, where $Q = P_1 \cap P_2$. We have $\alpha_\wrtr = -\mathbf{H}_{m_{e_{out}}} (\alpha_{\wrtr_1} \wedge \alpha_{\wrtr_2} )$, as in Lemma \ref{lem:tree_support_lemma}. Similar to the proof of Lemma \ref{lem:solution_asy_support}, since $Q-\real_{\geq 0} m_\lrtr \subset P$, we can apply Lemma \ref{lem:integral_lemma_modified} to conclude that $\alpha_\lrtr \in \mathcal{W}^{1}_{P}(M_\real) \cap \mathcal{W}^{1}_1(M_\real)$. Using the induction hypothesis and the relation $P_\lrtr =(P_{\lrtr_1} \cap P_{\lrtr_2}) -\real_{\geq 0} m_\lrtr$, we have $(\alpha_{\wrtr_1} \wedge \alpha_{\wrtr_2})_{\vert M_\real \setminus (P_{\lrtr_1} \cap P_{\lrtr_2})} \in \mathcal{W}^{-\infty}_2(M_\real \setminus (P_{\lrtr_1} \cap P_{\lrtr_2}))$, which gives $\alpha_{\lrtr \vert M_\real \setminus P_\lrtr} \in \mathcal{W}^{-\infty}_1(M_\real \setminus P_\lrtr)$.

	Since $\alpha_\lrtr \in \mathcal{W}^1_1(M_\real)$, we can write $d\alpha_\lrtr = \sum_{j} \beta_j$, where $\beta_j \in \mathcal{W}^2_{Q_j}(M_\real)$ for some codimension two polyhedral subsets $Q_j \subset M_\real$. In particular, $d \alpha_\lrtr |_{M_\real \setminus \bigcup_{j} Q_j } \in \mathcal{W}^{-\infty}(M_\real \setminus \bigcup_{j} Q_j )$ and $d \alpha_{\lrtr \vert M_\real \setminus P_\lrtr} \in \mathcal{W}^{-\infty}(M_\real \setminus P_\lrtr)$. Letting $\mathcal{P}_\lrtr$ be a polyhedral decomposition of $P_\lrtr$ such that $|\mathcal{P}_{\lrtr}^{[r-2]}|$ contains $P_\lrtr \cap \bigcup_{j} Q_j$, we obtain the desired result.
	\end{proof}


\begin{lemma}\label{lem:iterated_integral_constant}
	Let $\mathcal{P}_\lrtr$ be a polyhedral decomposition of $P_\lrtr$ which satisfies the second part of Lemma \ref{lem:solution_asy_support_labeled} and let $\sigma \in \mathcal{P}^{[r-1]}_\lrtr$. Then there exists a constant $c_{\lrtr,\sigma}>0$ such that 
	$
	\lim_{\hp \rightarrow 0} \int_{\varrho} \alpha_{\lrtr} = - c_{\lrtr,\sigma} 
	$
	for any embedded affine line $\varrho$ which intersects positively and transversally with $\sigma$ in $\Int_{re}(\sigma)$. 
\end{lemma}

\begin{proof}
	For any such $\varrho$, we have
	$
	\lim_{\hp \rightarrow 0} \int_{\varrho} \alpha_{\lrtr}  = -  \int_{\vec{\tau}(\mathcal{I}_\varrho)} \alpha_1 \wedge \cdots \wedge \alpha_k
	$, as in the proof of Lemma \ref{lem:integral_convergent}. Although $0 \in \vec{\tau}(\mathcal{I}_\varrho)$ instead of $0 \in \Int(\vec{\tau}(\mathcal{I}_\varrho))$, we still have $\int_{\vec{\tau}(\mathcal{I}_\varrho)} \alpha_1 \wedge \cdots \wedge  \alpha_k = -c$ for some constant $c>0$. It remains to argue that $c$ is independent of $\varrho$.
	
	Let $\varrho_1$ and $\varrho_2$ be paths as above. Join the end points of $\varrho_1$ and $\varrho_2$ by paths $\gamma_0$ and $\gamma_1$ which do not intersect in $P_\lrtr$ to form a cycle $C$. Then $\lim_{\hp \rightarrow 0} \int_{\gamma_i} \alpha_{\lrtr} = 0$ and $\lim_{\hp \rightarrow 0 } \int_{C} \alpha_{\lrtr} = \lim_{\hp \rightarrow 0} \int_{D} d \alpha_\lrtr = 0$ for some $2$-chain $D$ with $D \cap |\mathcal{P}^{[r-2]}_\lrtr| = \emptyset$. It follows that $\lim_{\hp \rightarrow 0} \int_{\varrho_1} \alpha_{\lrtr} = \lim_{\hp \rightarrow 0} \int_{\varrho_2} \alpha_{\lrtr}$.
	\end{proof}

We claim that
$
\varPhi:= \sum_{k \geq 1} \frac{1}{2^{k-1}}\sum_{\lrtr \in \lrtree{k}} \mathbf{L}_{k,\lrtr}(\incoming,\dots,\incoming)
$
defines an element of $\hat{\mathcal{G}}^*$ which satisfies equation \eqref{eqn:MC_equation} in $\hat{\mathcal{H}}^*$. Indeed, if we consider this claim in $\mathcal{G}^{<k,*}:= \mathcal{W}^{0}_* \otimes_{\mathbb{C}} \mathfrak{h}^{<k}$ and $\mathcal{H}^{<k,*}:= \big( \mathcal{W}^{0}_*/\mathcal{W}^{-1}_* \big) \otimes_{\mathbb{C}} \mathfrak{h}^{<k}$, we will have a finite number of terms and the proof of Lemma \ref{lem:well_definedness_solution} applies. The claim then follows by taking limits.

Let $\ltr \in \ltree{k}$. Since $\mathbf{L}_{k,\lrtr}(\incoming,\dots,\incoming)$ does not depend on the ribbon structure of $\lrtr$, we can make sense of the sum $\mathbf{L}_{k,\ltr}(\incoming,\dots,\incoming)$. Since the labeling of the incoming edges $e \in \ltr^{[1]}_{in}$ need not be distinct, we have
\begin{equation}\label{eqn:ribbon_tree_non_ribbon_tree_relation}
\frac{1}{|\Aut(\ltr)|}\mathbf{L}_{k,\ltr}(\incoming, \dots, \incoming) = \sum_{\underline{\lrtr} = \ltr} \frac{1}{2^{k-1}} \mathbf{L}_{k,\lrtr}(\incoming,\dots,\incoming)
\end{equation}
and hence 
$
\varPhi= \sum_{k \geq 1}\sum_{\ltr \in \ltree{k}}\frac{1}{|\Aut(\ltr)|} \mathbf{L}_{k,\ltr}(\incoming,\dots,\incoming).
$
Combining the above arguments yields the following modification of Theorem \ref{thm:theorem_1}.

\begin{theorem}\label{thm:theorem_1_modified}
		The Maurer--Cartan solution $\varPhi \in \hat{\mathcal{H}}^*$ can be expressed as a sum over trees,
		\[
		\varPhi = \sum_{k \geq 1} \sum_{\substack{\ltr \in \ltree{k}\\ \mathfrak{M}_{\ltr}(M_\real,\mathscr{D}_{in}) \neq \emptyset}} \frac{1}{|\Aut(\ltr)|} \alpha_{\ltr} g_\ltr,
		\]
		with $\alpha_\ltr \in \mathcal{W}^{1}_{P}(M_\real) \cap \mathcal{W}^{1}_1(M_\real)$ for the codimension one affine subspace $P_\ltr \subset P$ normal to $n_\ltr$. 
		
		Furthermore, when $\dim_{\real}(P_\ltr) = r-1$, there exists a polyhedral decomposition $\mathcal{P}_\ltr$ of $P_\ltr$ such that, for each $\sigma \in \mathcal{P}_\ltr^{[r-1]}$, there is a constant $c_{\ltr,\sigma}$ such that 
		$
		\lim_{\hp\rightarrow 0 }\int_{\varrho} \alpha_{\ltr} = -c_{\ltr,\sigma}
		$
		 for any affine line $\varrho$ intersecting positively\footnote{Positivity depends on $n_\ltr$, which is defined up to sign. However, this sign ambiguity cancels with that of $g_\ltr$, as mentioned in Definition \ref{def:multiplicity}.} with $\sigma$ in $\Int_{re}(\sigma)$
\end{theorem}

\begin{definition}\label{def:diagram_from_solution}
	Let $\varPhi$ be as in Theorem \ref{thm:theorem_1_modified}. Define a scattering diagram $\mathscr{D}(\varPhi)$ as follows. For each $\ltr \in \ltree{k}$ with $\dim_\real(P_\ltr) = r-1$, let $\sigma \in \mathcal{P}_\ltr^{[r-1]}$ be a maximal cell with associated constant $c_{\ltr,\sigma}$. Define a wall $\mathbf{w}_{\ltr,\sigma} = (m_{\ltr},n_{\ltr},P_{\ltr,\sigma},\Theta_{\ltr,\sigma})$ so that $m_{\ltr}$ and $n_{\ltr}$ are as in the case of weighted $k$-trees (see Definitions \ref{def:weighted_tree_tropical} and \ref{def:multiplicity}), $P_{\ltr,\sigma} = \sigma$ and $\Theta_{\ltr,\sigma} = \exp(\frac{c_{\ltr,\sigma}}{|\Aut(\ltr)|} g_\ltr)$. 
\end{definition}

We claim that $\mathscr{D}(\varPhi)$ is equivalent to $\mathcal{S}(\mathscr{D}_{in})$. We would like to apply the main result of \cite{kwchan-leung-ma} to conclude that $\mathscr{D}(\varPhi)$ is a consistent extension of $\mathscr{D}_{in}$. However, this result does not apply directly to the present situation, so must must supply some modifications. Firstly, we have $\mathscr{D}(\varPhi)^{<k} = \mathscr{D}(\varPhi^{<k})$, where $\varPhi^{<k}$ is the image $\varPhi$ in $\mathcal{H}^{<k,*}$ and $\mathscr{D}(\varPhi)^{<k}$ is the diagram obtained by replacing the wall-crossing automorphisms with their images under $\hat{\mathfrak{h}} \rightarrow \mathfrak{h}^{<k}$. To prove consistency of $\mathscr{D}(\varPhi)$, it suffices to prove consistency of $\mathscr{D}(\varPhi^{<k})$ for each $k$. For the latter, consider a polyhedral decomposition $\mathcal{J}(\mathscr{D}(\varPhi^{<k}))$ of $\text{Joints}(\mathscr{D}(\varPhi^{<k}))$ such that, for each $\mathfrak{j} \in \mathcal{J}(\mathscr{D}(\varPhi^{<k}))^{[r-2]}$, the intersection $P_{\ltr} \cap \mathfrak{j}$ is a facet of $\mathfrak{j}$ for all labeled trees $\ltr$ with $g_\ltr \neq 0 \in \mathfrak{h}^{<k}$. It suffices to prove consistency at each joint $\mathfrak{j}$.

Let $U$ be a convex neighborhood of $\Int_{re}(\mathfrak{j})$ such that $(U \setminus \mathfrak{j}) \cap P_\ltr \neq \emptyset$ only if $\dim_{\real}(P_\ltr) = r-1$. There is a decomposition 
$
\varPhi_{\vert U} = \sum_{(\ltr,\sigma) \in \mathbb{W}} \varPhi^{(\ltr,\sigma)} +\mathcal{E}.
$
Here $\mathbb{W}$ is the set of pairs $(\ltr,\sigma)$ for which $\dim_\real(P_\ltr) = r-1$ and $\sigma \in \mathcal{P}_{\ltr}$ and $\sigma \cap \Int_{re}(\mathfrak{j}) \neq \emptyset$. Restricted to $U \setminus \mathfrak{j}$, the summand $\varPhi^{(\ltr,\sigma)}$ is equal to $\frac{1}{|\Aut(\ltr)|} \alpha_{\ltr} g_\ltr$. The final term $\mathcal{E} = \sum_{\substack{P_\ltr \cap U \neq \emptyset\\ \dim_\real(P_\ltr) <r-1}} \frac{1}{|\Aut(\ltr)|} \alpha_{\ltr} g_\ltr$ satisfies $\mathcal{E}_{\vert U\setminus \mathfrak{j}} = 0$, as follows from our assumptions on $\mathcal{J}(\mathscr{D}(\varPhi^{<k}))$ and the fact that $(U \setminus \mathfrak{j}) \cap  P_\ltr = \emptyset $ for those $P_\ltr$ satisfying $\dim(P_\ltr) <r-1$. 
 
Since the sum $\sum_{(\ltr,\sigma) \in \mathbb{W}} \varPhi^{(\ltr,\sigma)}$ satisfies Assumptions I and II of \cite[Introduction]{kwchan-leung-ma}, the following result can be proved using the methods of \cite{kwchan-leung-ma}.

 \begin{prop}\label{prop:consistence_from_solution}
 	The scattering diagram $\mathscr{D}(\varPhi)$ is consistent.
 \end{prop}
 
By applying Theorem \ref{thm:kontsevich_soibelman_thm}, we conclude that the scattering diagrams $\mathscr{D}(\varPhi)$ and $\mathcal{S}(\mathscr{D}_{in})$ are equivalent. Similarly, in the perturbed case, $\mathscr{D}(\tilde{\varPhi})$ and $\mathcal{S}(\tilde{\mathscr{D}}_{in,l})$ are equivalent.

\subsection{Theta functions as flat sections}\label{sec:theta_function_from_MC}

Let $\varPhi \in \hat{\mathcal{H}}^{1}$ be a Maurer--Cartan element. Then $d_{\varPhi} := d + [\varPhi,-]$ is a differential which acts on the graded algebra $\hat{\mathcal{A}}^*$. The space of flat sections of $d_{\varPhi}$,
\[
\Ker(d_\varPhi)  = \{ s \in \hat{\mathcal{A}}^0 \ | \ d_{\varPhi}(s) = 0 \},
\]
inherits a product from $\hat{\mathcal{A}}^*$. Similarly, $\Ker^{<k}(d_\varPhi)$ inherits a product from $\mathcal{A}^{<k,*}$. The goal of this section is to relate $\Ker(d_\varPhi)$ or $\Ker^{<k}(d_\varPhi)$ with the theta functions introduced in Section \ref{sec:theta_function_def}.

\subsubsection{Wall-crossing of flat sections}\label{sec:wall_crossing_flat_sections}

In this section we prove a wall-crossing formula for flat sections $\Ker^{<k}(d_\varPhi)$ (and hence $\Ker(d_\varPhi)$) using arguments similar to those of \cite[Introduction]{kwchan-ma-p2}.

Consider a polyhedral decomposition $\mathcal{P}^{<k}$ of $\Supp(\mathscr{D}(\varPhi^{<k}))$ with the property that, for every $0 \leq l \leq r-1$ and $\sigma \in \mathcal{P}^{<k,[l]}$, we have $\sigma \subset P_\mathbf{w}$ for some wall $\mathbf{w} \in \mathscr{D}(\varPhi^{<k})$ and $P_\ltr \cap \sigma$ is a facet of $\sigma$ for every $P_\ltr$ with $g_\ltr \neq 0 \in \mathfrak{h}^{<k}$. Fix a maximal cell $\sigma \in\mathcal{P}^{<k,[r-1]}$. Let $U \subset M_\real \setminus |\mathcal{P}^{<k,[r-2]}|$ be a contractible open subset which is separated by $\Int_{re}(\sigma)$ into two connected components, $U_+$ and $U_-$. Associated to $\sigma$ is the wall-crossing automorphism
\[
\Theta_\sigma = \prod_{\substack{ \mathbf{w} \in \mathscr{D}(\varPhi^{<k}) \\ P_\mathbf{w} \cap U \cap \sigma \neq \emptyset}} \Theta_{\mathbf{w}}^{\sgn(n_\mathbf{w},v)},
\]
where $v \neq 0$ points into $U_+$. Results from \cite[\S 4]{kwchan-leung-ma} imply that there is a unique gauge $\facs$ which solves the equation
\begin{equation}\label{eqn:gauge_equivalent}
e^{\ad_\facs}  d e^{-\ad_\facs}
=
d_{\varPhi}
\end{equation}
and satisfies $\facs_{\vert U_-} = 0$. Moreover, this gauge is necessarily given by
\begin{equation}\label{eqn:wall_crossing_factor}
\facs
= \left\{
\begin{array}{ll}
\displaystyle \log(\Theta_\sigma) & \text{on $U_+$},\\
\displaystyle 0 & \text{on $U_-$}.
\end{array}\right.
\end{equation}
In words, $\varPhi$ behaves like a delta function supported on $\sigma$ and $\varphi$ behaves like a step function which jumps across $\sigma$. 


Let $s \in \Ker^{<k}(\varPhi)$. Since $\varPhi^{<k}_{\vert U_\pm} =0 \in \mathcal{H}^{<k,*}(U_\pm)$, we have $d(s_{\vert U_{\pm}}) = 0$. We can therefore treat $s_{\vert U_\pm}$ as a constant section over $U_\pm$, which we henceforth denote by $s_\pm  \in A^{<k}$.

Using equation \eqref{eqn:gauge_equivalent}, the condition $d_{\varPhi}(s) = 0$ is seen to be equivalent to the condition that the function $e^{-\ad_\facs}(s)$, which is defined on $U$, is $d$-flat. On the other hand, equation \eqref{eqn:wall_crossing_factor} gives
$$
e^{-\ad_\facs}(s) = \left\{
\begin{array}{ll}
\displaystyle \Theta_\sigma^{-1}(s_+) & \text{on $U_+$},\\
\displaystyle s_- & \text{on $U_-$}.
\end{array}
\right.
$$
We therefore conclude that $\Theta_\sigma(s_-) = s_+$. By applying this argument to a path $\gamma$ crossing finitely many walls generically in $\mathscr{D}(\varPhi^{<k})$, we obtain the following wall-crossing formula.

\begin{theorem}\label{thm:wall_crossing_for_sections}
	Let $s \in \Ker(d_\varPhi)$ and $Q,Q' \in M_\real \backslash \Supp(\mathscr{D}(\varPhi))$. Then
	\[
	s_{Q'} = \Theta_{\gamma,\mathscr{D}}(s_{Q})
	\]
	for any path $\gamma \subset M_\real \setminus \Joints(\mathscr{D})$ joining $Q$ to $Q'$, where $s_{Q^{\prime}}$ and $s_Q$ are restrictions of $s$ to sufficiently small neighborhoods containing $Q$ and $Q'$, respectively, and are treated as constant $A^0$-valued sections.
\end{theorem}

\subsubsection{Theta functions as elements of $\Ker(d_\varPhi)$}\label{sec:theta_function_alternative_description}

In this section we define, for each $\mathsf{m} \in \overline{M} \setminus \{0\}$, an element $\theta_\mathsf{m} \in \Ker(d_\varPhi)$. We work in the dg Lie algebra $\hat{\mathcal{H}} \oplus \hat{\mathcal{A}}[-1]$ and solve the Maurer--Cartan equation with input $\incoming + z^{\varphi(\mathsf{m})}$. We are therefore led to consider the operation $\mathbf{L}_{k,\mrtr}(\incoming + z^{\varphi(\mathsf{m})},\dots,\incoming + z^{\varphi(\mathsf{m})})$, defined as in Definition \ref{def:weighted_tree_operation} using the homotopy operator $\mathbf{H}$ of Section \ref{sec:modified_homotopy_operator}, except that we insert $z^{\varphi(\mathsf{m})}$ at the vertex attached to a marked edge $\breve{e}$ and insert $\incoming$ at unmarked edges.

Consider $\vec{\tau} : \mathfrak{M}_{\mrtr}(M_{\real}) \rightarrow \prod_{e \in \mrtr^{[1]}_{in} \setminus \{\breve{e}\}} N_{i_e}$ as in equation \eqref{eqn:tau_map}. We extend Definition \ref{def:trop_tree_orientation} to marked ribbon trees $\mrtr$ by induction along the core $\mathfrak{c}_{\mrtr} = (e_0 ,\dots,e_l)$, with associated labeled ribbon trees $\lrtr_1,\dots,\lrtr_l$ as in Definition \ref{def:marked_tree_core}. Set $\nu_{e_0} = 1$ and suppose that $\nu_{e_i}$ is defined. Consider the vertex $v_{i}$ connecting $\lrtr_{i+1}$ and $e_i$ to $e_{i+1}$. Set $\nu_{e_{i+1}} = (-1)^{|\nu_{e_{i+1}}|} \nu_{\lrtr_{i+1}} \wedge \nu_{e_i} \wedge ds_{e_{i+1}}$ if $\{\lrtr_{i+1},e_i,e_{i+1}\}$ is oriented clockwise, and $\nu_{e_{i}} = \nu_{e_i} \wedge \nu_{\lrtr_{i+1}}  \wedge ds_{e_{i+1}}$ otherwise. Write $\nu_{\mrtr}$ for $\nu_{e_{out}}$. 

\begin{lemma}
The equality $\vec{\tau}^*(d\eta_{i_{e_1}} \wedge \cdots \wedge d\eta_{i_{e_k}}) =  c \epsilon_{\mrtr} \nu_{\mrtr} + \varepsilon$ holds for some $c> 0$, where $\nu_{\mrtr}^{\vee}$ is the top polyvector field on $\real_{\leq 0}^{|\mrtr^{[1]}|}$ dual to $\nu_{\mrtr}$ and $\iota_{\nu_{\mrtr}^{\vee}} \varepsilon = 0$. In particular, when $\epsilon_{\mrtr} \neq 0$, the restriction $\vec{\tau}_{\vert \mathcal{I}_x}$ is an affine isomorphism onto its image $C(\vec{\tau},x) \subset  \prod_{e \in \mrtr^{[1]}_{in} \setminus \{ \breve{e}\}} N_{i_e}$, a top dimensional cone. 
	\end{lemma}

\begin{proof}
	We proceed by induction by splitting $\mrtr$ at $v_r$ into a labeled tree $\lrtr_1$ and a marked tree $\mrtr_2$. Assume that $\{\lrtr_1,\mrtr_2,e_{out}\}$ is oriented clockwise. The induction hypothesis gives
	$$
	\vec{\tau}^*(d\eta_{i_{e_1}} \wedge \cdots \wedge d\eta_{i_{e_k}})  = c (-1)^{|\nu_{\mrtr_2}|} \epsilon_{\mrtr_2} \nu_{\lrtr_1} \wedge \nu_{\mrtr_2}\wedge  \tau_{e_{out}}^*(n_{\lrtr_1}) + \varepsilon,
	$$
	with $\varepsilon$ as in the statement of the lemma. 	Since $\tau_{e_{out}}^*(n_{\lrtr_1}) = \sgn(\langle -m_{e_{out}},n_{\lrtr_1} \rangle) c' ds_{e_{out}}$ for some $c'>0$, this gives the desired equality.
	\end{proof}

Similar to Lemma \ref{lem:integral_convergent}, define $\alpha_{\mrtr}(x) := (-1)^l \int_{\mathcal{I}_x}\vec{\tau}^*(d\eta_{i_{e_1}} \wedge \cdots \wedge d\eta_{i_{e_k}})$ where $l$ is the length of the core $\mathfrak{c}_{\mrtr} = (e_0,\dots,e_l)$. We then have $\alpha_{\mrtr} = 0$ if $\epsilon_{\mrtr} =0$ and $\alpha_{\mrtr} \in \mathcal{W}^{-\infty}_{0}(M_{\real})$ if $P_{\mrtr} = \emptyset$. Moreover, the second statement of Lemma \ref{lem:integral_convergent} holds, after replacing $\mathcal{W}^{-\infty}_{1}(K)$ with $\mathcal{W}^{-\infty}_0(K)$. Parallel to Lemma \ref{lem:tree_support_lemma_labeled}, we have the following.

\begin{lemma}
	The equality $\mathbf{L}_{k,\mrtr}(\incoming + z^{\varphi(\mathsf{m})},\dots,\incoming + z^{\varphi(\mathsf{m})}) =  \alpha_{\mrtr} a_{\mrtr}$, holds, where $a_{\mrtr}$ is as in Definition \ref{def:marked_tree_core}.
	\end{lemma}

The argument from the proof of Lemma \ref{lem:solution_asy_support_labeled} gives the following result.

\begin{lemma}\label{lem:solution_asy_support_marked}
Let $\mrtr \in \mrtree{k}$.
	\begin{enumerate}
		\item We have $\alpha_\mrtr \in \mathcal{W}^{0}_{P_\mrtr}(M_\real) \cap \mathcal{W}^{0}_0(M_\real)$ if $\dim_\real(P_\mrtr) = r$ and $\alpha_\mrtr \in \mathcal{W}^{0}_0(M_\real)$ otherwise. In either case, $\alpha_{\mrtr \vert M_\real \setminus P_\mrtr} \in \mathcal{W}^{-\infty}_0(M_\real \setminus P_\mrtr)$.
		
		\item If $\dim_\real(P_\mrtr) = r$, then there exists a polyhedral decomposition $\mathcal{P}_\mrtr$ of $P_\mrtr$ such that $d(\alpha_\mrtr)_{\vert M_\real \setminus  |\mathcal{P}_\mrtr^{[r-1]}|} \in \mathcal{W}^{-\infty}_1(M_\real \setminus |\mathcal{P}_\mrtr^{[r-1]}|)$.
	\end{enumerate}
\end{lemma}

Motivated by the expression appearing in Theorem \ref{thm:theorem_1_modified}, define
\begin{equation}\label{eqn:theta_function_definition}
\theta_{\mathsf{m}} := \sum_{k \geq 1} \sum_{\mrtr \in \mrtree{k}} \frac{1}{2^{k-1}} \mathbf{L}_{k,\mrtr}(\incoming + z^{\varphi(\mathsf{m})},\dots,\incoming + z^{\varphi(\mathsf{m})}).
\end{equation}
By the same reasoning as was used to establish equation \eqref{eqn:ribbon_tree_non_ribbon_tree_relation}, we can write
\[
\theta_{\mathsf{m}} =\sum_{k \geq 1} \sum_{\substack{\mtr \in \mtree{k} \\ P_{\mtr} \neq \emptyset}} \frac{1}{|\Aut(\mtr)|} \alpha_{\mtr} a_{\mtr}
\]
Arguing as in Lemma \ref{lem:well_definedness_solution}, we find that $\varPhi + \theta_{\mathsf{m}} \in \hat{\mathcal{H}}^* \oplus \hat{\mathcal{A}}^*[-1]$ is a Maurer--Cartan element or, equivalently, $\varPhi \in \hat{\mathcal{H}}$ is Maurer--Cartan element and $\theta_{\mathsf{m}} \in \text{Ker}(d_{\varPhi})$.

The goal of the remainder of this section is to show that $\theta_{\mathsf{m}}(Q)= \vartheta_{\mathsf{m},Q}$, where the right hand side is the broken line theta function. We work in $\mathcal{H}^{<N,*} \oplus \mathcal{A}^{<N,*}[-1]$ for fixed $N$. Consider the scattering diagram $\mathscr{D}(\varPhi^{<N})$. Fix $\mtr \in \mtree{k}$ with $\epsilon_{\mtr} \neq 0$ and $\dim_{\real}(P_{\mtr}) = r$. Consider the core $\mathfrak{c}_{\mtr} = (e_0,\dots, e_l)$ with labeled trees $\ltr_1,\dots,\ltr_l$ attached to it at vertices $v_1 = \partial_{in}(e_1),\dots,v_l = \partial_{in}(e_l)$. In the case at hand, the map $ev : \overline{\mathfrak{M}}_{\mtr}(M_{\real},\mathscr{D}_{in},\mathsf{m}) \rightarrow P_{\mtr}$ is a diffeomorphism. Consider a polyhedral decomposition $\mathcal{P}_\mtr$ of $P_\mtr$ such that 
\begin{enumerate}
	\item Lemma \ref{lem:solution_asy_support_marked} is satisfied for $\mathcal{P}_\mtr$, and
	\item for any $\varsigma$ with $\varsigma(v_{out}) \notin |\mathcal{P}_\mtr^{[r-1]}|$, we have $\varsigma(v_i) \notin \Joints(\mathscr{D}(\varPhi^{<N}))$. 
	\end{enumerate}
	If $\varsigma$ is generic, that is satisfying $(2)$ above, then there exist walls $\mathbf{w}_j$ of $\mathscr{D}(\varPhi^{<N})$, defined by $\ltr_j$ with wall-crossing factor $\exp(\frac{c_{\mathbf{w}_j}}{|\Aut(\ltr_j)|} g_{\ltr_j})$ as in Definition \ref{def:diagram_from_solution}, such that $\varsigma(v_{j}) \in \text{int}(\mathbf{w}_j)$. Choose a non-decreasing surjection $\varkappa: \{1,\dots,l\} \rightarrow \{1,\dots,\ell\}$ such that $\varsigma(v_j) = \varsigma(v_{j'}) \in \Supp(\mathbf{w}_j)= \Supp(\mathbf{w}_{j'})$ if and only if $\varkappa(j) = \varkappa(j')$. Let $\text{PM}_i$ be the permutation group on the set $\varkappa^{-1}(i)$ and let $\text{PM}(\varkappa) = \prod_{i} \text{PM}_i$. Then, for each $\delta \in \text{PM}(\varkappa)$, we can form another marked $k$-tree $\delta(\mtr)$ by permuting the labeled trees $\ltr_j$ attached to the core. Denote by $\mtree{k}(\mtr)$ the $\text{PM}(\varkappa)$-orbit of $\mtr$ and by $\Iso(\varkappa,\mtr) = \prod_{i} \Iso_i(\varkappa,\mtr) \subseteq \text{PM}(\varkappa)$ the stabilizer subgroup of $\mtr$.

Let $\gamma$ be the restriction of $\varsigma$ to the interval corresponding to $\mathfrak{c}_{\mtr}$. Lift $\gamma$ to a broken line by setting $a_0 = z^{\varphi(\mathsf{m})}$ and, inductively, $a_{i+1} = g_{i+1} \cdot a_i$, where $g_{i+1}$ is the endomorphism of $A^{<N}$ given by
\begin{equation}\label{eqn:broken_line_construction}
g_{i+1} :=  \prod_{j \in \varkappa^{-1}(i+1)} \frac{ \sgn(\langle -m_{e_j}, n_{\ltr_j} \rangle ) c_{\mathbf{w}_j}}{|\Aut(\ltr_j)||\Iso_i (\varkappa,\mtr)|} g_{\ltr_j}.
\end{equation}
Recall that $a_{\gamma} := a_\ell$. Note that $|\Iso_i (\varkappa,\mtr)| = m_1!\cdots m_s!$ if there are $s$ distinct labeled trees in the set $\{L_j \ | \ \varkappa(j) = i \}$ which appear $m_1,\dots,m_s$ times. Hence $g_{i+1}$ is a  homogeneous factor of the product $\prod_{j \in \varkappa^{-1}(i+1)}\Theta_{\mathbf{w}_j}^{\sgn(\langle-m_{e_j}, n_{\ltr_j} \rangle)}$ appearing in Definition \ref{def:broken_lines}. Figure \ref{fig:broken_line_from_disk} illustrates the situation.

\begin{figure}[h]
	\centering
	\includegraphics[scale=0.8]{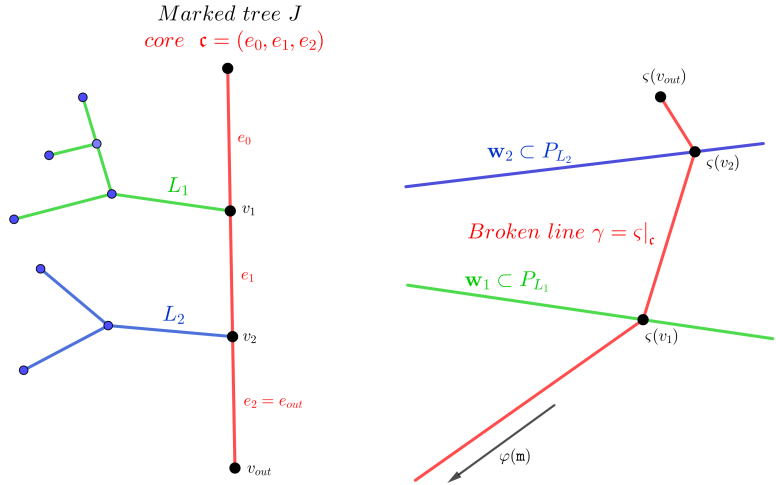}
	\caption{The relationship between the marked tree $\mtr$ and the broken line $\gamma$.}
	\label{fig:broken_line_from_disk}
\end{figure}

\begin{lemma}\label{lem:broken_line_from_theta}
	Near a generic point $Q \notin \bigcup_{\delta \in \text{PM}(\varkappa)}|\mathcal{P}^{[r-1]}_{\delta(\mtr)}|$, the equalities
\begin{equation*}
\sum_{\breve{\mtr} \in \mtree{k}(\mtr)} \frac{1}{|\Aut(\breve{\mtr})|} \alpha_{\breve{\mtr}} a_{\breve{\mtr}} = \frac{1}{| \Iso (\varkappa,\mtr)|} \sum_{\delta \in \text{PM}(\varkappa)} \frac{1}{|\Aut(\mtr)|}\alpha_{\delta(\mtr)} a_{\delta(\mtr)} = a_{\gamma}
\end{equation*}
hold. Here $a_{\gamma}$ is treated as a constant function near $Q$. 
\end{lemma}

\begin{proof}
	Notice that $|\Aut(\mtr)| =|\Aut(\breve{\mtr})| = \prod_{j=1}^l |\Aut(\ltr_j)|$ for a marked tree $\mtr$. Since $a_{\delta(\mtr)} = a_{\mtr}$, we need only show that $\sum_{\delta} \alpha_{\delta(\mtr)}$ takes the value $\epsilon_{\mtr}\prod_{j=1}^l c_{\mathbf{w}_j}$ near $Q$. Let $j_i$ be the minimal element of $\varkappa^{-1}(i)$. Split $\mtr$ by breaking the edge $e_{j_i-1}$ into two to obtain a subtree $\mtr_{i}$ with $e_{j_i-1}$ as the outgoing edge and a tree $\hat{\mtr}_{i}$ with incoming edge $e_{j_i-1}$. We then have $a_{\mtr}(Q) = a_{\hat{\mtr}_i}(Q) a_{\mtr_i}(v_{j_i})$. We will show that, for each $i$, the equality $\sum_{\delta \in \prod_{i'<i} \text{PM}_{i'}}  \alpha_{\delta(\mtr_{i})} = \epsilon_{\mtr_{i}}\prod_{j<j_i} c_{\mathbf{w}_j}$ holds in a neighborhood of $\varsigma(v_{j_i})$. We proceed by induction. Therefore we assume that $\ell = 1$ and treat the case in which all $\ltr_j$ are overlapping walls.
	
	Consider the map $\vec{\tau}_{\mathfrak{c}} = (\tau_{\mathfrak{c},1},\dots,\tau_{\mathfrak{c},l}): (\real_{\leq 0})^{|\mathfrak{c}_{\mtr} \setminus \{e_0\}|} \times M_{\real} \simeq \real_{\leq 0}^{l} \times M_{\real} \rightarrow \prod_{j} M_{\real}$ given by backward flow $\tau_{\mathfrak{c},j} :(\real_{\leq 0})^{l} \times M_{\real} \rightarrow M_{\real}^l$ by
	\[
	\tau_{\mathfrak{c},j}(\vec{s},x)= \tau_{s_j}^{e_j} \circ \cdots \circ \tau_{s_l}^{e_l} (x) = \tau^{m_{\mtr}}_{s_j+\cdots+s_l}(x) + s_{l-1} m_{\ltr_l} +  s_{l-2}(m_{\ltr_l} + m_{\ltr_{l-1}}) + \cdots + s_j(m_{\ltr_l} + \cdots + m_{\ltr_{j+1}}).
	\]
	We have $\alpha_{\mtr}(x) = (-1)^l \int_{\real_{\leq 0}^l \times \{x\}} \vec{\tau}^* \big(\alpha_{\ltr_1} \wedge\cdots \wedge \alpha_{\ltr_l} \big)$. Define a modified backward flow $\breve{\tau} = (\breve{\tau}_1,\dots,\breve{\tau}_l) : \real_{\leq 0}^l \times M_{\real} \rightarrow M_{\real}^l$ by $\breve{\tau}_j(\vec{s},x) = \tau^{m_{\mtr}}_{s_j+\cdots+s_l}(x)$. Then $\breve{\tau}$ and $\vec{\tau}$ are homotopic via $h(\vec{s},x,t) = (1-t) \vec{\tau}(\vec{s},x) + t\breve{\tau}(\vec{s},x)$. Observe that
	\[
	\int_{\partial(\real_{\leq 0}^l) \times \{x\} \times [0,1]} h^* \big( \alpha_{\ltr_1} \wedge\cdots \wedge \alpha_{\ltr_l} \big) = 0,
	\]
	since $h_*(\dd{t})$ is tangent to the wall $\mathbf{w}_j$ and $\alpha_{\ltr_j}$ is $1$-form on the normal of $\mathbf{w}_j$. Since $d\alpha_{\ltr_j} \in \mathcal{W}^{-\infty}_2(M_{\real} \setminus |\mathcal{P}_{\ltr_j}^{[r-2]}|)$ from Lemma \ref{lem:solution_asy_support_labeled} and $\text{Im}(h|_{\real_{\leq 0}^l \times W \times [0,1]}) \cap |\mathcal{P}_{\ltr_j}^{[r-2]}| =  \emptyset$ in small enough neighborhood $W$ of $Q$, we can verify that $\alpha_{\mtr}(x)$ and $(-1)^l\int_{\real_{\leq 0}^{l} \times \{x\}} \breve{\tau}^*\big(\alpha_{\ltr_1} \wedge\cdots \wedge \alpha_{\ltr_l} \big)$ differ near $Q$ by exponentially small terms in $\mathcal{W}^{-\infty}_0$. Further, the reparamaterization $s_j \mapsto s_j+\cdots+s_l$ gives
	\[
	\alpha_{\mtr}(x) = (-1)^l\int_{-\infty<s_1 \leq \cdots \leq s_l \leq 0,x} (\tau^{m_{\mtr}}_{s_1})^*(\alpha_{\ltr_1}) \wedge \cdots \wedge (\tau^{m_{\mtr}}_{s_l})^*(\alpha_{\ltr_l}).
	\]
	The permutation group on $l$ letters acts by $\alpha_{\delta(\mtr)}(x) = (-1)^l\int_{-\infty<s_{\delta(1)} \leq \cdots \leq s_{\delta(l)} \leq 0,x} (\tau^{m_{\mtr}}_{s_1})^*(\alpha_{\ltr_1}) \wedge \cdots \wedge (\tau^{m_{\mtr}}_{s_l})^*(\alpha_{\ltr_l})$, using which we compute $\sum_{\delta} \alpha_{\delta(\mtr)}(x) =(-1)^l \prod_{j}  \big( \int_{-\infty <s_j \leq 0,x} (\tau^{m_{\mtr}}_{s_j})^*(\alpha_{\ltr_j}) \big) $. Finally, we have
	\[
	\int_{-\infty <s_j \leq 0,x} (\tau^{m_{\mtr}}_{s_j})^*(\alpha_{\ltr_j}) = -\sgn(\langle -m_{\mtr}, n_{\ltr_j} \rangle) c_{\mathbf{w}_j},
	\]
	as in Theorem \ref{thm:theorem_1_modified}.
	\end{proof}

For a generic point $Q$, let $\mtree{k}(Q,\mathsf{m}) \subset \mtree{k}$ be the set of marked trees $\mtr$ with $Q \in P_{\mtr}$ and marked edge $\varphi(\mathsf{m})$. For any two $J,J' \in \mtree{k}(Q,\mathsf{m})$, notice that either $\mtree{k}(J) = \mtree{k}(J')$ or $\mtree{k}(J) \cap \mtree{k}(J') = \emptyset$. It follows that there is a decomposition $\mtree{k}(Q,\mathsf{m}) = \sqcup_{\mtr \in \mathbb{L}}\mtree{k}(\mtr)$ such that each $\mtree{k}(\mtr)$ corresponds to a unique broken line $\gamma$ via Lemma \ref{lem:broken_line_from_theta}. Conversely, given a broken line $\gamma$ with ends $(Q,\mathsf{m})$, one can construct a marked tree $\mtr \in \mtree{k}(Q,\mathsf{m})$ with the restriction of $\varsigma$ to the core $\mathfrak{c}_{\mtr}$ being $\gamma$, and labeled trees $\ltr_1,\dots,\ltr_l$ attached to $\mathfrak{c}_{\mtr}$ such that the relation \eqref{eqn:broken_line_construction} holds. As a conclusion, we have the following theorem.

\begin{theorem}\label{thm:theta_function_comparsion}
	For generic $Q \in M_{\real} \setminus \Supp(\mathscr{D}(\varPhi))$, we have
	$
	\theta_{\mathsf{m}}(Q) = \vartheta_{\mathsf{m},Q},
	$
	where $\theta_{\mathsf{m}}(Q)$ is the value of $\theta_{\mathsf{m}}$ at $Q$.
\end{theorem}

\subsection{Hall algebra scattering diagrams}\label{sec:hall_algebra}

We investigate non-tropical analogues of the results of Sections \ref{sec:tropical_counting_thm}-\ref{sec:theta_function_from_MC}. Our main case of interest is the Hall algebra scattering diagrams of \cite{bridgeland2017scattering}.

\subsubsection{Motivic Hall algebras}
\label{sec:motivicHall}

We recall the definition of Joyce's motivic Hall algebra. While Hall algebra scattering diagrams are the main example of non-tropical scattering diagrams, we will not use anything technical about Hall algebras. We will therefore be brief. The reader is referred to \cite{joyce2007, bridgeland2012} for details. See also \cite[\S \S 4-5]{bridgeland2017scattering}.

Let $\mathsf{Q}$ be a quiver\footnote{For simplicity, we restrict attention to the case of trivial potential.} with finite sets of nodes $\mathsf{Q}_0$ and arrows $\mathsf{Q}_1$. Let $M^{\oplus} = \mathbb{Z}_{\geq 0} \mathsf{Q}_0$ be the monoid of dimension vectors. For each $d \in M^{\oplus}$, denote by $R_d = \prod_{\alpha \in \mathsf{Q}_1} \Hom_{\mathbb{C}}(\mathbb{C}^{d_{s(\alpha)}}, \mathbb{C}^{d_{t(\alpha)}})$ the affine variety of complex representations of $\mathsf{Q}$ of dimension vector $d$. The reductive group $\mathsf{GL}_d = \prod_{i \in \mathsf{Q}_0} \mathsf{GL}_{d_i}(\mathbb{C})$ acts on $R_d$ by change of basis. The quotient stack $\mathcal{M}_d:= R_d \slash \mathsf{GL}_d$ is the moduli stack of representations of dimension vector $d$. Set $\mathcal{M} = \bigsqcup_{d \in M^{\oplus}} \mathcal{M}_d$.

Similarly, given $d_1, d_2 \in M^{\oplus}$, let $\mathcal{M}_{d_1,d_2}$ be the moduli stack of short exact sequences $0 \rightarrow U_1 \rightarrow U_2 \rightarrow U_3 \rightarrow 0$ of representations in which $U_1$ and $U_3$ have dimension vector $d_1$ and $d_2$, respectively. There is a canonical correspondence
\begin{equation}
\label{eq:HallCorrespondence}
\mathcal{M}_{d_1} \times \mathcal{M}_{d_2} \xleftarrow[]{\pi_1 \times \pi_3} \mathcal{M}_{d_1,d_2} \xrightarrow[]{\pi_2} \mathcal{M}_{d_1+d_2},
\end{equation}
a short exact sequence being sent by $\pi_1 \times \pi_3$ to its first and third terms and by $\pi_2$ to its second term. The map $\pi_1 \times \pi_3$ is of finite type while $\pi_2$ is proper and representable.

Let $K_0(\mathsf{St} \slash \mathcal{M}) = \bigoplus_{d \in M^{\oplus}} K_0(\mathsf{St} \slash \mathcal{M}_d)$, the Grothendieck ring of finite type stacks with affine stabilizers over $\mathcal{M}$. Push-pull along the correspondence \eqref{eq:HallCorrespondence} gives $K_0(\mathsf{St} \slash \mathcal{M})$ the structure of a $M^{\oplus}$-graded associative algebra, the motivic Hall algebra of $\mathsf{Q}$. The augmentation ideal $\mathfrak{h}_{\mathsf{Q}}$ of $K_0(\mathsf{St} \slash \mathcal{M})$, with its commutator bracket, is a $M^+$-graded Lie algebra, the motivic Hall--Lie algebra.

In the setting of scattering diagrams, it is convenient to use a specialization of $K_0(\mathsf{St} \slash \mathcal{M})$. Write $\mathsf{St}$ in place of $\mathsf{St}\slash_{\Spec(\mathbb{C})}$. Cartesian product with $\Spec(\mathbb{C})$ makes $K_0(\mathsf{St} \slash \mathcal{M})$ into a $K_0(\mathsf{St})$-module. Let $\Upsilon: K_0(\mathsf{St}) \rightarrow \mathbb{C}(t)$ be the unique ring homomorphism  which sends the class of a smooth projective variety to the Poincar\'{e} polynomial of its singular cohomology with complex coefficients. Then $K_0(\mathsf{St}\slash \mathcal{M}) \otimes_{K_0(\mathsf{St})} \mathbb{C}(t)$ becomes a $\mathbb{C}(t)$-algebra, the Hall algebra of stack functions \cite{joyce2007}. Its augmentation ideal $\mathfrak{h}_{\mathsf{Q}}^{\Upsilon}$ is again a $M^+$-graded Lie algebra. The Hall algebra scattering diagram of \cite{bridgeland2017scattering} takes values in the (non-tropical) Lie algebra $\mathfrak{h}_{\mathsf{Q}}^{\Upsilon}$.  

\subsubsection{Non-tropical Maurer--Cartan solutions}\label{sec:non_tropical_diagram}

We begin by describing an abstract setting in which scattering diagrams can be defined without the tropical assumption. We largely follow \cite{bridgeland2017scattering}, introducing modifications as needed. Let $\mathfrak{h}$ be a not-necessarily tropical $M_{\sigma}^+$-graded Lie algebra.

We henceforth consider scattering diagrams in $N_\real$ instead of $M_\real$. The relevant modification of Definition \ref{wall} is as follows. 

\begin{definition}\label{def:non_tropical_walls}
		A {\em wall} $\mathbf{w}$ in $N_\real$ is a pair $(P, \Theta)$ consisting of a codimension one closed convex rational polyhedral subset $P \subset N_\real$ and an element $\Theta \in \hat{G}_{P^{\perp}} := \exp( \hat{\mathfrak{h}}_{P^{\perp}})$, where $P^{\perp}$ consists of those $m \in M$ which are perpendicular to any $n \in N_\real$ which is tangent to $P$.
\end{definition}

We also require a modified definition of scattering diagrams.

\begin{definition}\label{def:non_tropical_scattering_diagram}
A {\em scattering diagram} $\mathscr{D}$ consists of data $\{\mathscr{D}^{<k}\}_{k \in \inte_{>0}}$, where $\mathscr{D}^{<k} = \left\{ ( P_\alpha, \Theta_\alpha) \right\}_{\alpha}$ is a finite collection of walls with $\dim_\real(P_{\mathbf{w}_1} \cap P_{\mathbf{w}_2})<r-1$ for any two distinct walls $\mathbf{w}_1$, $\mathbf{w}_2$. The diagrams $\mathscr{D}^{<k+1}  \;(\textnormal{mod } \mathfrak{h}^{\geq k})$ and $\mathscr{D}^{<k}$ are required to be equal up to refinement by taking polyhedral decompositions of the polyhedral subsets of $\mathscr{D}^{<k}$ and by adding walls with trivial wall-crossing automorphisms.
\end{definition}

We will henceforth assume that each $P_{\alpha}$ is rational polyhedral cone. In this case Definition \ref{def:non_tropical_scattering_diagram} agrees with the notion of a $\mathfrak{h}$-complex from \cite[\S 2]{bridgeland2017scattering}. Following \cite{bridgeland2017scattering}, fix an ordered basis $(f_1,\dots, f_r)$ of $M$, thereby identifying $M$ with $\mathbb{Z}^r$. We take $\sigma = \bigoplus_{i=1}^r \mathbb{Z}_{\geq 0} \cdot f_r$ to be the standard cone and consider $\mathscr{D}_{in} = \{ \mathbf{w}_i  = (P_i,\Theta_i)\}_{1\leq i \leq r}$ with $P_i = f_i^\perp \subset N_\real$. Write $g_i :=\log(\Theta_i) = \sum_{j>1} g_{ji}$ with $g_{ji} \in \mathfrak{h}_{j f_i}$. We assume that $[g_{j_1 i},g_{j_2 i}] = 0$ for each initial wall.

\begin{example}
Let $\mathsf{Q}$ be a quiver without edge loops. For any $i \in \mathsf{Q}_0$ and $k \in \mathbb{Z}_{\geq 0}$, the stack $\mathcal{M}_{k i}$ is isomorphic to the classifying stack $B \mathsf{GL}_k(\mathbb{C})$. Let $\mathcal{M}_{\langle i \rangle} = \bigsqcup_{k \geq 0} \mathcal{M}_{k i}$. Then the element $\Theta_i = [\mathcal{M}_{\langle i\rangle} \rightarrow \mathcal{M}]$ of the dimension-completed motivic Hall algebra satisfies the above assumptions.
\end{example}

Using the affine structure on $N_\real$, we can again define the dg Lie algebras $\mathcal{H}^{*}$, $\hat{\mathcal{H}}^{*}$ and $\mathcal{H}^{<k,*}$. The discussions in Section \ref{sec:tropical_dgla} continue to hold without the tropical assumption on $\mathfrak{h}$. Let $\incoming = \sum_{i=1}^r \incoming^{(i)}$ with $\incoming^{(i)}$ as in equation \eqref{eqn:incoming_for_each_wall}. To define $\mathbf{H}_m$ via equation \eqref{eqn:modified_homotopy_operator}, we must first choose a suitable direction $v^m \in N_{\mathbb{Q}} \setminus \{0\}$ along which to define the flow $\tau^{m}$. For that purpose, fix a line $\lambda : \real \rightarrow N_\real$ of slope $(a_1,\dots,a_r) \in \mathbb{R}^r$ such that $\lambda(0)=(-1,\dots,-1)$ and $\lambda(1)$ lies in the dual cone $\text{int}(\sigma^{\vee})$. We assume that $0<a_1 <\dots<a_r$ and that $\{a_1, \dots, a_r\}$ are algebraically independent over $\mathbb{Q}$.


Let $m \in M_\sigma^+$. Consider the polyhedral decomposition $\mathfrak{P}_m$ of the hyperplane $m^{\perp}$ induced by the finite hyperplane arrangement whose hyperplanes are of the form $m_1^\perp \cap m^\perp$, where $m_1 \nparallel m \in M_\sigma^+$ and $m_1 + m_2 = m$ for some $m_2\in M_\sigma^+$. By construction, $\lambda \cap m^{\perp}$ is contained in $\Int_{re}(-\beth)$ for some maximal cone $\beth \in \mathfrak{P}_m^{[r-1]}$. If $m = k f_i$ for some $k >0$ and $1\leq i\leq r$, then we set $v^{m} = - k f_i^\vee$. Otherwise, we take $v^m \in \Int_{re}(\beth)$ to be a rational point.

With the above notation, we obtain operators\footnote{Labeled (ribbon) $k$-trees are defined as in Definitions \ref{label_tree_def} and \ref{def:weighted_ribbon_k_tree} using the walls of Definition \ref{def:non_tropical_walls}.} $\mathbf{L}_{k,\lrtr}(\incoming,\dots,\incoming)$ as in Section \ref{sec:modified_homotopy_operator}, and hence also $\mathbf{L}_{k,\ltr}(\incoming,\dots,\incoming)$ by equation \eqref{eqn:ribbon_tree_non_ribbon_tree_relation}. Definition \ref{def:tropical_curve} is modified to talk about tropical disks in $(N_\real, \mathscr{D}_{in})$ of type $\ltr$, which are proper maps $\varsigma : |\ltr_{\vec{s}}| \rightarrow N_\real$ whose slope at an edge $e \in \bar{\ltr}^{[1]}$ is $v^{m_e}$. The moduli space $\mathfrak{M}_\ltr(N_\real,\mathscr{D}_{in})$ is defined accordingly and $P_\ltr := ev (\overline{\mathfrak{M}}_\ltr(N_\real,\mathscr{D}_{in}))$ is now a subset of $m_\ltr^{\perp}$. Lemma \ref{lem:tau_map_iso} holds after replacing $n_\lrtr$ with $m_{\lrtr}$, with the caveat that we can only conclude that $c \neq 0$, instead of $c>0$. 

With $\alpha_{\lrtr}$ defined as in Lemma \ref{lem:integral_convergent}, parts \textit{(1)} and \textit{(2)} of Lemmas \ref{lem:integral_convergent} and \ref{lem:solution_asy_support_labeled} hold by the same argument (after replacing $n_\lrtr$ with $m_\lrtr$ and $M_{\real}$ with $N_{\real}$). Lemma \ref{lem:iterated_integral_constant} is again valid, except that the sign of $c_{\lrtr,\sigma} \neq 0$ cannot be determined. By Lemma \ref{lem:tree_support_lemma_labeled}, we have $\mathbf{L}_{k,\lrtr}(\incoming,\dots,\incoming) = \alpha_{\lrtr} g_{\lrtr}$ and, since it is independent of ribbon structure, we can write $\mathbf{L}_{k,\ltr}(\incoming,\dots,\incoming) = \alpha_\ltr g_\ltr$.

\begin{lemma}\label{lem:hall_algebra_well_definedness_solution}
The sum
	$$\varPhi := \sum_{k \geq 1} \sum_{\substack{\lrtr \in \lrtree{k}\\ \mathfrak{M}_{\lrtr}(N_\real,\mathscr{D}_{in}) \neq \emptyset}} \frac{1}{2^{k-1}} \alpha_\lrtr g_\lrtr = \sum_{k \geq 1} \sum_{\substack{\ltr \in \ltree{k}\\ \mathfrak{M}_{\ltr}(N_\real,\mathscr{D}_{in}) \neq \emptyset}} \frac{1}{|\Aut(\ltr)|} \alpha_\ltr g_\ltr$$ is a Maurer--Cartan element of $\hat{\mathcal{H}}^*$. 
	\end{lemma}

\begin{proof}
The equality in the statement of the lemma holds by the same reasoning as in the tropical case. So we focus on proving that $\varPhi$ is a Maurer--Cartan element.

	Fix $k \in \mathbb{Z}_{>0}$ and work in $\mathcal{H}^{<k,*}$. Let $K\subset N_\real$ be a compact subset. As in the proof of Lemma \ref{lem:well_definedness_solution}, we must show that for sufficiently large $L >0$ we have $\mathbf{P}_L[\varPhi_L,\varPhi_L] = 0$ on $K$, where $\varPhi_L := \incoming - \mathbf{H}_L [\varPhi_L,\varPhi_L]$ and $\mathbf{P}_{L,m}(\beta) := (\tau^m_{-L})^*(\beta)$. Consider labeled ribbon trees $\lrtr_1,\lrtr_2$ with associated terms $\alpha_{\lrtr_i,L} g_{\lrtr_i}$, where $\alpha_{\lrtr_i,L} \in \mathcal{W}^1_{P_i}(N_\real) \cap \mathcal{W}^1_1(N_\real)$ and $P_i = m_{\lrtr_i}^\perp$. We can assume that $m_{\lrtr_1} \nparallel m_{\lrtr_2}$, as otherwise $\alpha_{\lrtr_1,L} \wedge \alpha_{\lrtr_2,L} \in \mathcal{W}^{-\infty}(N_\real)$ and hence $[\alpha_{\lrtr_1,L} g_{\lrtr_1}, \alpha_{\lrtr_2,L} g_{\lrtr_2}] = 0 \in \mathcal{H}^{<k,*}$. Joining the trees $\lrtr_1,\lrtr_2$ to obtain $\lrtr$, the transversal intersection $P_{\lrtr_1} \cap P_{\lrtr_2}\subset P_\lrtr = m_{\lrtr}^{\perp}$ is contained in the $(r-2)$-dimensional strata of the polyhedral decomposition $\mathfrak{P}_{m_\lrtr}$. By our choice of $v^{m_{\lrtr}}$, the flow $\tau^{m_{\lrtr}}$ is not tangent to $P_{\lrtr_1} \cap P_{\lrtr_2}$. As in proof of Lemma \ref{lem:well_definedness_solution}, we can therefore choose $L$ sufficiently large so that $(\tau_{-L}^{m_{\lrtr}})^*(\alpha_{\lrtr_1} \wedge \alpha_{\lrtr_2}) = 0$ on $K$.
	\end{proof}

Observe that, by the construction of $v^{m}$, the line $\lambda$ is disjoint from each $P_\lrtr$.

Having proved Lemma \ref{lem:hall_algebra_well_definedness_solution}, the conclusion of Theorem \ref{thm:theorem_1_modified} follows after replacing $M_\real$ with $N_\real$. 

\subsubsection{Consistent scattering diagrams from non-tropical Maurer--Cartan solutions}\label{sec:hall_algebra_MC_relation}

We establish the relation between Maurer--Cartan solutions $\varPhi \in \hat{\mathcal{H}}^*$ and consistent scattering diagrams. By construction $\varPhi = \varprojlim_k \varPhi^{<k}$, where $\varPhi^{<k} = \sum_{P} \alpha_{P} g_P$ is a finite sum indexed by polyhedral subsets $P$ of $N_\real$. From the discussion in Section \ref{sec:non_tropical_diagram}, we have $g_P \in \mathfrak{h}^{<k}$ and $\alpha_P \in \mathcal{W}^1_{\tilde{P}}(N_\real) \cap \mathcal{W}^1_1(N_\real)$, where $\tilde{P} \subset N_{\mathbb{R}}$ is a codimension one polyhedral subset containing $P$ and $\alpha_{P \vert N_\real \setminus P} \in \mathcal{W}^{-\infty}(N_\real \setminus P)$. Similarly to Section \ref{sec:wall_crossing_flat_sections}, consider a polyhedral decomposition $\mathcal{P}^{<k}$ of $\bigcup_{0 \neq g_P \in \mathfrak{h}^{<k}} P$ such that, for every $0 \leq l \leq r-1$ and $\sigma \in \mathcal{P}^{<k,[l]}$, we have $\sigma \subset P$ for some $\dim_\real(P) = r-1$ and $P \cap \sigma$ is a facet of $\sigma$ for every $P$ with $0 \neq g_P \in \mathfrak{h}^{<k}$.

Let $U$ be a convex open set such that $U \cap \tau = \emptyset$ whenever $\tau \neq \sigma \in \mathcal{P}^{<k}$ and $U \setminus \sigma = U_+ \sqcup U_-$ is a decomposition into connected components. Since $U$ is contractible, $H^1(\mathcal{H}^{<k,*}(U),d) = 0$, whence $\varPhi^{<k}|_{U}$ is gauge equivalent to $0$, that is, $e^{\text{ad}_{\facs}}  d e^{-\text{ad}_{\facs}} = d_{\varPhi^{<k}}$ on $U$, where $\facs$ satisfies $\facs_{\vert U_-} = 0$. We will use a homotopy operator $\hat{\mathsf{I}}$ acting on $\mathcal{H}^{<k,*}$ to solve for $\facs$. Assume that we are given a chain of affine subspaces $U_{\bullet}$ of $U$, as in Section \ref{sec:integral_lemma}, such that $v_{1}$ is transversal to $\sigma$ and $U_{1,+}$, the half space over which $v_{1}$ points inwards, contains $U_{+} \cup \sigma $. See Figure \ref{fig:Open_subset_relation}. Such a choice yields a homotopy operator $ I : \mathcal{W}^{0}_*(U) \rightarrow \mathcal{W}^{0}_{*-1}(U)$ by equation \eqref{eqn:generalized_integral_operator} which, by Lemma \ref{lem:integral_lemma_modified}, descends to $\mathcal{W}^0_*(U)/\mathcal{W}^{-1}_*(U)$. As in Definition \ref{pathspacehomotopy}, we then obtain a homotopy operator $\hat{\mathsf{I}}$, defined using (the $m$-independent) $I$ in place of $\mathsf{H}_m$, and operators $\hat{\mathsf{P}}$ and $\hat{\iota}$ on $\hat{\mathcal{H}}^*(U)$.

\begin{figure}[h]
	\centering
	\includegraphics[scale=0.6]{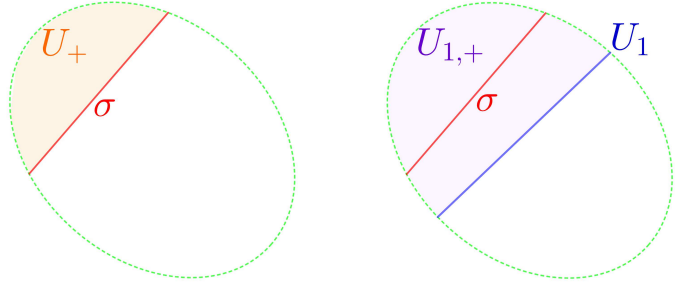}
	\caption{The sets $U_{1,+}$ and $U_+$.}
	\label{fig:Open_subset_relation}
\end{figure}

Arguments of \cite[\S 4]{kwchan-leung-ma} show that the unique gauge satisfying $\hat{\mathsf{P}}(\facs) = 0$ is given by $\facs = \varprojlim_k \facs^{<k}$, where $\facs^{<k} \in \mathcal{H}^{<k,0}$ is constructed inductively by
\[
\facs^{<(k+1)} = -\hat{\mathsf{I}} \Big(\varPhi + \sum_{l\geq 0 }\frac{\ad_{\facs^{<k}}^l}{(l+1)!} d\facs^{<k} \Big).
\]
Using Lemmas \ref{lem:support_product} and \ref{lem:integral_lemma_modified}, we inductively obtain
\begin{align}
\facs^{<k} &\in  \Big(\frac{\mathcal{W}^{0}_{\overline{U}_+}(U) \cap \mathcal{W}^{0}_0(U) + \mathcal{W}^{-1}_0(U) }{ \mathcal{W}^{-1}_0(U)} \Big) \otimes_{\mathbb{C}} \mathfrak{h}^{<k}_{\sigma^\perp}, \\
\frac{\ad_{\facs^{<s}}^l}{(l+1)!} d\facs^{<s} &\in \Big( \frac{\mathcal{W}^{1}_{\sigma}(U) \cap \mathcal{W}^{1}_1(U)+ \mathcal{W}^{0}_1(U)} { \mathcal{W}^{0}_1(U)} \Big) \otimes_{\mathbb{C}} \mathfrak{h}^{<k}_{\sigma^\perp} \label{eqn:non_tropical_solving_gauge}
\end{align}
for all $s \leq k$ and $l \geq 0$, where $\sigma^\perp $ is the subspace perpendicular to the tangents of $\sigma$.

Setting $l=0$ in equation \eqref{eqn:non_tropical_solving_gauge} gives $(d \facs^{<k})_{\vert U_+} = 0$. This suggests that $\lim_{\hp \rightarrow 0} \facs^{<k}|_{U_+}$ be treated as a (constant) element of $\mathfrak{h}^{<k}$. Denote this element by $\log(\Theta_\sigma^{<k})$. Note that $\log(\Theta_\sigma^{<k})$ is independent of $U$, as follows from the uniqueness of $\facs$ on the common intersection of two such open sets. 

\begin{remark}\label{rem:tropical_vs_non_tropical}
	When $\mathfrak{h}$ is tropical, we have $
	\varPhi_{\vert U} = \sum_{k \geq 1} \sum_{\substack{\ltr \in \ltree{k}\\ P_\ltr \cap U \neq \emptyset}} \frac{1}{|\text{Aut}(\ltr)|} \alpha_{\ltr} g_\ltr
	$ with $g_\ltr \in \mathfrak{h}_{m_\ltr,n_\ltr}$. This forces $[g_{\ltr_1},g_{\ltr_2} ]$ to vanish whenever $P_{\ltr_i} \cap U \neq \emptyset$, $i=1,2$, because $\dim_\real(P_{\ltr_i}) = r-1$ and $P_{\ltr_i} \cap U = \sigma \cap U$ by our choice of polyhedral decomposition $\mathcal{P}^{<k}$. The normals $n_{\ltr_1}$ and $n_{\ltr_2}$ to $P_{\ltr_1}$ and $P_{\ltr_2}$ are parallel while the vectors $m_{\ltr_i}$ are tangent to $P_{\ltr_i} \cap U = \sigma \cap U$, $i=1,2$. This gives $\langle m_{\ltr_j} ,n_{\ltr_i} \rangle = 0$ for $i,j =1,2$. By an induction argument, this implies $\frac{\ad_{\facs^{<s}}^l}{(l+1)!} d\facs^{<s} = 0 \in \mathcal{H}^*(U)$ for all $s,l$. If $\mathfrak{h}$ is not tropical, then $\frac{\ad_{\facs^{<s}}^l}{(l+1)!} d\facs^{<s}$ need not vanish and so contributes to the recursive construction of $\facs^{<k}$.
\end{remark}

\begin{definition}\label{def:non_tropical_diagram_from_solution}
		Let $\varPhi$ be as in Theorem \ref{thm:theorem_1_modified}. For each $k \in \inte_{>0}$, let $\mathscr{D}(\varPhi^{<k})$ be the scattering diagram with walls $\mathbf{w}_{\sigma} = (P_{\sigma} = \sigma,\Theta_{\sigma}^{<k})$ indexed by the maximal cells $\sigma \in \mathcal{P}^{<k,[r-1]}$.  
\end{definition}

Denote by $\mathscr{D}(\varPhi)$ the scattering diagram determined by $\{\mathscr{D}(\varPhi^{<k})\}_{k \in \inte_{>0}}$.

\begin{theorem}
\phantomsection
\label{thm:consistent_scattering_nontropical}

\begin{enumerate}
\item The diagram $\mathscr{D}(\varPhi)$ is consistent and the path ordered product $\Theta_{\lambda_{\vert [0,1]}, \mathscr{D}(\varPhi)}$ is equal to the product $g:=\Theta_1 \cdots \Theta_r$ of the wall-crossing factors of the initial walls.

\item The scattering diagram $\mathscr{D}(\varPhi)$ is equivalent to the scattering diagram (or $\mathfrak{h}$-complex) $\mathscr{D}(g)$ constructed in \cite[Lemma 3.2]{bridgeland2017scattering}. 
\end{enumerate}
\end{theorem}

\begin{proof}
	The proof of Proposition \ref{prop:consistence_from_solution} carries over with minor changes to show that $\mathscr{D}(\varPhi)$ is consistent. As noted after Lemma \ref{lem:hall_algebra_well_definedness_solution}, the path $\lambda_{\vert [0,1]}$ does not intersect any walls of $\mathscr{D}(\varPhi)$ which are supported on $P_\lrtr$. By construction, $\lambda_{\vert [0,1]}$ crosses the initial walls $\mathbf{w}_r,\dots,\mathsf{w}_1$ consecutively. The assumption $[g_{j_1 i},g_{j_2 i}] = 0$ ensures that the wall-crossing factor $\Theta_i$ from Definition \ref{def:non_tropical_diagram_from_solution} agrees with the wall-crossing factor of $\mathbf{w}_i$ determined by the gauge $\facs$, as constructed above; see also Remark \ref{rem:tropical_vs_non_tropical}. The path ordered product is therefore as stated. The equivalence between $\mathscr{D}(\varPhi)$ and $\mathscr{D}(g)$ is achieved by using \cite[Proposition 3.3]{bridgeland2017scattering} to show that $\mathscr{D}(\varPhi^{<k})$ and $\mathscr{D}(g^{<k})$ are equivalent for each $k \in \inte_{>0}$.
\end{proof}

\begin{example}
If the quiver $\mathsf{Q}$ is acyclic or, more generally, the quiver with potential $(\mathsf{Q},W)$ is genteel in the sense of \cite[\S 11.5]{bridgeland2017scattering} (and the motivic Hall algebra is modified so as to include the potential), then the consistent completion $\mathscr{D}(\varPhi)$ of the initial scattering diagram $\mathscr{D}_{in}$, with $\Theta_i = [\mathcal{M}_{\langle i\rangle} \rightarrow \mathcal{M}]$, is the Hall algebra scattering diagram of \cite{bridgeland2017scattering}. In the non-genteel case, additional walls must be added to $\mathscr{D}_{in}$ so as to recover the Hall algebra scattering diagram.
\end{example}

\subsubsection{Non-tropical theta functions}\label{sec:hall_algebra_theta_function}

Following \cite{bridgeland2017scattering}, we consider a $M_{\sigma} \oplus N$-graded algebra $B =\bigoplus_{(m, n) \in M_\sigma \times N} B_{m,n}$ together with a $M_{\sigma}$-graded $\mathfrak{h}$-action by derivations so that $\mathfrak{h}_m \cdot B_{0, n} = 0$ whenever $\langle m,n \rangle = 0$. We assume that, for each $n\neq 0$, there is a distinguished element $z^{n} \in B_{0,n}$ which we use to identify $B_{0,n}$ with $\comp \cdot z^{n}$. As in Definition \ref{def:tropical_dga}, define a  (not-necessarily graded commutative) dg algebra $\mathcal{B}^*(U)$. The dg Lie algebra $\mathcal{H}^*(U)$ acts on $\mathcal{B}^*(U)$, so we can again talk about theta functions as elements in $\Ker(d_{\Phi})$. Define the flow $\tau^{m,n}$ by choosing $v^{m,n} \in (-\sigma^{\vee} \cap N) \setminus \bigcup_{\substack{m_1,m_2 \neq 0\\ m_1+m_2=m}} m_1^{\perp}$. This defines the propagator $\mathbf{H}_{m,n}$ on $\mathcal{B}^*_{m,n}$. 

Set $N_{\sigma}^+ = \{ n \in N \mid \langle m, n \rangle \geq 0 \; \forall \, m \in M_{\sigma}^+ \} \setminus \{0\}$. For each $n \in N_{\sigma}^+$, define $\theta_n$ by equation \eqref{eqn:theta_function_definition}, where an edge $e_j$ in the core $\mathfrak{c}_{\mtr}$ is labeled by a pair $(m_{e_j},n) \in M_{\sigma}^+ \times N_{\sigma}^+$ (instead of by $m_{e_j}$, as described after Definition \ref{def:weighted_tree_tropical}) and the incoming edge $\breve{e}$ of $\mtr$ is labeled by $n$. We argue that $\varPhi + \theta_n \in \hat{\mathcal{H}} \oplus \hat{\mathcal{B}}[-1]$ is a Maurer--Cartan element by showing that $\mathbf{P}_{L}[\varPhi_{L} + \theta_{n,L}, \varPhi_{L}+\theta_{n,L}] = 0$ on a compact subset $K\subset N_{\real}$ for sufficiently large $L$. Here $\varPhi_L+ \theta_{n,L}$ is defined using a cut-off propagator. It suffices to consider a labeled ribbon tree $\lrtr$ and a marked ribbon tree $\mrtr$ with $g_{\lrtr}  \cdot a_{\mrtr} \neq 0$. Join $\lrtr$ and $\mrtr$ to form $\hat{\mrtr}$. Then $v^{m_{\hat{\mrtr}},n}$ is not tangent to $P_{\lrtr} \supset P_{\lrtr} \cap P_{\mrtr}$ and hence the proof of Lemma \ref{lem:hall_algebra_well_definedness_solution} shows that $(\tau^{m_{\hat{\mrtr}},n}_{-L})^* (\alpha_{\lrtr} \wedge \alpha_{\mrtr}) = 0$ on $K$. It follows that $\theta_n \in \text{Ker}(d_{\varPhi})$. The argument from Theorem \ref{thm:wall_crossing_for_sections} then shows that $\theta_n$ satisfies the wall-crossing formula. 

\begin{prop}
	For any path $\gamma \subset N_\real \setminus \Joints(\mathscr{D}(\varPhi))$ from $Q$ to $Q^{\prime}$, the wall-crossing formula
	\begin{equation}\label{eq:HallThetaWCF}
	\theta_n(Q^{\prime}) = \Theta_{\gamma,\mathscr{D}(\varPhi)} (\theta_{n}(Q))
	\end{equation}
	holds.
		
	Moreover, if $\mathfrak{h} = \mathfrak{h}_{\mathsf{Q}}^{\Upsilon}$ is the motivic Hall--Lie algebra, then the Hall algebra theta function $\vartheta_{n,Q}$, as defined in \cite[\S 10.5]{bridgeland2017scattering}, is related to $\theta_n$ by the formula $\vartheta_{n,Q} = \theta_n(Q)$.
	\end{prop}

\begin{proof}
	It remains to prove the final statement. Since $\theta_n$ and $\vartheta_{n,Q}$ satisfy the wall-crossing formula, it suffices to show that $\theta_{n}(Q) = z^n$ for $Q \in \text{int}(\sigma^\vee)$, since this condition characterizes $\vartheta_{n,Q}$. Note that there are no walls in $\text{int}(\sigma^{\vee}) \cup -\text{int}(\sigma^{\vee})$, as all walls lie in a hyperplane of the form $m^{\perp}$ for some $m \in \sigma$. Consider $\mtr \in \mtree{k}$ with core $\mathfrak{c}_{\mtr} = (e_0,\dots,e_l)$ and $v_1 = \partial_{in}(e_1),\dots,v_l = \partial_{in}(e_l)$. Let $\varsigma$ be a marked tropical disk with $\varsigma(v_{\text{out}}) = Q$. Then $\varsigma(v_i) \notin \text{int}(\sigma^{\vee})$ since $\varsigma(v_i)$ lies on a wall and $\varsigma^{\prime}$ lies in $-\sigma^{\vee}$ when restricted to $\mathfrak{c}_{\mtr}$, by the choice of $v^{m,n}$. Therefore we cannot have $\varsigma(v_{\text{out}}) = Q$ unless $\mtr^{[0]} = \emptyset$, which corresponds to the trivial marked disk with $a_{\mtr} = z^n$.
	\end{proof}

\subsubsection{Hall algebra scattering diagrams and theta functions for acyclic quivers}\label{sec:hall_algebra_acyclic_quiver}

In Sections \ref{sec:non_tropical_diagram} and \ref{sec:hall_algebra_theta_function}, there was no canonical choice of the vectors $v^{m}$ and $v^{m,n}$. In this section we take $\mathfrak{h} = \mathfrak{h}_{\mathsf{Q}}^{\Upsilon}$ for an {\em acyclic} quiver $\mathsf{Q}$, where canonical choices can be made. Let $\omega : M \times M \rightarrow \mathbb{Z}$ be the skew-symmetrized Euler form of $\mathsf{Q}$, so that $\omega(f_i,f_j) = a_{ji}-a_{ij}$ with $a_{ij}$ the number of arrows from $i$ to $j$. Relabeling if necessary, we can arrange that $a_{ji} =0$ whenever $i < j$. We will assume that $\omega$ is non-degenerate. Define $\sem: M \rightarrow N$ so that $\langle m', \sem(m) \rangle =\omega (m',m)$. We are therefore in the setting of \cite{gross2018canonical} (but without the tropical assumption). The following conditions hold:
\begin{enumerate}
	\item The inequality $\omega(f_i,f_j)\leq 0$ holds whenever $i<j$.
	\item If a subset $\mathtt{I} \subset \mathsf{Q}_0$ satisfies $\omega(f_i,f_j) = 0$ for any $i,j \in \mathtt{I}$, then $\mathfrak{h}_{\mathtt{I}}:= \bigoplus_{m \in \oplus_{i \in \mathtt{I}}\inte_{\geq 0 }f_i} \mathfrak{h}_m$ is an abelian Lie subalgebra of $\mathfrak{h}$.
	\end{enumerate}
	
	Fix $m = (m_1,\dots,m_r) \in M_{\sigma}^+$ and write $m = m^{\leq i} + m^{>i}$ with $m^{\leq i} = (m_1,\dots,m_i,0,\dots,0)$ and $m^{>i} = (0,\dots,0,m_{i+1},\dots,m_r)$. The above conditions imply that $\omega(m^{\leq i},m^{>i})\leq0$ and hence $\langle m^{> i}, -\sem(m) \rangle \leq 0$. Moreover, if $\langle m^{> i}, -\sem(m) \rangle  = 0$ for all $i$, then $\omega(f_i,f_j) = 0$ for any $i,j \in \mathtt{I}_m$, where $\mathtt{I}_m = \{ 1\leq i\leq r \ | \ m_i \neq 0  \}$.

	We can now make the canonical choice $v^{m}:= -\sem(m)$, leading to a canonically defined Maurer--Cartan element $\varPhi \in \hat{\mathcal{H}}^1$, and so canonically defined $\overline{\mathfrak{M}}_{\ltr}(N_\real,\mathscr{D}_{in})$, $P_{\ltr}$ and $\alpha_{\ltr}$. The proof of Lemma \ref{lem:hall_algebra_well_definedness_solution} is modified as follows. To begin, we prove by induction that $P_{\lrtr} \subset \{ x \in N_\real \mid \langle m^{>i}_{\lrtr},x \rangle \leq 0 \}$ for each $i=1,\dots,r-1$ and all $\lrtr$. The initial case is trivial. For the induction step, spilt $\lrtr$ into $\lrtr_1,\lrtr_2$. Then we have $P_{\lrtr_1} \cap P_{\lrtr_2} \subset \{ x \in N_\real \mid \langle m^{>i}_{\lrtr},x \rangle \leq 0 \}$ and the relation $\langle m^{>i}_{\lrtr},-\sem(m_{\lrtr}) \rangle \leq 0$ gives the desired inclusion. To conclude the proof, consider trees $\lrtr_1$, $\lrtr_2$ joining to give $\lrtr$. If $\langle m^{>i}_{\lrtr}, \sem(m_{\lrtr}) \rangle >0$ for some $i$, then by taking the hyperplane $(m^{>i}_{\lrtr})^{\perp} \cap m_{\lrtr}^{\perp}$ which separates $P_\lrtr$ and $\sem(m_\lrtr)$ in $m^{\perp}_\lrtr$, we can choose $L$ sufficiently large so that $\tau^{m_\lrtr}_{-L}(K) \cap P_\lrtr = \emptyset$ for any compact subset $K\subset N_\real$. Otherwise, the restriction of the Lie bracket to $\mathfrak{h}_{\mathtt{I}_{m_\lrtr}}$ is zero, which guarantees $[g_{\lrtr_1},g_{\lrtr_2}] = 0$. 

The constructions of Sections \ref{sec:non_tropical_diagram} and \ref{sec:hall_algebra_MC_relation} therefore produce a Maurer--Cartan solution $\varPhi$ and a consistent scattering diagram $\mathscr{D}(\varPhi)$. Let us show that the path ordered product along $\lambda$ is again $\Theta_1 \cdots \Theta_r$. It suffices to argue that $\lambda \cap P_{\lrtr} = \emptyset$ for any $\lrtr^{[0]} \neq \emptyset$, that is, $\lambda$ intersects only the initial walls. We have $P_{\lrtr} \subset m_{\lrtr}^\perp \cap \{ x \in N_\real \mid \langle m^{>i}_{\lrtr},x \rangle \leq 0 \}$ for each $i=1,\dots,r-1$ and all $\lrtr$. Let $b=(b_1,\dots,b_r) \in \lambda \cap m_{\lrtr}^{\perp}$. Then $b_1<\cdots<b_r$ and $b_i<0$ for the smallest $i$ such that $0 \neq m_{\lrtr}^{>i} \neq m_{\lrtr}$. Such an index $i$ exists because $m_{\lrtr}$ is not a multiple of a standard basis vector of $M$, as $P_{\lrtr}$ is not an initial wall. Therefore $\langle m_{\lrtr}^{>i}, b \rangle >0$ and hence $b \notin P_{\lrtr}$, as desired.

Motivated by the definition of Hall algebra broken lines \cite{cheung2016tropical}, define the flow $\tau^{m,n}$ using $v^{m,n} := -\sem(m) - n$. This defines $\mathbf{H}_{m,n}$, as in Section \ref{sec:hall_algebra_theta_function}, and so $\theta_n$ using equation \eqref{eqn:theta_function_definition}. We argue that $\varPhi + \theta_n$ is a Maurer--Cartan element of $\hat{\mathcal{H}} \oplus \hat{\mathcal{B}}[-1]$. As in Section \ref{sec:hall_algebra_theta_function}, we can show that $P_{\mrtr} \subset \{ x \in N_{\real} \ | \ \langle m^{>i}_{\mrtr}, x \rangle \leq 0\}$ for all marked ribbon trees $\mrtr$, since $\langle m_{\mrtr}^{>i} , -\sem(m_{\mrtr}) -n \rangle \leq 0$. Consider a labeled ribbon tree $\lrtr$ and a marked ribbon tree $\mrtr$ joining to give $\hat{\mrtr}$. If $\langle m_{\hat{\mrtr}}^{>i}, \sem(m_{\hat{\mrtr}})+ n \rangle =0$ for all $i$, then $g_{\lrtr} \cdot a_{\mrtr} = 0$, otherwise there exists an $i$ such that $\langle m_{\hat{\mrtr}}^{>i}, \sem(m_{\hat{\mrtr}})+ n \rangle >0$ and hence the hyperplane $(m_{\hat{\mrtr}}^{>i})^{\perp}$ would separate $P_{\lrtr} \cap P_{\mrtr}$ and $\sem(m_{\hat{\mrtr}})+ n$. We conclude that $\tau^{m_{\hat{\mrtr}},n}_{-L}(K) \cap (P_{\lrtr} \cap P_{\mrtr}) = \emptyset$ on a compact $K\subset N_\real$ for large enough $L$. This shows $\varPhi + \theta_n$ is a Maurer--Cartan solution.

\begin{theorem}
\phantomsection
\label{thm:non_tropical_WCF}
Let $\mathfrak{h} = \mathfrak{h}_{\mathsf{Q}}^{\Upsilon}$ for acyclic quiver $\mathsf{Q}$ with non-degenerate skew-symmetrized Euler form. The canonically constructed element $\varPhi + \theta_n$ has the following properties:
\begin{enumerate}
\item The Maurer--Cartan solution $\varPhi$ can be written as a sum over labeled trees, 
$$
\varPhi = \sum_{k \geq 1} \sum_{\substack{\ltr \in \ltree{k}\\ \mathfrak{M}_{\ltr}(N_\real,\mathscr{D}_{in}) \neq \emptyset}} \frac{1}{|\Aut(\ltr)|} \alpha_{\ltr} g_\ltr,
$$
with properties as in Theorem \ref{thm:theorem_1_modified}. 

\item The scattering diagram $\mathscr{D}(\varPhi)$ is consistent, and the path ordered product along $\lambda_{\vert [0,1]}$ is $g:= \Theta_1 \cdots \Theta_r$. Hence, $\mathscr{D}(\varPhi)$ is equivalent to the $\mathfrak{h}$-complex $\mathscr{D}(g)$ from \cite[Lemma 3.2]{bridgeland2017scattering}.

\item The section $\theta_n \in \Ker (d_{\varPhi})$ can be written as a sum over marked tropical trees, 
$$
\theta_n = \sum_{k \geq 1} \sum_{\substack{\mtr \in \mtree{k}\\ P_{\mtr} \neq \emptyset}} \frac{1}{|\Aut(J)|} \alpha_{\mtr} a_{\mtr},
$$
and is related to the Hall algebra theta function $\vartheta_{n,Q}$ by the formula $\vartheta_{n,Q} = \theta_n(Q)$.
\end{enumerate}
\end{theorem}

\begin{proof}
It remains to prove the third statement. Again, it suffices to show that $\theta_{n}(Q) = z^n$ for $Q \in \text{int}(\sigma^\vee)$. Consider $\mtr \in \mtree{k}$ and $\varsigma \in \overline{\mathfrak{M}}_{\mtr}(N_{\real},\mathscr{D}_{in},n)$ with $\varsigma(v_{out}) =Q$. Let $\mathfrak{c}_{\mtr} = (e_0,\dots,e_l)$ be the core of $\mtr$ with associated labeled trees $\ltr_1,\dots,\ltr_l$. Let $m_0 = 0$, $m_{j} = m_{j-1} + m_{\ltr_j}$ inductively and $n_j = -\sem(m_j) -n$. We then have $ n_{j} = (\varsigma_{\vert e_j})^{\prime}$. Moreover, $\langle m_j^{>i} , -\sem(m_j) \rangle \leq 0$ and hence $\langle m_j^{>i}, n_j \rangle \leq 0$, since $\langle m_j^{>i} , -n \rangle \leq 0$. Observe that this inequality is strict for some $i$ unless $\langle f_i,n\rangle =0$ for all $ i \in \mathtt{I}_{m_j}$, which in turn forces $a_{e_j} = 0$ and hence $a_{\mtr} = 0$. Since $\sigma^\vee$ is a subset of $\{ x \in N_\real \mid \langle m^{>i}_j,x \rangle \geq 0 \}$, whenever $\langle m_j^{>i}, n_j \rangle < 0$ the hyperplane $m^{>i}_j$ will separate $\sigma^\vee$ and $-n_j$. We conclude that any marked tropical disk with $\varsigma(v_{out}) = Q \in \text{int}(\sigma^\vee)$ must have $a_{\mtr}= 0$. 
\end{proof}

It is natural to ask if Bridgeland's Hall algebra theta functions admit a combinatorial description in terms of Hall algebra broken lines, as defined by Cheung \cite{cheung2016tropical}. This question was recently answered negatively in \cite[\S 5.3]{cheung2019}, where it was shown that Hall algebra theta functions which are defined as a sum over Hall algebra broken lines do not, in general, satisfy the wall-crossing formula \eqref{eq:HallThetaWCF}. Theorem \ref{thm:non_tropical_WCF} can be seen as realizing an alternative combinatorial description of Bridgeland's Hall algebra theta functions for certain quivers.

\bibliographystyle{plain}
\bibliography{geometry}


\end{document}